\documentclass[11pt,reqno]{amsart}
\usepackage{a4wide}
\numberwithin{equation}{section}
\usepackage{mathrsfs}
\usepackage{amsfonts}
\usepackage{amsmath}
\usepackage{stmaryrd}
\usepackage{amssymb}
\usepackage{amsthm}
\usepackage{mathrsfs}
\usepackage{url}
\usepackage{amsfonts}
\usepackage{amscd}
\usepackage{indentfirst}
\usepackage{enumerate}
\usepackage{amsmath,amsfonts,amssymb,amsthm}
\usepackage{amsmath,amssymb,amsthm,amscd}
\usepackage{graphicx,mathrsfs}
\usepackage{appendix}

\usepackage[numbers,sort&compress]{natbib}
\usepackage{color}
 \usepackage[colorlinks, linkcolor=blue, citecolor=blue]{hyperref}

\newcommand\R{\mathbb R}

\newcommand\mbb\mathbb
\newcommand\mbf\mathbf
\newcommand\mcal\mathcal
\newcommand\mfrak\mathfrak
\newcommand\mrm\mathrm
\newcommand\msf\mathsf
\renewcommand\a\alpha
\renewcommand\b\beta
\newcommand\g\gamma
\newcommand\G\Gamma
\renewcommand\d\delta
\newcommand\D\Delta
\newcommand\e\varepsilon
\newcommand\z\zeta
\renewcommand\t\theta
\newcommand\Th\Theta
\newcommand\la\lambda
\newcommand\La\Lambda
\newcommand\s\sigma
\newcommand\si\varsigma
\newcommand\Si\Sigma
\newcommand\ups\upsilon
\newcommand\U\Upsilon
\newcommand\ph\varphi
\renewcommand\o\omega
\renewcommand\O\Omega
\newcommand\wt\widetilde
\newcommand\wh\widehat
\newcommand\ol\overline
\newcommand\ul\underline
\newcommand\mr\mathring
\newcommand\ub\underbrace
\newcommand\pa\partial
\newcommand\n\nabla
\newcommand\fa\forall
\newcommand\ex\exists
\newcommand\es\emptyset
\newcommand\wk\rightharpoonup
\newcommand\inc\hookrightarrow
\newcommand\linf\varliminf
\newcommand\lsup\varlimsup
\newcommand\os\overset
\newcommand\us\underset
\newcommand\sr\stackrel
\newcommand\Ot\Leftarrow
\newcommand\To\Rightarrow
\newcommand\map\mapsto
\newcommand\ot\leftarrow
\newcommand\lot\longleftarrow
\newcommand\lto\longrightarrow
\newcommand\tot\leftrightarrow
\newcommand\ltot\longleftrightarrow
\newcommand\sm\backslash
\renewcommand\Cup\bigcup
\renewcommand\Cap\bigcap
\newcommand\sub\subset
\newcommand\Sub\Subset
\newcommand\sne\subsetneq
\newcommand\bus\supset
\newcommand\Bus\Supset

\newcommand\eq\equiv
\newcommand\ox\otimes
\newcommand\Ox\bigotimes
\newcommand\pl\oplus
\newcommand\Pl\bigoplus
\newcommand\x\times
\renewcommand\c\circ
\newcommand\q\quad
\renewcommand\l\left
\renewcommand\r\right
\newcommand\fr\frac

\usepackage[numbers,sort&compress]{natbib}
\usepackage[dvipsnames]{xcolor}
\usepackage{color}

\definecolor{bondiblue}{rgb}{0.0, 0.58, 0.71}
\def\sideremark#1{\ifvmode\leavevmode\fi\vadjust{\vbox to0pt{\vss
			\hbox to 0pt{\hskip\hsize\hskip1em
				\vbox{\hsize2.1cm\tiny\raggedright\pretolerance10000
					\noindent #1\hfill}\hss}\vbox to15pt{\vfil}\vss}}}%

\newcommand{\Lp}{\color{blue}}

\newtheorem{Thm}{Theorem}[section]
\newtheorem{Lem}[Thm]{Lemma}
\newtheorem{Cor}[Thm]{Corollary}
\newtheorem{Prop}[Thm]{Proposition}

\newtheorem{Rem}[Thm]{Remark}

\begin{document}

\title[Lane-Emden problem]
{Non-degeneracy and local uniqueness of positive solutions to the Lane-Emden problem in dimension two}
\author[M. Grossi, I. Ianni, P. Luo and S. Yan]{Massimo Grossi, Isabella Ianni, Peng Luo, Shusen Yan}

 \address[Massimo Grossi]{Dipartimento di Matematica Guido Castelnuovo, Universit\`a Sapienza, P.le Aldo Moro 5, 00185 Roma, Italy}
\email{massimo.grossi@uniroma1.it}

\address[Isabella Ianni]{Dipartimento di Scienze di Base e Applicate per l'Ingegneria, Universit\`a Sapienza, Via Scarpa 16, 00161 Roma, Italy}
 \email{isabella.ianni@uniroma1.it}

 \address[Peng Luo]{School of Mathematics and Statistics, Central China Normal University, Wuhan 430079,  China }
 \email{pluo@mail.ccnu.edu.cn}

\address[Shusen Yan]{School of Mathematics and Statistics, Central China Normal University, Wuhan 430079, China}
\email{syan@mail.ccnu.edu.cn}

\begin{abstract}
We are concerned with the Lane-Emden problem
\begin{equation*}
\begin{cases}
-\Delta u=u^{p}  &{\text{in}~\Omega},\\[0.5mm]
 u>0  &{\text{in}~\Omega},\\[0.5mm]
u=0 &{\text{on}~\partial \Omega},
\end{cases}
\end{equation*}
where $\Omega\subset \R^2$ is a smooth bounded domain and $p>1$ is sufficiently large.

Improving some known asymptotic estimates on the solutions, we prove the non-degeneracy
and local uniqueness of the multi-spikes positive solutions for general domains. Our methods  mainly
use ODE's theory, various local Pohozaev identities, blow-up analysis and the properties of Green's function.

\end{abstract}

\date{\today}
\maketitle
{\small
\section{Introduction and main results}
\setcounter{equation}{0}
In this paper, we consider the following Lane-Emden problem
\begin{equation}\label{1.1}
\begin{cases}
-\Delta u=u^{p}  &{\text{in}~\Omega},\\[0.5mm]
 u>0  &{\text{in}~\Omega},\\[0.5mm]
u=0 &{\text{on}~\partial \Omega},
\end{cases}
\end{equation}
where $\Omega\subset \R^2$ is a smooth bounded domain and $p>1$ is sufficiently large.

\vskip 0.2cm

The Lane-Emden equation models the mechanical structure of self-gravitating spheres, we
refer to \cite{Hor} for a physical introduction. From the mathematical side, problem \eqref{1.1} is an extremely simple looking
semilinear elliptic equation with a power focusing nonlinearity and  it
provides a great number of interesting phenomena. Indeed it is well known that the number of solutions
to problem \eqref{1.1} strongly depends on the exponent $p$ of the nonlinearity and on both the geometry and the topology of the domain $\Omega$.

\vskip 0.2cm

We recall that in any smooth bounded domain $\Omega$, problem \eqref{1.1} admits at least one solution for any $p>1$, which can be obtained by standard variational methods, for example minimizing the associated energy functional on the Nehari manifold.
The solution is unique and nondegenerate in any domain $\Omega$, when $p$ is close enough to $1$ (see \cite{DamascelliGrossiPacella, DancerMa2003,  L94}). Moreover when $\Omega$ is a ball (\cite{GNN}) or, more in general, a symmetric and convex domain with respect to two orthogonal direction (\cite{Dancer88, DamascelliGrossiPacella}), the uniqueness and nondegeneracy hold true for any $p>1$.
Recently in \cite{DGIP2019} uniqueness and nondegeneracy for the solutions to \eqref{1.1} has been proved in any convex  domain $\Omega$, when $p$ is sufficiently large.		

\vskip 0.2cm

On the other hand multiplicity results for problem \eqref{1.1} are known in many cases, for instance in some dumb-bell shaped domains at suitable values of $p$ (\cite{Dancer88}) or in annular domains (see for example \cite{CaWa, GGPS, YYLi, BCGP}).
We specially mention  \cite{EMP2006} where, in not simply connected  domains, the existence of solutions
of \eqref{1.1} which  concentrate at $k $ points as $p\rightarrow +\infty$, for any $k\in\mathbb N$, is proven. The energy of these \emph{multi-spike} solutions satisfies
\begin{equation}\label{11-11-01}
\sup_{p}p\|\nabla u_p\|^2_2<\infty,
\end{equation}
more precisely $
p\|\nabla u_p\|^2_2\rightarrow 8\pi e\cdot k,
$
 and moreover the location of the $k$ concentration points  depends on the
critical point of the related Kirchoff-Routh function. \\
Observe that earlier results in
\cite{RW1994,RW1996} (later improved in \cite{AG2003,EG2004}) show that the least energy solutions satisfy  $
p\|\nabla u_p\|^2_2\rightarrow 8\pi e,
$ hence \eqref{11-11-01},
and  concentrate at \emph{one} point as $p\rightarrow +\infty$ (\emph{$1$-spike solutions}). 
Furthermore both the families of solutions found in \cite{EMP2006} and the least energy solutions (see \cite{AG2003}) satisfy a bound on the $L^{\infty}$-norm, since it can be proved that $\|u_{p}\|_{\infty}\rightarrow \sqrt e, $ as $p\rightarrow +\infty$.

\vskip 0.2cm

Nevertheless, despite of all these results, a complete understanding of  the properties and number of the positive solutions to the Lane-Emden problem is still far from being achieved.

\vskip 0.2cm

Very recently  a priori bounds, uniform in $p$, has been obtained in \cite{KS2018} \emph{for all} the solutions to \eqref{1.1}. Furthermore in \cite{DIP2017-1} (later sharpened independently in  \cite{DGIP2018,T2019}) a complete characterization of the asymptotic behavior as $p\rightarrow +\infty$ of all the solutions to \eqref{1.1} was derived, using blow-up techniques, under the energy bound assumption \eqref{11-11-01}. This result shows that the solutions necessarily are \emph{multi-spike} solutions, namely they behave like the ones found in \cite{EMP2006} (see Theorem A below for more details). One can then exploit the information on the asymptotic behavior to further investigate other properties of the solutions for $p$ large. This has been pursued in \cite{DGIP2019}, where the Morse index of \emph{$1$-spike} solutions and their non-degeneracy has been studied, leading to the aforementioned uniqueness result in convex domains for $p$ large. We stress that a crucial role was played by the a priori bounds in \cite{KS2018}, which in convex domains are equivalent to the energy bounds \eqref{11-11-01}.

\vskip 0.2cm

In this paper we deal with the non-degeneracy and uniqueness of \emph{multi-spike} solutions of \eqref{1.1} in general domains. 
This two topics are strictly related, furthermore non-degeneracy plays a crucial role in bifurcation theory.

\vskip 0.2cm

In order to state our results we need to introduce some notations and recall the asymptotic characterization proved in \cite{DIP2017-1}. We denote by  $G(x,\cdot)$ the Green's function
 of $-\Delta$ in $\Omega$, i.e. the solution to
\begin{equation}\label{greensyst}
\begin{cases}
-\Delta G(x,\cdot)= \delta_x  &{\text{in}~\Omega}, \\[1mm]
G(x,\cdot)=0  &{\text{on}~\partial\Omega},
\end{cases}
\end{equation}
where $\delta_x$ is the Dirac function. We have the well known decomposition formula of $G(x,y)$,
\begin{equation}\label{GreenS-H}
G(x,y)=S(x,y)-H(x,y)  ~\mbox{for}~(x,y)\in \Omega\times \Omega,
\end{equation}
where  $S(x,y):=-\frac{1}{2\pi}\log |x-y|$ and $H(x,y)$ is the regular part of $G(x,y)$. Next, for any $x\in \Omega$, we recall that the Robin function is defined as
\begin{equation}\label{Robinf}
R(x):=H(x,x).
\end{equation}

For $a=(a_1,\cdots,a_k)$, with $a_j\in \Omega$, $j=1,\cdots,k$ we define the Kirchoff-Routh function  $\Psi_{k}: \Omega^{k} \rightarrow \R$ as
\begin{equation}\label{stts}
\Psi_{k}(a):= \sum^k_{j=1} \Psi_{k,j}(a),\quad~\mbox{ with }~\Psi_{k,j}(a):=  R\big(a_j\big)- \sum^{k}_{m=1,m\neq j} G\big(a_j,a_m\big).
\end{equation}
Finally let us introduce the function
\begin{equation}
\label{defU}
U(x)=-2\log \big(1+\frac{|x|^2}{8}\big),\end{equation}
which is the \emph{unique} positive solution of the Liouville equation
\begin{equation*}
\begin{cases}
-\Delta U=e^U \,\,~\mbox{in}~\R^2,\\[1mm]
\displaystyle\int_{\R^2}e^Udx=8\pi.
\end{cases}
\end{equation*}
We recall the known asymptotic characterization of the positive solutions of problem \eqref{1.1},  under assumption \eqref{11-11-01}.

\vskip 0.2cm

\noindent\textbf{Theorem A}~(\cite{DGIP2018,DIP2017-1})\textbf{.} \emph{Let $u_p$ be a family of solutions to  \eqref{1.1} satisfying \eqref{11-11-01}.
Then there exist a finite number of $k$ of distinct points $x_{\infty,j}\in \Omega$, $j=1,\cdots,k$ and a subsequence of $p$ (still denoted by $p$)   such that setting
$\mathcal{S}:=\big\{x_{\infty,1},\cdots, x_{\infty,k}\big\}$,
one has
\begin{equation}\label{11-14-03N}
\lim_{p\rightarrow +\infty} p u_{p}=8\pi \sqrt{e}\sum^k_{j=1}G(x,x_{\infty,j})\,\,~\mbox{in} ~ C^2_{loc}(\Omega\backslash \mathcal{S}),
\end{equation}
the energy satisfies
\begin{equation*}
\lim_{p\rightarrow +\infty} p \int_{\Omega}|\nabla u_{p}(x)|^2dx=8\pi e\cdot k,
\end{equation*}
and the concentrated points $x_{\infty,j}$, $j=1,\cdots,k$ fulfill the system
\begin{equation*}
\nabla_x \Psi_k\big(x_{\infty,1},\cdots,x_{\infty,k}\big)=0.
\end{equation*}
Moreover, for some small fixed $r>0$, let $x_{p,j}\in \overline{B_{2r}(x_{\infty,j})}\subset\Omega$ be the sequence defined as
\begin{equation}\label{def:xpj}
u_{p}(x_{p,j})=\max_{\overline{B_{2r}(x_{\infty,j})}}u_{p}(x),
\end{equation}
then for any $j=1,\cdots,k$, it holds
\begin{equation*}
\lim_{p\rightarrow +\infty}x_{p,j}=x_{\infty,j},
\end{equation*}
\begin{equation}\label{ConvMax}
\lim_{p\rightarrow +\infty}u_{p}(x_{p,j})=\sqrt{e},
\end{equation}
\begin{equation*}
\lim_{p\rightarrow +\infty}\varepsilon_{p,j}=0,
\end{equation*}
where    $\varepsilon_{p,j}:=\Big(p\big(u_{p}(x_{p,j})\big)^{p-1}\Big)^{-1/2}$. And setting
\begin{equation}\label{defwpj}
w_{p,j}(y):=\frac{p}{u_{p}(x_{p,j})}\Big(u_{p}(x_{p,j}+\varepsilon_{p,j}y)-
u_{p}(x_{p,j})\Big),~y\in \Omega_{p,j}:=\frac{\Omega-x_{p,j}}{\varepsilon_{p,j}},
\end{equation}
one has
\begin{equation}\label{5-8-2}
\lim_{p\rightarrow +\infty} w_{p,j}=U\,\,~\mbox{in}~C^2_{loc}(\R^2).
\end{equation} }

\vskip 0.1cm

Our first result is about the non-degeneracy for  the \emph{multi-spike}  solutions of \eqref{1.1}, when   $p$ is large:
\begin{Thm}\label{th1.1}
Let   $u_p$ be a solution of \eqref{1.1} satisfying \eqref{11-11-01}. Let
$k$ be the number in Theorem A  and $\xi_p\in H^1_0(\Omega)$ be a solution of $\mathcal{L}_p\big(\xi_p\big)=0$, where
$$\mathcal{L}_p\big(\xi\big):= -\Delta \xi -pu_p^{p-1}\xi$$
is the linearized operator of the  Lane-Emden problem \eqref{1.1} at the solution $u_{p}$.
Suppose that $x_{\infty}:=(x_{\infty,1},\cdots,x_{\infty,k})$
is a nondegenerate critical point of the Kirchoff-Routh function $\Psi_{k}$  defined in \eqref{stts}, then there exists $p^{\ast}>1$ such that
	\[\xi_p \equiv 0, \qquad\mbox{ for } p\geq p^{\ast}.\]
\end{Thm}

\vskip 0.2cm

We also prove a local uniqueness result for the \emph{$1$-spike} solutions to \eqref{1.1} (case $k=1$ in Theorem A). Notice that $\Psi_{1}=R$, namely the Kirchoff-Routh function reduces to the Robin function in this case.
\begin{Thm}\label{th1-1}
Let $u_p^{(1)}$ and $u_p^{(2)}$ be two solutions to \eqref{1.1} with
\begin{equation}\label{energy1}
\lim_{p\rightarrow +\infty} p \int_{\Omega}|\nabla u^{(l)}_{p}(x)|^2dx=8\pi e~\,~\mbox{for}\,~l=1,2,
\end{equation}
which concentrate at the same  critical point $x_{\infty,1}\in\Omega$ of the Robin function $R$.
If $x_{\infty,1}$ is non-degenerate, then  there exists $p^{\ast}>1$ such that \[u_p^{(1)}\equiv u_p^{(2)},\qquad \mbox{ for }p\geq p^{\ast}.\]
\end{Thm}

\vskip 0.2cm

We stress that both Theorem \ref{th1.1} and Theorem \ref{th1-1} require the non-degeneracy of  the critical points of the Kirchhoff-Routh function $\Psi_{k}$. The existence of critical points to $\Psi_{k}$ and  their  non-degeneracy is a widely discussed topic,  we refer for instance  to \cite{Musso1,Bartsch1,Bartsch,Micheletti} and  the references therein.

\vskip 0.2cm

We also remark that our results hold for any smooth bounded domain $\Omega$. In the particular case when $\Omega$ is a convex domain, it is easy to show that \emph{any} solution to \eqref{1.1} satisfies the assumptions of Theorem \ref{th1-1}. Indeed in this case Theorem A holds with $k=1$, since
from the a priori bounds in \cite{KS2018} we know that \eqref{11-11-01} is satisfied  and,  furthermore, from \cite{Grossi} we also know that problem \eqref{1.1} possesses no solutions
 with $k$ spikes for $k\geq 2$. As a consequence any solution to \eqref{1.1}, when $\Omega$ is convex, is necessarily a \emph{$1$-spike} solution which satisfies \eqref{energy1} and concentrates at a critical point of the Robin function. Finally we also know that in convex domains the  Robin function $R$ is strictly convex and  its unique critical point is non-degenerate (\cite{CF1985}).
 \\
Hence from Theorem \ref{th1-1} we re-obtain  the uniqueness result already proved in \cite{DGIP2019}:
\begin{Cor} Let $\Omega\subset \mathbb R^{2}$ be a smooth bounded and convex domain. Then there exists $p^{\ast}>1$ such that problem \eqref{1.1} admits a unique solution for any $p\geq p^{\ast}$.
\end{Cor}

Notice that the proof in \cite{DGIP2019} was based on Morse index computations, since it exploited a uniqueness result in convex domains, for solutions to \eqref{1.1}  having Morse index equal to one, proved by Lin \cite{L94}. 

If $\Omega$ is convex the unique solution has \emph{one spike} and its Morse index is one (see \cite{DGIP2019}).  In more general domains instead a \emph{$1$-spike} solution concentrates at a  critical point of the Robin function $R$, but we know that  it has Morse index larger than one if this critical point is not a local minimum for $R$ (see \cite[Theorem 1.1]{DGIP2019}). Theorem~\ref{th1-1} tells us that each non-degenerate critical point of the Robin function can
  generate exactly one single spike solution for \eqref{1.1} in any domain, so if all the critical points of $R$ are non-degenerate, then the number of \emph{$1$-spike}
  solutions equals the number of the critical points of $R$.

\vskip 0.2cm

We point out that  we are not able to prove  the local uniqueness of \emph{multi-spike} solutions ($k\geq 2$) for problem \eqref{1.1} (see Remark \ref{rem5} in Section 5). Observe also that the non-degeneracy result in Theorem \ref{th1.1} was already known for $1$-spike solutions (see \cite[Theorem 1.1]{DGIP2019}).

\vskip 0.2cm

The proofs of Theorem  \ref{th1.1} and Theorem  \ref{th1-1} are obtained  by combining  various  local Pohozaev identities, blow-up analysis and the properties of Green's function,  inspired by \cite{Deng,Grossi2,GMPY20}.

\vskip 0.2cm

They are  based on the following  crucial sharper asymptotics results for the multi-spike solutions of \eqref{1.1} than the one stated in
 Theorem~A, and which are of independent interest.

\begin{Thm}\label{th1} Let $u_p$ be a family of solutions to \eqref{1.1} and \eqref{11-11-01},
$x_{\infty,j}$, $x_{p,j}$,  $\e_{p,j}$ and $ w_{p,j}$ with $j=1,\cdots,k$ be defined in Theorem A above, then
 for any   fixed small  constant $\delta>0$, it holds
\begin{equation}\label{5-7-52}
u_{p}(x_{p,j})=\sqrt{e}\left(1-\frac{\log p}{p-1}+
\frac{ 1}{ p}\Big(4\pi \Psi_{k,j}(x_{\infty})+3\log 2+2 \Big)
\right)+O\Big(\frac{1}{p^{2-\delta}}\Big)\,\,~\mbox{for}~j=1,\cdots,k,
\end{equation}
where $\Psi_{k,j}$ is the function in \eqref{stts} and $x_{\infty}:=\big(x_{\infty,1},\cdots,x_{\infty,k}\big)$
.
Consequently,
\begin{equation}\label{nn3-29-03}
\e_{p,j}= e^{-\frac{p}4}\Bigl(
 e^{-\big( 2\pi \Psi_{k,j}(x_{\infty})+\frac{3\log 2}{2}+\frac{3}{4}\big) }+O\big(\frac{1}{p^{1-\delta}}\big)\Bigr),
\end{equation}
and
\begin{equation}\label{3-29-03}
\frac{\e_{p,j}}{\e_{p,s}}=e^{2\pi \big(\Psi_{k,s}(x_{\infty})-\Psi_{k,j}(x_{\infty})\big)}
 +O\big(\frac{1}{p^{1-\delta}}\big)\,\, ~\mbox{for}~1\leq j,s\leq k.
\end{equation}
Moreover,  it holds
\begin{equation}\label{lst}
w_{p,j}=U+\frac{w_0}{p} +O\left(\frac{1}{p^2}\right)~\mbox{in}~C^2_{loc}(\R^2),
\end{equation}
where $w_0$ solves the non-homogeneous linear equation
\begin{equation}\label{lst1}
-\Delta w_{0}-e^{U}w_{0}=-\frac{U^2}{2}e^{U}~\mbox{in}~\R^2.
\end{equation}
\end{Thm}

Observe that, for single spike solutions of \eqref{1.1}, the non-degeneracy was proved in \cite{DGIP2019} by estimating all the
eigenvalues of the operator $\mathcal{L}_p$.  For multi-spike solutions of \eqref{1.1} our proof here is based on different arguments, but we believe that   one can characterize the asymptotic behavior of the eigenvalues/eigenfunctions of  $\mathcal{L}_p$ as $p\rightarrow +\infty$ also in this case, using the  estimates \eqref{nn3-29-03} and \eqref{3-29-03}. This would lead in particular  to the computation of the Morse index of multi-spike solutions. 
\\
In higher dimension results in this direction can be found in  \cite{BLR95} and \cite{CKL16}, where  the Morse index of the multi-bubbling  solutions of the following problem were calculated
\begin{equation}\label{00-30-10}
\begin{cases}
-\Delta u=u^{\frac{N+2}{N-2}-\e}~&\mbox{in}~\Omega,\\[1mm]
u>0~&\mbox{in}~\Omega,\\[1mm]
u=0~&\mbox{on}~\partial \Omega,
\end{cases}
\end{equation}
 where $\Omega$ is a bounded domain in $\R^N$($N\geq 3$) and $\e>0$ is a small parameter. Note that, unlike our case,  the concentration rates  of the  bubbles of a multi-bubbling solution for \eqref{00-30-10} are all of the same order.  We plan to further investigate these topics in a subsequent paper.
 %
%
%

\vskip 0.2cm

The paper is organized as follows. In Section \ref{s} we collect some preliminary known results, give crucial computations involving the Green's function and introduce suitable local Pohozaev identities. In Section \ref{s1} we study the  asymptotic behavior
 of  the multi-spike solutions of \eqref{1.1} and prove Theorem~\ref{th1}.
 Section \ref{s2} is devoted to the discussion on the non-degeneracy of the solutions of \eqref{1.1} and ends with the proof of Theorem \ref{th1.1}.
 In Section \ref{s3} we consider single spike solutions of \eqref{1.1} and  prove Theorem  \ref{th1-1}. We have postponed to Appendix \ref{s6} some detailed proofs about the tedious computations involving the Green's function.

\vskip 0.2cm

Throughout this paper, we use the same $C$ to denote various generic positive constants independent with $p$
and $\|\cdot\|$  to denote the basic norm in the Sobolev space $H^1_0(\Omega)$ and $\langle\cdot,\cdot\rangle$ to mean the corresponding inner product. We will use $\partial$ or $\nabla$ to denote the partial derivative for any function $h(y,x)$ with respect to $y$, while we will use $D$ to denote the partial derivative for any function $h(y,x)$ with respect to $x$.

\section{Some known facts and key computations} \label{s}
 We collect  some known results which will be useful throughout  the paper.

 \subsection{Linearization of the Liouville equation}$\,$  \vskip 0.2cm

Next lemma is  a well known characterization of the kernel of the linearized operator of the
Liouville equation  at the solution $U$ (see for instance  \cite{EG2004}).
\begin{Lem}\label{llm}
Let $U$ be the function defined in \eqref{defU} and $v\in C^2(\R^2)$ be a solution of the following problem
\begin{equation}\label{3-29-01}
\begin{cases}
-\Delta v=e^Uv\,\,~\mbox{in}~\R^2,\\[1mm]
\displaystyle\int_{\R^2}|\nabla v|^2dx<\infty.
\end{cases}
\end{equation}
Then it holds
\begin{equation*}
v(x)\in \mbox{span}\left\{\frac{\partial U(x)}{\partial x_1},
\frac{\partial U(x)}{\partial x_2},\frac{\partial U(\frac{x}{\lambda})}{\partial \lambda}\Big|_{\lambda=1}\right\}.
\end{equation*}
\end{Lem}
\begin{Lem}\label{lll3}
Let $f\in C^1\big([0,+\infty)\big)$ such that $\displaystyle\int^{+\infty}_{0}t|\log t|\cdot |f(t)|dt<+\infty$. Then there exists a $C^2$ radial solution $w(r)$ of the equation
\begin{equation}\label{6-4-1}
\Delta w+8e^{U(\sqrt{8}y)}w=f(|y|)~~\,~\mbox{in}~\R^2,
\end{equation}
such that as $r\to +\infty$,
\begin{equation}\label{5-26-1}
 \partial_rw(r)=\left(\int^{+\infty}_0t\frac{t^2-1}{t^2+1}f(t)dt\right)\frac{1}{r}+
 O\left(\frac{1}{r}\int^{+\infty}_rs|f(s)|ds+\frac{|\log r|}{r^3}\right).
\end{equation}
\end{Lem}
\begin{proof}
See Lemma 2.1 in \cite{EMP2006}.
\end{proof}

\vskip 0.4cm

\subsection{Quadratic forms} $\,$  \vskip 0.2cm

For each point $x_{p,j},$ $j=1,\ldots, k$ in  \eqref{def:xpj}, let
us define the following two quadric forms \begin{equation}\label{07-08-20}
\begin{split}
P_{j}(u,v):=&- 2\theta\int_{\partial B_\theta(x_{p,j})}
\big\langle \nabla u ,\nu\big\rangle
\big\langle \nabla v,\nu\big\rangle \,d\sigma
+  \theta  \int_{\partial B_\theta(x_{p,j})}
\big\langle \nabla u , \nabla v \big\rangle\,d\sigma,
\end{split}
\end{equation}
and
\begin{equation}\label{abd}
Q_{j}(u,v):=- \int_{\partial B_\theta(x_{p,j})}\frac{\partial v}{\partial \nu}\frac{\partial u}{\partial x_i}\,d\sigma-
 \int_{\partial B_\theta(x_{p,j})}\frac{\partial u}{\partial \nu}\frac{\partial v}{\partial x_i}\,d\sigma
+ \int_{\partial B_\theta(x_{p,j})}\big\langle \nabla u,\nabla v \big\rangle \nu_i\,d\sigma,
\end{equation}
where $u,v\in C^{2}(\overline{\Omega})$ and  $\theta>0$ is such that $B_{2\theta}(x_{p,j})\subset\Omega$.
\begin{Lem}\label{llas}
If $u$ and $v$ are harmonic in $ B_d(x_{p,j})\backslash \{x_{p,j}\}$, then $P_{j}(u,v)$ and $Q_{j}(u,v)$ are  independent of $\theta\in (0,d]$.
\end{Lem}
\begin{proof}
Let $\Omega'\subset \Omega$ be a smooth bounded domain, integrating by parts, we have
\begin{equation}\label{07-07-1}
\begin{split}
-&\int_{\Omega'}\big\langle \nabla v,x-x_{p,j}\big\rangle \Delta u -\int_{\Omega'}\big\langle \nabla u,x-x_{p,j}\big\rangle \Delta v\\=&-\int_{\partial \Omega'} \frac{\partial u}{\partial \nu}\big\langle \nabla v,x-x_{p,j}\big\rangle-
\int_{\partial \Omega'} \frac{\partial v}{\partial \nu}\big\langle \nabla u,x-x_{p,j}\big\rangle
+\int_{\partial \Omega'} \big\langle \nabla u,  \nabla  v \big \rangle  \Big\langle  \nu, x-x_{p,j} \Big\rangle.
\end{split}
\end{equation}
Let $\theta_{1}<\theta_{2}<d$ and $\Omega'=B_{\theta_{2}}(x_{p,j})\backslash B_{\theta_{1}}(x_{p,j})$ in  \eqref{07-07-1}. Since $u$ and $v$ are harmonic in $ B_d(x_{p,j})\backslash \{x_{p,j}\}$, we find
\begin{equation*}
-\int_{\partial \Omega'} \frac{\partial u}{\partial \nu}\big\langle \nabla v,x-x_{p,j}\big\rangle-
\int_{\partial \Omega'} \frac{\partial v}{\partial \nu}\big\langle \nabla u,x-x_{p,j}\big\rangle
+\int_{\partial \Omega'} \big\langle \nabla u,  \nabla  v \big \rangle  \Big\langle  \nu, x-x_{p,j} \Big\rangle=0,
\end{equation*}
namely
\[
-2\theta_{2}\int_{\partial B_{\theta_{2}}(x_{p,j})} \frac{\partial u}{\partial \nu}\frac{\partial  v}{\partial\nu}
+\theta_{2}\int_{\partial B_{\theta_{2}}(x_{p,j})} \big\langle \nabla u,  \nabla  v\big\rangle =-2\theta_{1}\int_{\partial B_{\theta_{1}}(x_{p,j})} \frac{\partial u}{\partial \nu}\frac{\partial  v}{\partial\nu}
+\theta_{1}\int_{\partial B_{\theta_{1}}(x_{p,j})} \big\langle \nabla u,  \nabla  v\big\rangle .
\]
This shows that
$P_{j}(u,v)$  is  independent of $\theta\in (0,d]$.

\vskip 0.1cm

On the other hand, integrating by parts, we get
\begin{equation*}
\begin{split}
-\int_{\Omega'} \Big(\Delta u \frac{\partial v}{\partial x_i}+\Delta v \frac{\partial u}{\partial x_i}\Big)=&
-\int_{\partial \Omega'}
\Big(\frac{\partial v}{\partial \nu} \frac{\partial u}{\partial x_i}
+\frac{\partial u}{\partial \nu} \frac{\partial v}{\partial x_i}\Big)
+\int_{\Omega'} \Big(\big\langle \nabla v, \nabla \frac{\partial u}{\partial x_i}\big\rangle
+\big\langle \nabla u, \nabla \frac{\partial v}{\partial x_i}\big\rangle\Big)\\=&
-\int_{\partial \Omega'}
\Big(\frac{\partial v}{\partial \nu} \frac{\partial u}{\partial x_i}
+\frac{\partial u}{\partial \nu} \frac{\partial v}{\partial x_i}\Big)
+\int_{\partial\Omega'}\big\langle \nabla u,\nabla v \big\rangle \nu_i.
\end{split}
\end{equation*}
Since $u$ and $v$ are harmonic in $ B_d(x_{p,j})\backslash \{x_{p,j}\}$, then arguing like before we get that $Q_{j}(u,v)$ is  independent of $\theta\in (0,d]$.
\end{proof}

Next we derive some key computations about the quadratic forms $P_{j}$ and $Q_{j}$ and Green's function.
\begin{Prop}\label{lem2-1}
It holds
\begin{equation}\label{1-1}
P_{j}\Big(G(x_{p,s},x), G(x_{p,m},x)\Big)=
\begin{cases}
-\frac{1}{2\pi} ~&\mbox{for}~s=m=j,\\[1mm]
0~&\mbox{for}~s\neq j ~{or}~m \neq j,
\end{cases}
\end{equation}
and
 \begin{equation}\label{abb1-1}
P_{j}\Big(G(x_{p,s},x),\partial_hG(x_{p,m},x)\Big)=
\begin{cases}
\frac{1}{2}\frac{\partial R(x_{p,j})}{\partial h} ~&\mbox{for}~s=m=j,\\[1mm]
 -D_{x_h}G \big(x_{p,s},x_{p,j}\big) ~&\mbox{for} ~m=j,~s\neq j,\\[1mm]
0 ~&\mbox{for}~m\neq j.
\end{cases}
\end{equation}
Moreover
\begin{equation}\label{aluo1}
Q_{j}\Big(G(x_{p,m},x),G(x_{p,s},x)\Big)=
\begin{cases}
-\frac{\partial R(x_{p,j})}{\partial{x_i}} ~&\mbox{for}~s=m=j,\\[1mm]
D_{x_i}G(x_{p,m},x_{p,j})
 ~&\mbox{for}~m\neq j,~s=j,\\[1mm]
D_{x_i}G(x_{p,s},x_{p,j})
~&\mbox{for}~m=j,~s\neq j,\\[1mm]
0 ~&\mbox{for}~s,m\neq j,
\end{cases}
\end{equation}
and
\begin{equation}\label{aluo41}
Q_{j}\Big(G(x_{p,m},x),\partial_h G(x_{p,s},x)\Big)=
\begin{cases}
- \frac{\partial^2 R(x_{p,j})}{\partial{x_ix_h}} ~&\mbox{for}~s=m=j,\\[1mm]
D_{x_i}\partial_h G(x_{p,s},x_{p,j})
 ~&\mbox{for}~m= j,~s\neq j,\\[1mm]
D^2_{x_ix_h} G(x_{p,m},x_{p,j})
~&\mbox{for}~m\neq j,~s=j,\\[1mm]
0 ~&\mbox{for}~s,m\neq j.
\end{cases}
\end{equation}
where $G(x,y)$ and $R(x,y)$ are the Green and Robin function, respectively (see \eqref{greensyst}, \eqref{GreenS-H}, \eqref{Robinf}), $\partial_iG(y,x):=\frac{\partial G(y,x)}{\partial y_i}$ and
$D_{x_i}G(y,x):=\frac{\partial G(y,x)}{\partial x_i}$.
\end{Prop}
\begin{proof}
Here we prove only \eqref{1-1}. Since the computations of \eqref{abb1-1}--\eqref{aluo41} are similar, we put the details in Appendix \ref{s6}.

\vskip 0.1cm

By the bilinearity of $P_{j}(u,v)$ and \eqref{GreenS-H}, we have
\begin{equation}\label{gil21}
\begin{split}
P_{j}\Big(G(x_{p,j},x),G(x_{p,j},x)\Big)=&
 P_{j}\Big(S(x_{p,j},x),S(x_{p,j},x)\Big)
 -2P_{j}\Big(S(x_{p,j},x),H(x_{p,j},x)\Big)\\&
 +P_{j}\Big(H(x_{p,j},x),H(x_{p,j},x)\Big).
 \end{split}
\end{equation}
After direct calculations, we know
\begin{equation}\label{laa}
\begin{split}
D_{x_i} S\big(x_{p,j},x\big) =-\frac{x_i-x_{p,j,i}}{2\pi |x_{p,j}-x|^{2}} \,~\mbox{and}~\nu_i=
\frac{x_i-x_{p,j,i}}{|x_{p,j}-x|}.
 \end{split}
\end{equation}
Putting \eqref{laa} in the term $P_{j}\Big(S(x_{p,j},x),S(x_{p,j},x)\Big)$, we get
\begin{equation}\label{gil22}
P_{j}\Big(S(x_{p,j},x),S(x_{p,j},x)\Big)=-\frac{1}{2\pi}.
\end{equation}
Now we calculate $P_{j}\Big(S(x_{p,j},x),H(x_{p,j},x)\Big)$.  Since $D_{\nu} H(x_{p,j},x)$   is bounded in $B_d(x_{p,j})$,  we know
\begin{equation}\label{gil27}
P_{j}\Big(S(x_{p,j},x),H(x_{p,j},x)\Big)=
O\Big(\theta\int_{\partial B_\theta(x_{p,j})} \big| D S(x_{p,j},x)\big|\Big)=O\big(\theta \big),
\end{equation}
and
\begin{equation}\label{gil28}
P_{j}\Big(H(x_{p,j},x),H(x_{p,j},x)\Big)=
O\Big(\theta\int_{\partial B_\theta(x_{p,j})} \big| D H(x_{p,j},x)\big|\Big)=O\big(\theta^2 \big).
\end{equation}
Letting $\theta\rightarrow 0$, from \eqref{gil21}, \eqref{gil22}, \eqref{gil27} and \eqref{gil28}, we get
\begin{equation*}
\begin{split}
P_{j}\Big(G(x_{p,j},x),G(x_{p,j},x)\Big)=-\frac{1}{2\pi}.
 \end{split}
\end{equation*}
Let $m\neq j$, since $G(x_{p,m},x)$ and $D_{\nu}G(x_{p,m},x)$ are bounded in $B_d(x_{p,j})$, then we find
\begin{equation}\label{gil29}
\begin{split}
P_{j}\Big(G(x_{p,s},x),G(x_{p,m},x)\Big)=
O\Big(\theta\int_{\partial B_\theta(x_{p,j})} \big| D G(x_{p,s},x)\big|\Big)=O\big(\theta \big).
 \end{split}
\end{equation}
Letting $\theta\rightarrow 0$ in \eqref{gil29} and using the symmetry of $P_{j}\big(u,v\big)$, we deduce that   $$P_{j}\Big(G(x_{p,s},x),G(x_{p,m},x)\Big)=0\,\,~\mbox{for}~s\neq j~\mbox{or}~m\neq j.$$
\end{proof}

\subsection{Pohozaev-type identities} $\,$  \vskip 0.2cm

We prove some  local Pohozaev identities.
\begin{Lem}
Let $u\in C^2(\Omega)$ be a solution of \eqref{1.1},  $x_{p,j}$, $j=1,\ldots, k$ be the points defined in \eqref{def:xpj} and let $\theta>0$ be such that $B_{2\theta}(x_{p,j})\subset\Omega$. Then
\begin{equation}\label{aclp-1}
Q_{j}(u,u)= \frac{2}{p+1}\int_{\partial B_{\theta}(x_{p,j})}  u^{p+1} \nu_id\sigma
\end{equation}
and
\begin{equation}\label{aclp-10}
P_{j}(u,u)=\frac{2\theta}{p+1} \int_{\partial B_{\theta}(x_{p,j})}  u^{p+1}  d\sigma-\frac{4}{p+1}\int_{B_{\theta}(x_{p,j})}  u^{p+1} dx,
\end{equation}
where $P_{j}$ and $Q_{j}$ are the quadratic forms defined in \eqref{07-08-20} and \eqref{abd} and $\nu=\big(\nu_{1},\nu_2\big)$ is the outward unit normal of $\partial B_{\theta}(x_{p,j})$.
\end{Lem}
\begin{proof}
The identity \eqref{aclp-1} follows by  multiplying $\frac{\partial u }{\partial x_i}$ on both sides of \eqref{1.1}, integrating on $B_{\theta}(x_{p,j})$ and applying the divergence theorem and Green's identities:
\[
-\int_{\partial B_{\theta}(x_{p,j})}\frac{\partial u}{\partial \nu}\frac{\partial u}{\partial x_i}
+\frac{1}{2}\int_{\partial B_{\theta}(x_{p,j})}|\nabla u|^2\nu_i= \frac{1}{p+1}\int_{\partial B_{\theta}(x_{p,j})}  u^{p+1} \nu_i.\]
 Also
\eqref{aclp-10} follows by multiplying $\big\langle x-x_{p,j}, \nabla u \big\rangle$ on both sides of \eqref{1.1} and integrating on $B_{\theta}(x_{p,j})$:
\[\begin{split}
\frac{1}{2}\int_{\partial B_{\theta}(x_{p,j})} &\big\langle x-x_{p,j},\nu\big\rangle   |\nabla u|^2
-\int_{\partial B_{\theta}(x_{p,j})}\frac{\partial u}{\partial\nu}
\big\langle x-x_{p,j}, \nabla u \big\rangle \\
&=\frac{1}{p+1}
\int_{\partial B_{\theta}(x_{p,j})}  u^{p+1}  \big\langle x-x_{p,j},\nu\big\rangle
- \frac{2}{p+1} \int_{B_{\theta}(x_{p,j})}  u^{p+1} .
\end{split}
\]
\end{proof}

\begin{Prop} \label{prop:PohozaevLin}
Let $u\in C^2(\Omega)$ be a solution of \eqref{1.1},  $\xi\in C^{2}(\Omega)$ be a solution of $-\Delta \xi=pu^{p-1}\xi$ in $\Omega$ and  $x_{p,j}$, $j=1,\ldots, k$ be the points defined in \eqref{def:xpj} and let $\theta>0$ be such that $B_{2\theta}(x_{p,j})\subset\Omega$.
Then
\begin{equation}\label{dafd}
Q_{j}\big(\xi,u\big)=\int_{\partial B_{\theta}(x_{p,j})}u^p \xi\nu_i\,d\sigma,
\end{equation}
and
\begin{equation}\label{07-08-22}
P_{j}\big(\xi,u\big)
=\theta \int_{\partial B_{\theta}(x_{p,j})} u^p\xi \,d\sigma
-2\int_{B_{\theta}(x_{p,j})} u^p\xi \,dx,
\end{equation}
where   $Q_{j}$ and $P_{j}$ are the quadratic forms in \eqref{07-08-20} and \eqref{abd}, $\nu =\big(\nu_{1},\nu_2\big)$ is the outward unit normal of $\partial B_d(x_{p,j})$.
\end{Prop}
\begin{proof}
From \begin{equation}\label{5-20-1}
-\Delta u=u^p\,\,~\mbox{and}~
-\Delta \xi=pu^{p-1}\xi,
\end{equation}
we find
\begin{equation}\label{5-20-2}
-\frac{\partial \xi}{\partial x_i} \Delta u
-\frac{\partial u}{\partial x_i} \Delta \xi =u^p \frac{\partial \xi}{\partial x_i}
+pu^{p-1}\xi\frac{\partial u}{\partial x_i}.
\end{equation}
Then integrating in any smooth domain $\Omega'\subset\Omega$ on both sides of \eqref{5-20-2}, we get by the divergence theorem and Green's identities
 \begin{equation}\label{daa2}
 \begin{split}
-\int_{\partial \Omega'}&\frac{\partial \xi}{\partial \nu}
\frac{\partial u}{\partial x_i}-
\int_{\partial \Omega'}\frac{\partial u}{\partial \nu}\frac{\partial \xi}{\partial x_i}
+\int_{\partial \Omega'}\big\langle \nabla u,\nabla \xi \big\rangle \nu_i
=  \int_{\partial \Omega'}u^p \xi \nu_i.
\end{split}
\end{equation}
Hence \eqref{dafd} follows from \eqref{daa2} taking $\Omega'=B_{\theta}(x_{p,j})$.\\

Also, for $y\in\mathbb R^{2}$, from \eqref{5-20-1}, we get
\begin{equation}\label{5-20-3}
\begin{split}
\int_{ \Omega'}&\Big(-\big\langle x-y,\nabla \xi\big\rangle\Delta u
-\big\langle x-y,\nabla u\big\rangle \Delta \xi \Big)dx\\=&
\int_{ \Omega'} \Big(u^p \big\langle x-y,\nabla \xi\big\rangle
+pu^{p-1}\xi \big\langle x-y,\nabla u\big\rangle\Big)dx.
\end{split}\end{equation}
And by direct computation, we find
\begin{equation}    \label{d5-20-3}
\mbox{RHS of (\ref{5-20-3})}
= \int_{\partial \Omega'} u^p(x)\xi\big\langle x-y,\nu\big\rangle
-2\int_{ \Omega'} u^p(x)\xi (x).
    \end{equation}
Next,   we have
 \begin{equation}\label{e5-20-3}
 \begin{split}
 \mbox{LHS of (\ref{5-20-3})}
  = &
-\int_{\partial \Omega'}\frac{\partial\xi}{\partial\nu}
 \big\langle x-y,\nabla u\big\rangle
  -\int_{\partial \Omega'}\frac{\partial u}{\partial\nu} \big\langle x-y,\nabla \xi\big\rangle \\
    &
    + \int_{ \Omega'}\nabla \xi\cdot \nabla
     \big\langle x-y,\nabla u\big\rangle
  +\int_{ \Omega'}\nabla   u\cdot \nabla\big\langle x-y,\nabla \xi\big\rangle   \\
  = &
-\int_{\partial \Omega'}\frac{\partial\xi}{\partial\nu}
 \big\langle x-y,\nabla u\big\rangle
  -\int_{\partial \Omega'}\frac{\partial u}{\partial\nu} \big\langle x-y,\nabla \xi\big\rangle \\
    &
    + 2 \int_{ \Omega'}\nabla \xi\cdot \nabla u
  +\int_{ \Omega'} \Big\langle x-y,\nabla \big(\nabla \xi\cdot\nabla u\big)\Big\rangle  \\=&
 \int_{\partial \Omega'}
    \big\langle \nabla u,  \nabla \xi\big\rangle \cdot
    \big\langle x-y,\nu\big\rangle -\int_{\partial \Omega'}\frac{\partial\xi}{\partial\nu}
    \big\langle x-y,\nabla u\big\rangle
    -\int_{\partial \Omega'}\frac{\partial u}{\partial\nu}
    \big\langle x-y,\nabla \xi\big\rangle.                                                \end{split}
 \end{equation}
 Hence  from \eqref{5-20-3}, \eqref{d5-20-3} and \eqref{e5-20-3}, we get
\begin{equation}\label{daaclp1}
\begin{split}
 &\int_{\partial \Omega'}
\big\langle \nabla u,  \nabla \xi\big\rangle \cdot
\big\langle x-y,\nu\big\rangle -\int_{\partial \Omega'}\frac{\partial\xi}{\partial\nu}
\big\langle x-y,\nabla u\big\rangle\\& -\int_{\partial \Omega'}\frac{\partial u}{\partial\nu}
\big\langle x-y,\nabla \xi_{p}\big\rangle
= \int_{\partial \Omega'} u^p\xi\big\langle x-y,\nu\big\rangle
-2\int_{ \Omega'} u^p\xi.
\end{split}
\end{equation}
Then \eqref{07-08-22} follows from \eqref{daaclp1} taking $\Omega'=B_{\theta}(x_{p,j})$ and $y=x_{p,j}$.\end{proof}

We will also need the following basic formula.
\begin{Lem}[c.f. \cite{GT1983}]\label{lem2-5}
Let $\Omega'\subset \mathbb R^{2}$ be a smooth bounded domain,  $u,v\in C^2(\overline{\Omega'})$. Then
\begin{equation*}
\int_{\Omega'}\Big(u\Delta v-v\Delta u\Big)dx=\int_{\partial \Omega'}\Big(u\frac{\partial v}{\partial \nu}-v\frac{\partial u}{\partial \nu} \Big)d\sigma,
\end{equation*}
where $\nu$ is the outward unit normal of $\partial \Omega'$.
\end{Lem}

\vskip 0.5cm

\section{The asymptotic behavior of the positive solutions}  \label{s1}

In this section we will prove Theorem \ref{th1}.
A crucial step will be to improve
 \eqref{5-8-2} expanding $w_{p,j}$ up to the second order (see Section \ref{subsection:w}, in particular propositions \ref{prop3-2} and \ref{blem3-1a}). The complete  proof of Theorem \ref{th1} is then postponed to Section \ref{sub:proofThm11}. Finally in Section \ref{subsection:expansionu} we collect key further expansions for the solutions $u_{p}$ which will be useful for the proofs of theorems \ref{th1.1} and \ref{th1-1}.

 \subsection{Asymptotics for $w_{p,j}$}\label{subsection:w}

$\,$  \vskip 0.2cm

Let $w_{p,j}$ ve the rescaled solution as defined in \eqref{defwpj}, from \cite{DIP2017-1} we know that the following hold:
\begin{Lem}\label{llma}
For any small fixed $\delta,d>0$, there exist $R_\delta>1$ and $p_\delta> 1$ such that
\begin{equation*}
w_{p,j} \leq \big(4-\delta\big)\log \frac{1}{|y|}+C_\delta,~~\mbox{for}~j=1,\cdots,k,
\end{equation*}
for some $C_\delta>0$, provided $R_\delta\leq |y|\leq \frac{d}{\e_{p,j}}$ and $p\geq p_\delta$.
\end{Lem}
\begin{Lem}
For any small fixed $d>0$ and $j=1,\cdots,k$, it holds
\begin{equation}\label{abs}
\lim_{p\to +\infty}\int_{\frac{d}{\varepsilon_{p,j}}(0)}\Big(1+\frac{ w_{p,j}(z)}{p}
  \Big)^pdz
   = 8\pi.
\end{equation}
Hence, setting
\begin{equation}\label{def_Cpj}
C_{p,j}:= \displaystyle\int_{B_d(x_{p,j})}
u_{p}^{p}(y)dy,
\end{equation}
then one has
\begin{equation}\label{luoluo1}
C_{p,j}=\frac{u_{p}(x_{p,j})}{p}\int_{B_{\frac{d}{\varepsilon_{p,j}}}(0)}\left(1+\frac{w_{p,j}}{p}\right)^{p}=\frac{1}{p} \Big(8\pi \sqrt{e}+o(1)\Big).
\end{equation}
\end{Lem}
\begin{proof} \eqref{abs} has been proved in \cite{DIP2017-1}, using the convergence in \eqref{5-8-2} and, thanks to Lemma \ref{llma}, the dominated convergence theorem. From definition \eqref{defwpj}, using \eqref{ConvMax} and
\eqref{abs}, we have
\begin{equation*}
\begin{split}
C_{p,j}=&\frac{u_p(x_{p,j})}{p}\int_{B_{\frac{d}{\varepsilon_{p,j}}}(0)}\Big(1+\frac{ w_{p,j}}{p}
  \Big)^p
   \\=& \frac{\big(\sqrt{e}+o(1)\big)}{p}\big(8\pi+o(1)\big)=\frac{1}{p} \Big(8\pi \sqrt{e}+o(1)\Big).
\end{split}
\end{equation*}
\end{proof}

 Next we derive an expansion of $u_p$ which will be useful to improve the asymptotic expansion of $w_{p,j}$ (see  \eqref{9-27-2}). We will also apply it to get  the non-degeneracy result (see the proofs of propositions  \ref{dprop-luo1} and \ref{prop-gl}).
We point out that at the end of this section we will be able to further improve this expansion (see Lemma \ref{prop:expansionupwithdelta}).

\begin{Lem}\label{prop3-1}
For any fixed small $d>0$, it holds
\begin{equation}\label{luo-1}
u_p(x)= \sum^k_{j=1}C_{p,j}G(x_{p,j},x)+
o\Big(\sum^k_{j=1}\frac{\varepsilon_{p,j}}{p}\Big)\,\,
~\mbox{in}~C^1\Big(\Omega\backslash \displaystyle\bigcup^k_{j=1} B_{2d}(x_{p,j})\Big),
\end{equation}
where $ C_{p,j}$ is defined in \eqref{def_Cpj}.
\end{Lem}
\begin{proof}
For $x\in \Omega\backslash  \bigcup^k_{j=1} B_{2d}(x_{p,j})$, we have
\begin{equation}\label{5-8-11}
\begin{split}
u_p(x)=& \int_{\Omega}G(y,x)
 u_{p}^{p}(y) dy\\=&\sum^k_{j=1}
 \int_{B_d(x_{p,j})}G(y,x)
u_{p}^{p}(y)dy+\int_{\Omega\backslash \bigcup^k_{j=1} B_{d}(x_{p,j})}G(y,x)
u_{p}^{p}(y)dy.
\end{split}
\end{equation}
Using \eqref{11-14-03N} we have
\begin{equation}\label{5-8-12}
\begin{split}
 \int_{\Omega\backslash  \bigcup^k_{j=1} B_{d}(x_{p,j})}G(y,x)
u_{p}^{p}(y)dy
=O\Big(\frac{C^p}{p^p}\Big)~~\mbox{uniformly for}~ x\in \Omega\backslash  \bigcup^k_{j=1} B_{2d}(x_{p,j}).
\end{split}\end{equation}
Moreover, it holds
\begin{equation}\label{5-8-13}
\begin{split}
 \int_{B_d(x_{p,j})}G(y,x)
u_{p}^{p}(y)dy=  G(x_{p,j},x)\int_{B_d(x_{p,j})}
u_{p}^{p}(y)dy+ \int_{B_d(x_{p,j})}\big(G(y,x)-G(x_{p,j},x)\big)
u_{p}^{p}(y)dy,
\end{split}
\end{equation}
where by Taylor's expansion, it follows
\begin{equation}\label{5-8-14}
\begin{split}
\int_{B_d(x_{p,j})}&\Big(G(y,x)-G(x_{p,j},x)\Big) u_{p}^{p}(y)dy
\\=&\frac{\varepsilon_{p,j} u_{p}(x_{p,j})}{p}
\int_{B_{\frac{d}{\varepsilon_{p,j}}}(0)}
\big\langle\nabla G(x_{p,j},x),z\big\rangle\Big(1+\frac{w_{p,j}(z)}{p}\Big)^pdz\\&
+O\left(\frac{\varepsilon^{3/2}_{p,j} }{p}
\int_{B_{\frac{d}{\varepsilon_{p,j}}}(0)}
\big| z\big|^{\frac{3}{2}}\cdot \Big(1+\frac{w_{p,j}(z)}{p}\Big)^pdz\right).
\end{split}
\end{equation}
Using Lemma \ref{llma}, the fact that $\big|\nabla G(x_{p,j},x)\big|$ is bounded uniformly in
$\Omega\backslash  \bigcup^k_{j=1} B_{2d}(x_{p,j})$ and the dominated convergence theorem,  we have
\begin{equation}\label{d5-8-14}
\begin{split}
\lim_{p\rightarrow \infty}&\int_{B_{\frac{d}{\varepsilon_{p,j}}}(0)}
\big\langle\nabla G(x_{p,j},x),z\big\rangle\Big(1+\frac{w_{p,j}(z)}{p}\Big)^pdz\\
=&\int_{\R^2}
\big\langle\nabla G(x_{\infty,j},x),z\big\rangle e^{U(z)}dz=0, ~~\mbox{uniformly in $\Omega\backslash  \bigcup^k_{j=1} B_{2d}(x_{p,j})$},
\end{split}
\end{equation}
and
\begin{equation}\label{f5-8-14}
\lim_{p\rightarrow \infty}
\int_{B_{\frac{d}{\varepsilon_{p,j}}}(0)}
\big| z\big|^{\frac{3}{2}}\cdot \Big(1+\frac{w_{p,j}(z)}{p}\Big)^pdz
=\int_{\R^2} |z|^{\frac{3}{2}}  e^{U(z)}dz.
\end{equation}
Hence from \eqref{5-8-14}, \eqref{d5-8-14} and \eqref{f5-8-14}, we get
\begin{equation}\label{g5-8-14}
\begin{split}
\int_{B_d(x_{p,j})}&\Big(G(y,x)-G(x_{p,j},x)\Big) u_{p}^{p}(y)dy=
o\Big(\frac{\varepsilon_{p,j}}{p}\Big) ~\mbox{uniformly in}~ \Omega\backslash  \bigcup^k_{j=1} B_{2d}(x_{p,j}).
\end{split}
\end{equation}
 Then \eqref{5-8-13} and \eqref{g5-8-14} imply
\begin{equation}\label{5-8-15}
\begin{split}
 \int_{B_d(x_{p,j})}G(y,x)
u_{p}^{p}(y)dy=  G(x_{p,j},x)\int_{B_d(x_{p,j})}
u_{p}^{p}(y)dy+
o\Big(\frac{\varepsilon_{p,j}}{p}\Big) ~\mbox{uniformly in}~ \Omega\backslash  \bigcup^k_{j=1} B_{2d}(x_{p,j}).
\end{split}
\end{equation}
So from \eqref{5-8-11}, \eqref{5-8-12} and \eqref{5-8-15}, we find
\begin{equation*}
u_p(x)= \sum^k_{j=1}C_{p,j}G(x_{p,j},x)+
o\Big(\sum^k_{j=1}\frac{\varepsilon_{p,j}}{p}\Big)\,\,
~\mbox{uniformly in}~ \Omega\backslash \displaystyle\bigcup^k_{j=1} B_{2d}(x_{p,j}).
\end{equation*}
Similar  to the above estimates, we also get
\begin{equation*}
\frac{\partial u_p(x)}{\partial x_i}= \sum^k_{j=1}\Big( \int_{B_d(x_{p,j})}
u_{p}^{p}(y)dy\Big) D_{x_i} G(x_{p,j},x)+
o\Big(\frac{\varepsilon_{p}}{p}\Big)\,\,~\mbox{uniformly in}~\Omega\backslash \displaystyle\bigcup^k_{j=1} B_{2d}(x_{p,j}).
\end{equation*}
\end{proof}
Let us define
\begin{equation}\label{dsa}
v_{p,j}:=p\big(w_{p,j}-U\big),
\end{equation}
where $w_{p,j}$ is the rescaled function in \eqref{defwpj} and $U$ is its limit function introduced in \eqref{defU}.
\begin{Lem}For any small fixed $d_0>0$,
it holds
\begin{equation}\label{5-7-33}
\begin{split}
-\Delta v_{p,j}=  v_{p,j}e^{U}+  v_{p,j} \overline{g}_p +g^*_p~\,\quad \mbox{for }|x|\leq \frac{d_0}{\e_{p,j}},
\end{split}\end{equation}
where
\begin{equation}\label{St0}
\overline{g}_p(x):=\frac{p\Big(e^{w_{p,j}}-e^{U}\Big)-v_{p,j}e^{U}}{v_{p,j}}= O\Big(\frac{ |w_{p,j}-U|}{ 1+|x|^{4-\delta} }\Big) \quad\mbox{for }|x|\leq \frac{d_0}{\e_{p,j}},
\end{equation}
\begin{equation}\label{St}
g^*_p(x):=p e^{w_{p,j}}
\Big(e^{p\log \big(1+\frac{w_{p,j}}{p}\big)-w_{p,j}}-1\Big)=O\left(\frac{\big(\log (1+|x|)\big)^2}{ 1+|x|^{4-\delta} }\right)~\quad\mbox{for}~~|x|\leq \frac{d_0}{\e_{p,j}},
\end{equation}
and
\begin{equation}\label{St1}
g^*_p(x)\rightarrow -\frac{U^{2}(x)}{2}e^{U(x)}~\,\,\mbox{as}~~p\rightarrow +\infty.
\end{equation}
\end{Lem}
\begin{proof}
 By direct computations,  we have
\begin{equation}\label{5-7-31}
\begin{split}
-\Delta v_{p,j}=&p \big(-\Delta w_{p,j}+\Delta U\big)=p\Big(\big(1+\frac{w_{p,j}}{p}\big)^p-e^{U}\Big)
\\=&
p\Big(e^{w_{p,j}}-e^{U} \Big)+\underbrace{p e^{w_{p,j}}
\Big(e^{p\log \big(1+\frac{w_{p,j}}{p}\big)-w_{p,j}}-1\Big)}_{:=g^*_p(x)}.
\end{split}\end{equation}
Since $\log (1+x)=x-\frac{x^2}{2}+O(x^3)$ and $e^x=1+x+O(x^2)$, then for any small  fixed  $\delta,d_0>0$, using Lemma \ref{llma}, we can check that \eqref{St} and \eqref{St1} hold.

\vskip 0.1cm

Also by Taylor's expansion and Lemma \ref{llma}, it follows that
\begin{equation}\label{5-7-32}
\begin{split}
p\Big(e^{w_{p,j}}-e^{U}\Big)=&pe^{U}\Big(e^{w_{p,j}-U}-1 \Big)\\=&
pe^{U} \left( w_{p,j}-U +O\Big(|w_{p,j}-U|^2\cdot e^{\xi\big(w_{p,j}-U\big)}\Big)\right)\\=&
e^{U(x)}v_{p,j}+ O\Big(e^{(1-\xi)U+ \xi w_{p,j} } |v_{p,j}|\cdot |w_{p,j}-U| \Big) \\=&
v_{p,j}\left(e^{U}+
\underbrace{O\Big(\frac{ |w_{p,j}-U|}{ 1+|x|^{4-\delta} }\Big)}_{:=\overline{g}_p(x)}\right)~\,\mbox{for}~~|x|\leq \frac{d_0}{\e_{p,j}}~~\mbox{and some}~~\xi\in (0,1).
\end{split}
\end{equation}
Hence \eqref{5-7-31} and \eqref{5-7-32} imply \eqref{5-7-33}.
\end{proof}
We will establish the following bound:
\begin{Prop}\label{lem3-1a}
Let $v_{p,j}$ be as in \eqref{dsa}.  Then for any small fixed $d_0>0$ and fixed $\tau\in (0,1)$, there exists $C>0$ such that
\begin{equation}\label{lpy1}
| v_{p,j}(x)|\leq C(1+|x| )^\tau~\mbox{in}~ B_{\frac{d_0}{\e_{p,j}}}(0).
\end{equation}
\end{Prop}

As a consequence of Proposition \ref{lem3-1a}, we can pass to the limit into equation
\eqref{5-7-33} and prove

\begin{Prop}\label{prop3-2}
Let $v_{p,j}$ be the function defined in \eqref{dsa}.  Then  it holds
\begin{equation*}
\lim_{p\rightarrow +\infty} v_{p,j}=w_0 ~\mbox{in}~C^2_{loc}(\R^2),
\end{equation*}
where
 $w_0$ solves the non-homogeneous linear equation
\begin{equation}\label{6-4-2}
-\Delta w_0-e^{U}w_0=-\frac{U^2}{2}e^{U}~\mbox{in}~\R^2
\end{equation}
and for any $\tau\in (0,1)$, there exists $C>0$ such that
\begin{equation}\label{boundw0}
| w_{0}(x)|\leq C(1+|x| )^\tau.
\end{equation}
\end{Prop}
\begin{Rem}\label{rem1-1}
Let $\widetilde{w}_0(y):=w_0(\sqrt{8}y)$, then by \eqref{6-4-2},
$\widetilde{w}_0(y)$ solves \eqref{6-4-1} with $f(t)=\frac{16\log^2(1+t^2)}{(1+t^2)^2}$.
\end{Rem}
\vskip 0.4cm
\begin{proof}[\textbf{Proof of Proposition \ref{lem3-1a}}]
To prove \eqref{lpy1}, we use  some ideas in  \cite[Estimate C]{CL2002}  and
\cite[Appendix B]{LY2018}. Let
\begin{equation*}
N_p:=\max_{|x|\leq \frac{d_0}{\e_{p,j}}}\frac{| v_{p,j}(x)|}{(1+|x|)^\tau},
\end{equation*}
then obviously
$$\eqref{lpy1} \Leftrightarrow N_p \leq C.$$
We will prove that $N_p \leq C$ by contradiction. The proof is organized into two steps. Set
\begin{equation*}
N_p^*:=\max_{|x|\leq \frac{d_0}{\e_{p,j}}}\max_{|x'|=|x|}\frac{|v_{p,j}(x)-v_{p,j}(x')|}{(1+|x|)^{\tau}}.
\end{equation*}
\emph{Step 1.
If $N_p\to +\infty$, then it holds
\begin{equation}\label{sst}
N_p^*=o(1)N_p.
\end{equation}}
\proof[Proof of Step 1.]
Suppose this is not true. Then there exists $c_0>0$ such that $N^*_p\geq c_0N_p$. Let $x'_p$ and
$x''_p$ satisfy $|x'_p|=|x''_p| {\leq\frac{d_{0}}{\varepsilon_{p,j}}}$ and
\begin{equation*}
N_p^*=  \frac{|v_{p,j}(x'_p)-v_{p,j}(x''_p)|}{(1+|x'_p|)^{\tau}}.
\end{equation*}
Without loss of generality, we may assume that $x'_p$ and
$x''_p$ are symmetric with respect to the $x_1$ axis. Set
\begin{equation*}
\omega^*_p(x):=v_{p,j}(x)-v_{p,j}(x^-),~\,~x^-:=(x_1,-x_2),\  \mbox{ for }x=(x_{1},x_{2}),  ~\,x_2>0.
\end{equation*}
 Hence $x^{''}_{p}={x^{'}_{p}}^{-}$ and
$N_p^*=\frac{ \omega^*_p(x'_{p}) }{ (1+|x'_{p}| )^{\tau} }$.
Let us define
\begin{equation}\label{assf}
\omega_p(x):=\frac{\omega^*_p(x)}{(1+x_2)^{\tau}}.
\end{equation}
Then direct calculations show that $\omega_p$ satisfies
\begin{equation}\label{fs}
-\Delta \omega_p-\frac{2\tau}{1+x_2}\frac{\partial \omega_p}{\partial x_2}+
\frac{\tau(1-\tau)}{(1+x_2)^2}\omega_p=\frac{g_p(x)}{(1+x_2)^\tau}+
\frac{\widetilde{{g}}_p(x)}{(1+x_2)^\tau},
\end{equation}
where  $g_p(x):=g^*_p(x)-g^*_p(x^-)$, $g^*_p$ is defined in \eqref{St}  and
\begin{equation}\label{fst}
\widetilde{g}_p(x):=
p\Big(e^{w_{p,j}(x)}-e^{w_{p,j}(x^-)}\Big)=O\left(\frac{|\omega_p|\cdot\big(\log (1+|x|)\big)}{ 1+|x|^{4-\delta} }\right)~\,\mbox{for}~~|x|\leq \frac{d_0}{\e_{p,j}}.
\end{equation}

Also let $x^{**}_p$ satisfy $|x^{**}_p|\leq\frac{d_{0}}{\varepsilon_{p,j}}$, $x^{**}_{p,2}\geq 0$ and
\begin{equation}\label{assf1}
|\omega_p(x^{**}_p)|= N^{**}_p:=\max_{|x|\leq \frac{d_0}{\e_{p,j}},~x_2\geq 0} |\omega_p(x)|.
\end{equation}
Then it follows
\begin{equation}\label{dsd}
N^{**}_p\geq \frac{|v_{p,j}(x'_p)-v_{p,j}(x''_p)|}{(1+|x'_p|)^{\tau}}=N^*_p\to +\infty.
\end{equation}
We may assume that $\omega_p(x^{**}_p)>0$. We claim that
\begin{equation}\label{1h}
|x^{**}_p|\leq C.
\end{equation}
The proof of  \eqref{1h} is divided into   two parts.

\vskip 0.2cm

\noindent
\emph{Part 1. We prove that $|x^{**}_p|\leq \frac{d_0}{2\e_{p,j}}$.}

\vskip 0.2cm
Suppose this is not true. Observe that
\begin{equation}\label{9-27-1}
\begin{split}
\omega^*_p(x^{**}_p)=&v_{p,j}(x^{**}_p)-v_{p,j}(x^{**-}_p)
=p\Big(\big(w_{p,j}(x^{**}_p)-w_{p,j}(x^{**-}_p)\big)-\big(U(x^{**}_p)-U(x^{**-}_p)\big)\Big)\\=&
\frac{p^2}{u_p(x_{p,j})}\Big( u_{p,j}(\e_{p,j}x^{**}_p+x_{p,j})-u_{p,j}(\e_{p,j}x^{**-}_p+x_{p,j})
\Big)\\&
-p\Big(U(x^{**}_p)-U(x^{**-}_p)\Big)\\=&
\frac{p^2}{u_p(x_{p,j})}\Big( u_{p,j}(\e_{p,j}x^{**}_p+x_{p,j})-u_{p,j}(\e_{p,j}x^{**-}_p+x_{p,j})
\Big),
\end{split}\end{equation}
here  in last equality we use the fact that $U$ is radial.
Moreover using \eqref{luo-1} with $d=\frac{d_0}{4}$, we get
\begin{equation}\label{9-27-2}
\begin{split}
 u_{p,j}&(\e_{p,j}x^{**}_p+x_{p,j})-u_{p,j}(\e_{p,j}x^{**-}_p+x_{p,j})
\\=&\sum^k_{l=1}C_{p,l}
\Big(G(x_{p,l},\e_{p,j}x^{**}_p+x_{p,j})-G(x_{p,l},\e_{p,j}x^{**-}_p+x_{p,j})  \Big)
+o\Big(\sum^k_{l=1}\frac{\e_{p,l}}{p} \Big).
\end{split}\end{equation}
Since $\frac{d_0}{2\e_{p,j}} \leq |x^{**}_p|\leq \frac{d_0}{\e_{p,j}}$, then
\begin{equation}\label{9-27-3}
G(x_{p,l},\e_{p,j}x^{**}_p+x_{p,j})-G(x_{p,l},\e_{p,j}x^{**-}_p+x_{p,j}) =O\Big(\e_{p,j}|x^{**}_{p,2}|\Big)\,\, ~\mbox{for}~
l=1,\cdots,k.
\end{equation}
If $l\neq j$,
\begin{equation*}
\begin{split}
G(&x_{p,l}, \e_{p,j}x^{**}_p+x_{p,j})-G(x_{p,l},\e_{p,j}x^{**-}_p+x_{p,j}) \\=&
O\Big(\e_{p,j}\big|x^{**}_p- x^{**-}_p\big|\cdot
\big| \nabla G(x_{p,l}, \xi\e_{p,j}x^{**}_p+(1-\xi)\e_{p,j}x^{**-}_p +x_{p,j}) \big|
\Big)~\,\mbox{with}~\,\xi\in (0,1),
\end{split}
\end{equation*}
and
$$\big| \nabla G(x_{p,l}, \xi\e_{p,j}x^{**}_p+(1-\xi)\e_{p,j}x^{**-}_p +x_{p,j}) \big| =O(1).$$
Then it follows
\begin{equation*}
\begin{split}
G(&x_{p,l}, \e_{p,j}x^{**}_p+x_{p,j})-G(x_{p,l},\e_{p,j}x^{**-}_p+x_{p,j})  =
 O\Big(\e_{p,j}|x^{**}_{p,2}|\Big)~\mbox{for}~l\neq j.
\end{split}
\end{equation*}
If $l= j$, since $|x^{**}_p|=|x^{**-}_p|$, then
\begin{equation*}
\begin{split}
G(&x_{p,j}, \e_{p,j}x^{**}_p+x_{p,j})-G(x_{p,j},\e_{p,j}x^{**-}_p+x_{p,j}) \\=&
-H(x_{p,j}, \e_{p,j}x^{**}_p+x_{p,j})+H(x_{p,j},\e_{p,j}x^{**-}_p+x_{p,j}) \\=&
O\Big(\e_{p,j}\big|x^{**}_p- x^{**-}_p\big|\cdot
\big| \nabla H(x_{p,j}, \xi\e_{p,j}x^{**}_p+(1-\xi)\e_{p,j}x^{**-}_p +x_{p,j}) \big|
\Big)~\,\mbox{with}~\,\xi\in (0,1),
\end{split}
\end{equation*}
and
$$\big| \nabla H(x_{p,j}, \xi\e_{p,j}x^{**}_p+(1-\xi)\e_{p,j}x^{**-}_p +x_{p,j}) \big| =O(1).$$
Then by above computations, it follows
\begin{equation*}
\begin{split}
G(&x_{p,j}, \e_{p,j}x^{**}_p+x_{p,j})-G(x_{p,j},\e_{p,j}x^{**-}_p+x_{p,j})  =
 O\Big(\e_{p,j}|x^{**}_{p,2}|\Big).
\end{split}
\end{equation*}
So from  \eqref{luoluo1}, \eqref{9-27-2} and \eqref{9-27-3}, we have
\begin{equation}\label{9-27-4}
\begin{split}
 u_{p,j} (\e_{p,j}x^{**}_p+x_{p,j})-u_{p,j}(\e_{p,j}x^{**-}_p+x_{p,j})
=O\Big(\frac{\e_{p,j}|x^{**}_{p,2}|}{p}\Big)
+o\Big(\sum^k_{l=1}\frac{\e_{p,l}}{p} \Big).
\end{split}\end{equation}
Now from  \eqref{9-27-1}, \eqref{9-27-4} and the fact $u_p(x_{p,j})\to \sqrt{e}$ by \eqref{ConvMax}, we get
\begin{equation}\label{llls}
\begin{split}
\omega^*_p(x^{**}_p)
=  O\Big(p {\e_{p,j}|x^{**}_{p,2}|}\Big)
+o\Big(p\sum^k_{l=1} \e_{p,l}\Big). \end{split}\end{equation}
As a result,
\begin{equation*}
N^{**}_p\leq C p  \e_{p,j} |x^{**}_{p,2}|^{1-\tau}+o(1)
\leq  C p  \e_{p,j}^{\tau}  +o(1)=o(1).
\end{equation*}
This is a contradiction and concludes the proof of \emph{Part 1.}

\vskip 0.2cm

\noindent
\emph{Part 2. We prove that $|x^{**}_p|\leq C$.}

\vskip 0.2cm
 Suppose this is not true.
Now by \emph{Part 1}, we have $x^{**}_p\in B_{\frac{d_0}{2\e_{p,j}}}(0)$. Thus
\begin{equation}\label{sts}|\nabla \omega_p(x^{**}_p)|=0~\mbox{and}~\Delta \omega_p(x^{**}_p)\leq 0.
\end{equation}
So from
\eqref{fs}, \eqref{fst} and \eqref{sts}, we find
\begin{equation}\label{fs2}
\begin{split}
0\leq &-\Delta \omega_p(x^{**}_p)=-
\frac{\tau(1-\tau)}{(1+x^{**}_{p,2})^2}\omega_p(x^{**}_p)
+\frac{g_p(x^{**}_p)}{(1+x^{**}_{p,2})^\tau}+ \frac{\widetilde{{g}}_p(x^{**}_p)}{(1+x^{**}_{p,2})^\tau}\\=&-
\left(\frac{\tau(1-\tau)}{(1+x^{**}_{p,2})^2}+O\Big(\frac{ \log (1+|x^{**}_{p,2}|) }{ (1+|x^{**}_{p}|^{4-\delta})(1+x^{**}_{p,2})^\tau }\Big)
\right)\omega_p(x^{**}_p)
+\frac{g_p(x^{**}_p)}{(1+x^{**}_{p,2})^\tau}.
\end{split}\end{equation}
Also by assumption $|x^{**}_p|\to \infty$, then for some large $R$, there exists $p_0$ such  that $|x^{**}_p|\geq R$ for $p\geq p_0$. And then for some $\delta\in (0,1)$, it follows
\begin{equation}\label{ssat}
O\Big(\frac{ \log (1+|x^{**}_{p,2}|) }{ (1+|x^{**}_{p}|^{4-\delta})(1+x^{**}_{p,2})^\tau }\Big)\leq \frac{2}{(1+x^{**}_{p})^{4-2\delta}}\leq  \frac{\tau(1-\tau)}{2(1+x^{**}_{p,2})^2}\,\,\,\mbox{for}~~p\geq p_0.
\end{equation}
Hence using \eqref{fs2}, \eqref{ssat} and the fact that $|x^{**}_p|\to \infty$, we deduce
\begin{equation*}
\frac{\omega_p(x^{**}_p)}{(1+x^{**}_{p,2})^2}\leq \frac{Cg_p(x^{**}_p)}{(1+x^{**}_{p,2})^{\tau}}.
\end{equation*}
Then it follows
\begin{equation}\label{sh}
 \omega_p(x^{**}_p) \leq Cg_p(x^{**}_p)(1+x^{**}_{p,2})^{2-\tau}.
\end{equation}
Combining \eqref{St} and \eqref{sh}, we obtain
\begin{equation*}
N^{**}_p\leq
\frac{C\big(\log (1+|x_p^{**}|)\big)^2}{(1+|x_p^{**}|)^{2+\tau-\delta}}\to 0~~\mbox{as}~~p\to +\infty,
\end{equation*}
which is a contradiction with \eqref{dsd}. Then the proof of \emph{Part 2} is concluded and \eqref{1h} follows.\\

Now let \[\omega_p^{**}(x)=\frac{\omega_p(x)}{N_p^{**}},\]
where $\omega_p$ is defined in \eqref{assf} and $N_p^{**}$ is defined in \eqref{assf1}.
From \eqref{fs}, it follows that $\omega_p^{**}(x)$ solves
\begin{equation}\label{afs}
-\Delta \omega^{**}_p-\frac{2\tau}{1+x_2}\frac{\partial \omega^{**}_p}{\partial x_2}+
\frac{\tau(1-\tau)}{(1+x_2)^2}\omega^{**}_p=\frac{g_p(x)}{N_p^{**}(1+x_2)^\tau}+
\frac{\widetilde{{g}}_p(x)}{N_p^{**}(1+x_2)^\tau}.
\end{equation}
Moreover $|\omega^{**}_{p}(x)|\leq 1$ and
$|\omega^{**}_{p}(x^{**}_p)|=1$, hence $\omega^{**}_p(x)\to \omega$ uniformly
in any compact subset of $\R^2$. And   $\omega\not\equiv0$ because $\omega_p^{**}(x_{p}^{**})=1$ and  $|x_{p}^{**}|\le C$. Observe that from \eqref{St},
$$\frac{g_p(x)}{N_p^{**}(1+x_2)^{\tau}}\to 0~\mbox{uniformly in any compact set of}~\big\{|x|\leq \frac{d_0}{\e_{p,j}}\big\}~\mbox{as}~p\to +\infty,$$
and that for any  $\phi(x)\in C^{\infty}_0(B_{\frac{d_0}{\e_{p,j}}}(0))$, we have also
\begin{equation*}
\begin{split}
\frac{1}{N_p^{**}} &\int_{B_{\frac{d_0}{\e_{p,j}}}(0)}\Big( \widetilde{g}_p(x) -e^{U(x)} \omega_p(x)  \Big)  \phi(x)dx \\=&
\frac{1}{N_p^{**}} \int_{B_{\frac{d_0}{\e_{p,j}}}(0)}\left( p\big(e^{w_{p,j}(x)}-e^{w_{p,j}(x^-)}\big)-e^{U(x)}\omega_p(x)\right)  \phi(x)dx\\=&
O\Big(\int_{B_{\frac{d_0}{\e_{p,j}}}(0)}\big|e^{w_{p,j}(x)}-U(x)\big| e^{U(x)}\cdot|\phi(x)|dx\Big)
\to 0.
\end{split}
\end{equation*}
Here  we use Lemma \ref{llma} and the dominated convergence theorem.
Hence, passing to the limit $p\to \infty$ into \eqref{afs},  we can deduce that   $\omega$ solves
\begin{equation*}
\begin{cases}
-\Delta \omega -\frac{2\tau}{1+x_2}\frac{\partial \omega}{\partial x_2}
+\left(\frac{\tau(1-\tau)}{(1+x_2)^2}-e^{U(x)} \right)\omega =0,~x_2>0,\\[2mm]
\omega(x_1,0)=0.
\end{cases}
\end{equation*}
Hence $\bar \omega=(1+x_2)^{\tau}\omega$ satisfies
\begin{equation*}
\begin{cases}
-\Delta \bar \omega  -e^{U} \bar \omega =0,~x_2>0,\\[2mm]
\bar\omega(x_1,0)=0.
\end{cases}
\end{equation*}
And then using Lemma \ref{llm}, we have
\begin{equation}\label{lss}
\bar \omega(x)=c_0\frac{\partial U(\frac{x}{\lambda})}{\partial \lambda}\Big|_{\lambda=1}+c_1  \frac{\partial U(x)}{\partial x_1}+c_2
\frac{\partial U(x)}{\partial x_2},~~
\mbox{for some constants $c_0$, $c_1$ and $c_2$}.
\end{equation}
On the other hand, by the definition of $v_{p,j}(x)$ and $w^*_p(x)$, we know
\begin{equation*}
\begin{split}
\nabla \omega_p^*(0)=&\nabla \big( v_{p,j}(x)-v_{p,j}(x^-)\big)\big|_{x=0}\\=&
p\big(\nabla w_{p,j}(x)-\nabla U(x)\big)\big|_{x=0}-p\big(\nabla w_{p,j}(x^-)-\nabla U(x^-)\big)\big|_{x=0} \\=&
\frac{p^2\varepsilon_{p,j}}{u_{p}(x_{p,j})}\nabla u_{p}(x_{p,j}+\varepsilon_{p,j}x)\big|_{x=0}-
\frac{p^2\varepsilon_{p,j}}{u_{p}(x_{p,j})}\nabla u_{p}(x_{p,j}+\varepsilon_{p,j}x^-)\big|_{x=0}=0.
\end{split}
\end{equation*}
Here the last equality follows by the fact that $x_{p,j}$ is a local maximum point of $u_{p}(x)$.
And then
\begin{equation*}
\nabla \big((1+x_2)^{\tau} w_p^{**}\big)  \big|_{x=0}=
\frac{\nabla \omega_p^*(0)}{N_p^{**}}= 0,
\end{equation*}
which means $\nabla \bar\omega(0)=0$. Then from \eqref{lss}, we find $c_1=c_2=0$.
Moreover from $\bar\omega(x_1,0)=0$, we also have $c_0=0$. So $\bar \omega=0$, which is a contradiction.
Hence \eqref{sst} follows, which concludes the proof of \emph{Step 1}.
\vskip 0.2cm
\noindent\emph{Step 2. We prove that $N_{p}\leq C$ by contradiction.}
\proof[Proof of Step 2]
Suppose by contradiction that $N_p\to +\infty$, as $p\to +\infty$ and set
\begin{equation*}
\psi_{p,j}(r)=\frac{1}{2\pi}\int^{2\pi}_0v_{p,j}(r,\theta)d \theta,~\,\,r=|x|.
\end{equation*}
Then by \emph{Step 1}, it follows that
\begin{equation*}
\begin{split}
 \frac{|\psi_{p,j}(r)|}{(1+r)^\tau}=&
 \frac{1}{2\pi(1+r)^\tau} {\Big |}\int^{2\pi}_0v_{p,j}(r,\theta)d \theta
 {\Big |}\\\leq  &
 \max_{\theta\in [0,2\pi]} \frac{|v_{p,j}(r,\theta)|}{ (1+r)^\tau}
 +
 \frac{1}{2\pi(1+r)^\tau}  \left| \int^{2\pi}_0\Big (v_{p,j}(r,\theta)- \max_{\theta\in [0,2\pi]}  \big | v_{p,j}(r,\theta) \big | \Big)  d \theta  \right|.
\end{split}
\end{equation*}
Then we have
\begin{equation*}
\begin{split}
  \max_{r\leq \frac{d_0}{\e_{p,j}}}\frac{|\psi_{p,j}(r)|}{(1+r)^\tau}\leq &
  \max_{r\leq \frac{d_0}{\e_{p,j}}} \max_{\theta\in [0,2\pi]} \frac{|v_{p,j}(r,\theta)|}{ (1+r)^\tau}
 +C \max_{r\leq \frac{d_0}{\e_{p,j}}} \max_{\theta_1,\theta_2\in [0,2\pi]}
 \frac{\big|v_{p,j}(r,\theta_1)-v_{p,j}(r,\theta_2)\big|}{ (1+r)^\tau}
 \\=&
  \max_{|x|\leq \frac{d_0}{\e_{p,j}}} \frac{|v_{p,j}(x)|}{ (1+|x|)^\tau} +
 C \max_{|x|\leq \frac{d_0}{\e_{p,j}}} \max_{|x'|=|x|}\frac{|v_{p,j}(x)-v_{p,j}(x')|}{ (1+|x'|)^\tau}
\\=&N_p+C N^*_p =N_p\big(1+o(1)\big).
\end{split}
\end{equation*}
And similarly, it follows
\begin{equation*}
\begin{split}
  \max_{r\leq \frac{d_0}{\e_{p,j}}}\frac{|\psi_{p,j}(r)|}{(1+r)^\tau}\geq &
  \max_{r\leq \frac{d_0}{\e_{p,j}}} \max_{\theta\in [0,2\pi]} \frac{|v_{p,j}(r,\theta)|}{ (1+r)^\tau}
 -C \max_{r\leq \frac{d_0}{\e_{p,j}}} \max_{\theta_1,\theta_2\in [0,2\pi]}
 \frac{\big|v_{p,j}(r,\theta_1)-v_{p,j}(r,\theta_2)\big|}{ (1+r)^\tau}
 \\=&
 \max_{|x|\leq \frac{d_0}{\e_{p,j}}} \frac{|v_{p,j}(x)|}{ (1+|x| )^\tau} -
 C \max_{|x|\leq \frac{d_0}{\e_{p,j}}} \max_{|x'|=|x|}\frac{|v_{p,j}(x)-v_{p,j}(x')|}{ (1+|x'|)^\tau}
\\=&N_p-C N^*_p =N_p\big(1+o(1)\big).
\end{split}
\end{equation*}
Hence we can find that
\begin{equation*}
\max_{r\leq \frac{d_0}{\e_{p,j}}}\frac{|\psi_{p,j}(r)|}{(1+r)^\tau}=N_p\big(1+o(1)\big).
\end{equation*}

Assume that $\frac{|\psi_{p,j}(r)|}{(1+r)^\tau}$ attains its maximum at $r_p$.
Then we claim
\begin{equation}\label{aatty}
r_p\leq C.
\end{equation}
In fact, let $\phi(x)=\frac{\partial U(\frac{x}{\lambda})}{\partial \lambda}\big|_{\lambda=1}$ and  we recall that
$-\Delta\phi(x)=e^{U(x)}\phi(x)$.
Now multiplying \eqref{5-7-33}   by $\phi$ and using integration by parts, we have
\begin{equation}\label{ss1}
\begin{split}
\int_{|x|=r}&\Big[\frac{\partial v_{p,j}}{\partial \nu}\phi-  \frac{\partial \phi}{\partial \nu}v_{p,j}\Big]\,d\sigma
\\=&
\int_{|x|\leq r}  v_{p,j}\overline{g}_p\phi\, dx +
\int_{|x|\leq r} g^*_p\phi\, dx\\
=&  O\Big(N_p \int_{|x|\leq r}\overline{g}_p\phi(1+|x|)^{\tau}\,dx+
\int_{|x|\leq r} g^*_p\phi\,dx\Big)  \\
=&
o(1)N_p+O(1).
\end{split}
\end{equation}
Here the last identity holds by \eqref{5-8-2}.
Also for $r^2>8$, we know
\begin{equation}\label{ss2}
\begin{split}
\int_{|x|=r}\Big[\frac{\partial v_{p,j}}{\partial \nu}\phi-  \frac{\partial \phi}{\partial \nu}v_{p,j}\Big]\,d\sigma
= & \frac{8-r^2}{8+r^2}\int_{|x|=r} \frac{\partial v_{p,j}}{\partial \nu}\, d\sigma
 -\phi'(r)\int_{|x|=r}  v_{p,j}\, d\sigma\\
=& 2\pi \frac{8-r^2}{8+r^2} r\psi_{p,j}'(r)   +64\pi\frac{ r^2(1+r)^{\tau}}{(8+r^2)^2}N_p.
\end{split}
\end{equation}
Also for $r^2>8$ and $\tau\in(0,1)$,  it holds
\begin{equation*}
\frac{ r^2(1+r)^{\tau}}{(8+r^2)^2}\leq  \frac{C_\tau}{r},
\end{equation*}
for example, we take $C_\tau:=\displaystyle\max_{r^2\geq 8}\frac{ r^3(1+r)^{\tau}}{(8+r^2)^2}$.
Hence for $r^2>8$, from \eqref{ss1} and \eqref{ss2}, we have
\begin{equation*}
|\psi'_{p,j}(r)|=  \frac1r o\big(N_p\big)+  \frac 1r O(1).
\end{equation*}
And then we find \begin{equation}\label{aas}
\begin{split}
(1+r_p)^{\tau}N_p\big(1+o(1)\big)=&|\psi_{p,j}(r_p)|\leq \int^{r_p}_3|\psi'_{p,j}(r)|dr +|\psi_{p,j}(3)|
\\ \leq &o\big(N_p \big)+O\big(1\big)+|\psi_{p,j}(3)|,
\end{split}\end{equation}
also it follows
\begin{equation}\label{aasa}
\psi_{p,j}(3)\leq 4^{\tau}\max_{r\leq \frac{d_0}{\e_{p,j}}}\frac{|\psi_{p,j}(r)|}{(1+r)^\tau}=O\big( N_p \big).
\end{equation}
Hence  from \eqref{aas} and \eqref{aasa}, we have
$(1+r_p)^{\tau}N_p= O\big(N_p \big)$,
 which implies \eqref{aatty}.
\vskip 0.2cm

Now integrating \eqref{5-7-33} and using Lemma \ref{llma}, we get
\begin{equation}\label{equazionePsi}
-\Delta \psi_{p,j}=\psi_{p,j}e^{U}+\overline{h}_p+h^*_p,\qquad~\mbox{ for}~|x|\leq \frac{d_0}{\e_{p,j}},
\end{equation}
where $\gamma>0$ is a small fixed constant,
$$\overline{h}_p(r):=\frac{1}{2\pi}\displaystyle\int^{2\pi}_0\psi_{p,j}(r,\theta) \overline{g}_p(r,\theta)d \theta,$$
and
$$ h^*_{p}(r):=\frac{1}{2\pi}\displaystyle\int^{2\pi}_0 g^*_p(r,\theta)d \theta,$$
where $\overline{g}_p$  and $g^*_p$ are the same as in \eqref{St0} and \eqref{St} correspondingly.

Next we define $\psi^*_{p,j}(x)=\frac{\psi_{p,j}(|x|)}{\psi_{p,j}(r_p)}$ and   pass to the limit in the equation \eqref{equazionePsi} divided by $\psi_{p,j}(r_p)$. One has
\begin{equation*}
|\psi^*_{p,j}(x)|\leq \frac{C\big|\psi_{p,j}(|x|)\big|}{N_p}\leq C\big(1+|x|\big)^{\tau}.
\end{equation*}
Also from \eqref{St}, we know
\begin{equation*}
\frac{h^*_p(r)}{\psi_{p,j}(r_p)}\leq \frac{Ch^*_p(r)}{N_p}\leq \frac{C}{N_p}\to 0.
\end{equation*}
Moreover from \eqref{5-7-32}, Lemma \ref{llma} and the dominated convergence theorem,  it follows
\begin{equation*}
\int_{B_{\frac{d_0}{\e_{p,j}}}(0)}
\frac{\overline{h}_p(|x|)}{\psi_{p,j}(r_p) } \phi(x)dx \to 0\,\,~\mbox{for any radial function}~\phi(x)\in C^{\infty}_0\big(B_{\frac{d_0}{\e_{p,j}}}(0)\big).
\end{equation*}
Hence from the above computations and the dominated convergence theorem, we can deduce that $\psi^*_{p,j}\to \psi(|x|)$ in  $C^2_{loc}(\R^2)$ and $\psi$ satisfies
\begin{equation*}
-\psi''-\frac{1}{r}\psi'=e^{U}\psi.
\end{equation*}
Therefore, it follows
\begin{equation}\label{stsyd}
\psi(r)=c_0\frac{8-r^2}{8+r^2},~\mbox{with some constant}~c_0>0.
\end{equation}
Since $\psi^*_{p,j}{ (r_{p})}=1$ and $r_p\leq C$, we find $\psi\not\equiv 0$.

On the other hand, we know that $\psi_{p,j}(0){  =v_{p,j}(0)=0}$  by definition.  Therefore, we have $\psi(0)=0$, which, together with
\eqref{stsyd}, implies $\psi\equiv 0$. This is a contradiction. This gives \eqref{lpy1} and completes the proof of Proposition \ref{lem3-1a}.
\end{proof}

Next  we introduce
\begin{equation}\label{sat}
k_{p,j}:=p\big(v_{p,j}-w_0\big),
\end{equation}
where $v_{p,j}$ is defined in \eqref{dsa} and $w_{0}$ is its limit function as $p\rightarrow +\infty$ (see Proposition \ref{prop3-2}).
\begin{Lem}For any small fixed $d_0>0$, it holds
\begin{equation}
\label{b5-7-33}
\begin{split}
-\Delta k_{p,j}=  k_{p,j}e^{U}+  h^*_p~\,\,\,\mbox{ in}~~\Big\{|x|\leq \frac{d_0}{\e_{p,j}}\Big\},
\end{split}
\end{equation}
with $ h^*_p(x)=O\Big(\frac{ \log(1+|x|)}{ 1+|x|^{4-\delta-2\tau}}\Big)$.
\end{Lem}
\begin{proof}
 {Recalling that $v_{p,j}=p\big(w_{p,j}-U\big)$ and by direct computations,  we have
\begin{equation}\label{b5-7-31}
\begin{split}
-\Delta k_{p,j}=& p\Big( p \big(-\Delta w_{p,j}+\Delta U\big)+\Delta w_0\Big)
\\=&p\Big(p\big(e^{w_{p,j}}-e^{U} \big)+ g^*_p-e^{U}w_0+\frac{U^2}{2}e^{U}\Big)\\=&
\underbrace{p\Big(p\big(e^{w_{p,j}}-e^{U} \big)-e^{U}w_0\Big)}_{I_1} + \underbrace{p\Big(g^*_p+\frac{U^2}{2}e^{U}\Big)}_{I_2},
\end{split}\end{equation}
where $g^*_p$ is the same as in \eqref{dsa}.
Now similarly as in \eqref{5-7-32} we compute
\begin{equation*}
\begin{split}
p\big(e^{w_{p,j}}-e^{U} \big) -e^{U}w_0=
  \Big( v_{p,j}-w_0\Big) e^{U}+O\Big(\frac{ |v_{p,j}|}{p( 1+|x|^{4-\delta} ) }\Big).
\end{split}\end{equation*}
And then, by Proposition \ref{lem3-1a}, it follows
\begin{equation}\label{is}
\begin{split}
I_1= k_{p,j} e^{U}+O\Big(\frac{ 1}{ 1+|x|^{4-\delta-2\tau}}\Big).
\end{split}\end{equation}
Recalling $\log (1+x)=x-\frac{x^2}{2}+O(x^3)$ and $e^x=1+x+O(x^2)$, we have
\begin{equation*}
 g^*_p =pe^{w_{p,j}}\Big(-\frac{w^2_{p,j}}{2p}+O\big(\frac{|w_{p,j}|^3}{p^2}\big)\Big)=
 e^{w_{p,j}}\Big(-\frac{w^2_{p,j}}{2}+O\big(\frac{|w_{p,j}|^3}{p}\big)\Big).
\end{equation*}
Hence using Lemma \ref{llma}, \eqref{dsa} and Proposition \ref{lem3-1a}, it holds
\begin{equation}\label{b5-7-32}
\begin{split}
I_2=&
p\left(e^{w_{p,j}}\Big(-\frac{w^2_{p,j}}{2}+O\big(\frac{|w_{p,j}|^3}{p}\big)\Big) +\frac{U^2}{2}e^{U}\right) \\=&
p\left(-\frac{w^2_{p,j}}{2}e^{w_{p,j}} +\frac{U^2}{2}e^{U}\right)+O\Big(e^{w_{p,j}}|w_{p,j}|^3 \Big)\\{=}&
O\Big(\frac{ |v_{p,j}| \cdot\log(1+|x|)}{ 1+|x|^{4-\delta}}\Big)+O\Big(\frac{ \log^3(1+|x|)}{ 1+|x|^{4-\delta}}\Big)\\=&O\Big(\frac{ \log(1+|x|)}{ 1+|x|^{4-\delta-2\tau}}\Big)~\qquad\mbox{ in}~~\Big\{|x|\leq \frac{d_0}{\e_{p,j}}\Big\}.
\end{split}
\end{equation}
And then \eqref{b5-7-31}, \eqref{is} and \eqref{b5-7-32} imply \eqref{b5-7-33}.}
\end{proof}

 Now we establish the following bound on $k_{p,j}$:
\begin{Prop}\label{blem3-1a}
Let $k_{p,j}$ be as in \eqref{sat}. Then for any small fixed $d_0,\tau_1>0$, there exists $C>0$ such that
\begin{equation}\label{blpy1}
|k_{p,j}(x)|\leq C(1+|x| )^{\tau_1}~\mbox{ in}~ B_{\frac{d_0}{\e_{p,j}}}(0).
\end{equation}
 As a consequence
\[w_{p,j}=U+\frac{w_{0}}{p}+O\left(\frac{1}{p^{2}}\right)
\,\,\mbox{ in }C^{2}_{loc}(\mathbb R^{2}).\]
\end{Prop}
\begin{proof} The proof of \eqref{blpy1} is similar to the one of Proposition \ref{lem3-1a}. Hence we postpone it to the Appendix \ref{76}.
\end{proof}

We can now improve \eqref{abs} as follows.
\begin{Prop}\label{Prop:5-12-52}
For any small fixed $d>0$, there exists a small constant $\delta\in (0,1)$ such that
\begin{equation}\label{5-12-52}
\int_{\frac{d}{\varepsilon_{p,j}}(0)}\Big(1+\frac{ w_{p,j}(z)}{p}
  \Big)^pdz
   = 8\pi\left(1-\frac{3}{p}\right)+O\Big(\frac{1}{p^{2-\delta}}\Big).
\end{equation}
\end{Prop}
\begin{proof}
Firstly, since $\frac{w_{p,j}(x)}{p}=\frac{u_p(x_{p,j}+\e_{p,j} x)}{u_p(x_{p,j})}-1$, then $-1<\frac{w_{p,j}(x)}{p}\leq 0$. And using the fact that $\log(1+a)<a$ for any $a\in (-1,0)$ and Lemma \ref{llma}, we have
  \begin{equation}\label{luo-a1}
  \begin{split}
 \int_{B_\frac{d}{\varepsilon_{p,j}}(0)\backslash B_p(0)}& \Big(1+\frac{  w_{p,j}}{p}
\Big)^p \\=& \int_{B_\frac{d}{\varepsilon_{p,j}}(0)\backslash B_p(0)} e^{ p\log \big(1+\frac{w_{p,j}}{p}
  \big) }\leq
   \int_{B_\frac{d}{\varepsilon_{p,j}}(0)\backslash B_p(0)} e^{w_{p,j} }\\=& O\Big(
   \int_{B_\frac{d}{\varepsilon_{p,j}}(0)\backslash B_p(0)}\frac{1}{1+|z|^{4-\delta}}dz\Big)=O\Big(\frac{1}{p^{2-\delta}}\Big)~\,\mbox{for some}~\delta\in (0,1) .
  \end{split}
  \end{equation}
Next using Proposition \ref{lem3-1a}, we know
\begin{equation*}
|w_{p,j}|=\big|U+\frac{v_{p,j}}{p}\big|=O\Big(\log p+p^{\tau-1}\Big)=O\Big(\log p \Big)~\,\mbox{in}\, B_p(0),
\end{equation*}
which gives
\begin{equation}\label{adsst}
\frac{|w_{p,j}|}{p},\,\frac{|w_{p,j}|^2}{p},\,\frac{|w_{p,j}|^3}{p}\to 0~\,\mbox{uniformly in}\, B_p(0).
\end{equation}
Recalling that $w_{p,j}=U+\frac{v_{p,j}}{p}$ by definition  \eqref{dsa} and using Taylor's expansion, it follows
  \begin{equation*}
  \begin{split}
 \int_{ B_p(0)}& \Big(1+\frac{ w_{p,j}}{p}
 \Big)^p = \int_{  B_p(0)} e^{ p\log \big(1+\frac{w_{p,j}}{p}
 \big) }\\=&
  \int_{ B_p(0)} e^{w_{p,j}}\Big(1-\frac{w^2_{p,j}}{2p}+o\big(\frac{w^2_{p,j}}{p}\big) \Big) \\=&
  \int_{ B_p(0)} e^{U}\Big(1+\frac{v_{p,j}}{p}+O\big(\frac{|v_{p,j}|^2}{p^2}\big) \Big)
  \Big(1-\frac{w^2_{p,j}}{2p}+O\big(\frac{|w_{p,j}|^3}{p^2}\big) \Big).
  \end{split}
  \end{equation*}
Hence using  \eqref{adsst}, the bound in Proposition \ref{lem3-1a} and the dominated convergence theorem one has
\begin{equation*}
 \int_{ B_p(0)} \Big(1+\frac{ w_{p,j}}{p}
 \Big)^p =
 \int_{ B_p(0)} e^{U}\Big(1+\frac{v_{p,j}}{p}-\frac{w^2_{p,j}}{2p}\Big)+O\Big( \frac{1}{p^2}  \Big),\\
\end{equation*}
which, since $v_{p,j}=w_{0}+\frac{k_{p,j}}{p}$ by  \eqref{sat},  becomes
 \begin{equation*}
   \begin{split}
 \int_{ B_p(0)} \Big(1+\frac{ w_{p,j}}{p}
 \Big)^p= &
 \int_{ B_p(0)} e^{U}\Big(1+\frac{w_0}{p}-\frac{U^2}{2p}\Big)+O\Big( \frac{1}{p^2}  \Big)\\&
   +\frac{1}{p^2}\int_{B_p(0)}e^{U}\Big(k_{p,j}-\frac{1}{2}\big(U+w_{p,j}\big)v_{p,j}\Big).
  \end{split} \end{equation*}
Then we can pass to the limit by the classical dominated convergence
theorem, thanks to the bounds   on $w_{0}$,  $k_{p,j}$, $w_{p,j}$ and  $v_{p,j}$ (see \eqref{boundw0},  Proposition \ref{blem3-1a}, Lemma \ref{llma} and Proposition \ref{lem3-1a}, respectively), and deduce that
    \begin{equation}\label{heart}
  \begin{split}
 \int_{ B_p(0)} \Big(1+\frac{ w_{p,j}}{p}
 \Big)^p =
 &
 \int_{\R^2} e^{U}+ \frac{1}{p}   \int_{\R^2} e^{U}\Big( w_0 -\frac{U^2}{2}\Big) +O\Big( \frac{1}{p^2}  \Big)\\
 =&
 \ 8\pi-\frac{1}{p}
 \int_{\R^2}  \Delta w_0+O\Big( \frac{1}{p^2}  \Big).
  \end{split}
  \end{equation}
  Also from Lemma \ref{lll3} and Remark \ref{rem1-1}, we have
\begin{equation}\label{5-26-2}
 \widetilde{w}_0(y)=w(|y|)+ c_0\frac{\partial U(\frac{y}{\lambda})}{\partial \lambda}\Big|_{\lambda=\frac{1}{\sqrt{8}}}+ \sum_{i=1}^2 c_i\frac{\partial U(\sqrt{8} y)}{\partial y_i}\,\,~\mbox{with some}~c_i\in \R~\mbox{for}~i=0,1,2.
\end{equation}
Then using \eqref{5-26-1} and \eqref{5-26-2}, we find
\begin{equation}\label{abcd}
\begin{split}
\nabla \widetilde{w}_0(y)=&\Big(16\int^{+\infty}_0t\frac{t^2-1}{(t^2+1)^3}\log^2 (1+t^2)dt \Big)\frac{y}{|y|^2}\\&+\nabla \left(c_0\frac{\partial U(\frac{y}{\lambda})}{\partial \lambda}\Big|_{\lambda=\frac{1}{\sqrt{8}}}+ \sum_{i=1}^2 c_i\frac{\partial U(\sqrt{8} y)}{\partial y_i}\right)+o\big(\frac{1}{|y|}\big)
\\=& \Big(16\int^{+\infty}_0t\frac{t^2-1}{(t^2+1)^3}\log^2 (1+t^2)dt \Big)\frac{y}{|y|^2}+o\big(\frac{1}{|y|}\big)\,\,~\mbox{as}~|y|\to +\infty.
\end{split}\end{equation}
Also by direct computation, we have
\begin{equation}\label{abc1}
16\int^{+\infty}_0t\frac{t^2-1}{(t^2+1)^3}\log^2 (1+t^2)dt =12.
\end{equation}
Hence \eqref{abcd} and \eqref{abc1} imply
\begin{equation*}
\begin{split}
\nabla \widetilde{w}_0(y)= \frac{12y}{|y|^2}+o\big(\frac{1}{|y|}\big)\,\,~\mbox{as}~|y|\to +\infty.
\end{split}
\end{equation*}
So we find
\begin{equation}\label{5-11-9}
\int_{\R^2}  \Delta w_0(x)dx=\lim_{R\to \infty}\int_{\partial B_R(0)} \frac{\partial w_0(x)}{\partial \nu} d\sigma=
 24\pi.
\end{equation}
Hence \eqref{5-12-52} follows by \eqref{luo-a1}, \eqref{heart} and \eqref{5-11-9}.
\end{proof}

\subsection{Proof of Theorem \ref{th1}}\label{sub:proofThm11}
\begin{proof}
By the Green representation theorem and \eqref{11-14-03N}, we have
\begin{equation}\label{s5-8-3}
\begin{split}
u_p(x_{p,j})=& \int_{\Omega}G(y,x_{p,j})
 u_{p}^p(y)\,dy\\=&
 \sum^k_{l=1}\int_{B_d(x_{p,l})}G(y,x_{p,j})u_{p}^p(y)\,dy
+\int_{\Omega\backslash  \bigcup^k_{l=1}B_{d}(x_{p,l})}G(y,x_{p,j})
u_{p}^{p}(y)\,dy\\=&
 \sum^k_{l=1}\int_{B_d(x_{p,l})}G(y,x_{p,j})u_{p}^p(y)\,dy
+O\Big(\frac{C^p}{p^p}\Big).
\end{split}
\end{equation}
Scaling and using the properties of Green's function, we know
 \begin{equation}\label{s5-8-4}
 \begin{split}
  \int_{B_d(x_{p,j})}&G\big(y,x_{p,j}\big)u_{p}^p(y)\,dy
         \\=&u_p^p (x_{p,j}) \big(\varepsilon_{p,j}\big)^2
         \int_{B_{\frac{d}{\varepsilon_{p,j}}}(0)}G
         \big(x_{p,j},x_{p,j}+\varepsilon_{p,j}z\big)
         \Big(1+\frac{w_{p,j}(z)}{p}
  \Big)^pdz
         \\=&
         -\frac{u_p(x_{p,j})}{p} \int_{B_{\frac{d}{\varepsilon_{p,j}}}(0)}H\big(x_{p,j},x_{p,j}+
         \varepsilon_{p,j}z\big)\Big(1+\frac{w_{p,j}(z)}{p}
  \Big)^pdz
          \\&-\frac{u_p(x_{p,j})}{2\pi p}
  \int_{B_{\frac{d}{\varepsilon_{p,j}}}(0)}\log |z|\Big(1+\frac{w_{p,j}(z)}{p}
   \Big)^pdz
          \\&-\frac{u_p(x_{p,j})\log \varepsilon_{p,j}}{2\pi  p}
  \int_{B_{\frac{d}{\varepsilon_{p,j}}}(0)}\Big(1+\frac{w_{p,j}(z)}{p}
   \Big)^pdz.
 \end{split}\end{equation}
Now using Proposition \ref{Prop:5-12-52}, we find
\begin{equation}\label{s6-4-11}
\begin{split}
 \int_{B_{\frac{d}{\varepsilon_{p,j}}}(0)}& H(x_{p,j},x_{p,j}+
\varepsilon_{p,j}z)\Big(1+\frac{w_{p,j}(z)}{p}
 \Big)^pdz\\=&
  H(x_{p,j},x_{p,j})\int_{B_{\frac{d}{\varepsilon_{p,j}}}(0)} \Big(1+\frac{w_{p,j}(z)}{p}
  \Big)^pdz
  +O\Big(\varepsilon_{p,j}\Big) \\
  =&8\pi\big(1-\frac{3}{p}\big)  H(x_{p,j},x_{p,j})+o\Big(\frac{1}{p}\Big)=8\pi  H(x_{\infty,j},x_{\infty,j})+O\big(\frac{1}{p}\big).
  \end{split}
\end{equation}
Also by the classical dominated convergence theorem and Lemma \ref{llma}, we know
\begin{equation}\label{6-4-12}
\begin{split}
 \lim_{p\to \infty}&\int_{B_{\frac{d}{\varepsilon_{p,j}}}(0)}\log |z|\Big(1+\frac{w_{p,j} (z)}{p}
  \Big)^pdz\\=&\int_{\R^2} \log |z| e^{U(z)}
  =\int_{\R^2}\big(\log \frac{|z|}{\sqrt{8}}\big)\frac{1}{\big(1+\frac{|z|^2}{8}\big)^2}dz+12\pi\log 2
  \\=&8\int_{\R^2} \frac{\log  |y|}{\big(1+|y|^2\big)^2}dy+12\pi\log 2=12\pi\log 2,
\end{split}\end{equation}
here we use the following fact
\begin{equation*}
\int_{\R^2}\frac{\log  |y| }{\big(1+|y|^2\big)^2}dy
=2\pi \int^{+\infty}_{0}  \frac{r \log r }{\big(1+r^2\big)^2}dr\overset{r=\frac{1}{t}}{=}-2\pi \int^{+\infty}_0 \frac{t\log t}{\big(1+t^2\big)^2}dt.
\end{equation*}
Then from Proposition \ref{Prop:5-12-52}, \eqref{s5-8-4}, \eqref{s6-4-11} and  \eqref{6-4-12},  we know
 \begin{equation}\label{saa5-8-4}
 \begin{split}
\int_{B_d(x_{p,j})} &G(y,x_{p,j})u_{p}^p(y)\,dy\\=&
-\frac{ u_p(x_{p,j}) }{ p}\Big(
8\pi  H(x_{\infty,j},x_{\infty,j})+6\log 2
+O\big(\frac{1}{p}\big)\Big)  -\frac{4u_p(x_{p,j})\log \varepsilon_{p,j}}{p}\Big(1-\frac{3}{p} +O\big(\frac{1}{p^{2-\delta}}\big)\Big).
 \end{split}\end{equation}
 Next for $l\neq j$, similar to \eqref{s6-4-11}, it follows
  \begin{equation}\label{sab5-8-4}
 \begin{split}
\int_{B_d(x_{p,l})} &G(y,x_{p,j})u_{p}^p(y)\,dy\\=&
 \frac{ u_p(x_{p,l}) }{ p}\Big(8\pi G(x_{p,l},x_{p,j})  +O\big(\frac{1}{p}\big)\Big)
=
 \frac{ u_p(x_{p,l}) }{ p}\Big(8\pi G(x_{\infty,l},x_{\infty,j})  +O\big(\frac{1}{p}\big)\Big),
 \end{split}\end{equation}
 and  recalling that $u_p(x_{p,l})\rightarrow \sqrt{e}$ for $l=1,\cdots,k$, we know
 \begin{equation}\label{sac5-8-4}
 \frac{u_p(x_{p,l})}{u_p(x_{p,j})}\to 1~\mbox{for all}~l\neq j.
 \end{equation}
 Hence combining
 \eqref{sab5-8-4} and \eqref{sac5-8-4}, we get
 \begin{equation}\label{sad5-8-4}
 \begin{split}
\int_{B_d(x_{p,l})} &G(y,x_{p,j})u_{p}^p(y)\,dy
= \frac{ u_p(x_{p,j}) }{ p}\Big( 8\pi G(x_{\infty,l},x_{\infty,j})  +O\big(\frac{1}{p}\big)\Big).
 \end{split}\end{equation}
Then  from \eqref{s5-8-3}, \eqref{saa5-8-4} and \eqref{sad5-8-4}, we have
 \begin{equation}\label{5-7-51}
 \begin{split}
 \log \varepsilon_{p,j}=-
 \frac{p}{4} \left(\frac{1+\frac{1}{p}\Big(8\pi  \Psi_{k,j}(x_{\infty})+6\log 2\Big) +O\big(\frac {1}{p^2}\big)}{1-\frac{3}{p} +O\big(\frac{1}{p^{2-\delta}}\big)}\right) ~\,\mbox{with}\,\,x_{\infty}:=\big(x_{\infty,1},\cdots,x_{\infty,k}\big),
 \end{split}\end{equation}
 where $\Psi_{k,j}$ is the Kirchhoff-Routh function in \eqref{stts}.
Taking the definition of  $\varepsilon_{p,j}$ into \eqref{5-7-51}, we deduce that
\begin{equation*}
\begin{split}
 \log u_p(x_{p,j})= &
 \frac{p}{2(p-1)} \left(\frac{1+\frac{ 1}{p}\Big(8\pi \Psi_{k,j}(x_{\infty})+6 \log 2\Big)+O\big(\frac {1}{p^{2}}\big)}{1-\frac{3}{p} +O\big(\frac{1}{p^{2-\delta}}\big)}\right)-\frac{\log p}{p-1}\\=&
  \frac{p}{2(p-1)} \left( 1+\frac{ 1}{p}\Big(8\pi \Psi_{k,j}(x_{\infty})+6\log 2+3 \Big) \right)-\frac{\log p}{p-1} +O\big(\frac{1}{p^{2-\delta}}\big)\\=&
 \frac{1}{2}\left(1+\frac{1}{p}+O\big(\frac{1}{p^2}\big)\right)\cdot \left( 1+\frac{ 1}{p}\Big(8\pi \Psi_{k,j}(x_{\infty})+6\log 2+3 \Big) \right)-\frac{\log p}{p-1} +O\big(\frac{1}{p^{2-\delta}}\big)
  \\=&
  \frac{1}{2}+  \frac{ 1}{ p}\Big(4\pi \Psi_{k,j}(x_{\infty})+3\log 2+2 \Big) -\frac{\log p}{p-1} +O\big(\frac{1}{p^{2-\delta}}\big),
\end{split}
\end{equation*}
which implies \eqref{5-7-52}.

\vskip 0.1cm

%

Next we have
\[
\begin{split}
\e_{p,j}=&\Big(p\big(u_{p}(x_{p,j})\big)^{p-1}\Big)^{-1/2}= \frac1{\sqrt p} e^{ -\frac{p-1}2 \log u_{p}(x_{p,j})  }\\
=&\frac1{\sqrt p} e^{ -\frac{p-1}2 \Bigl( \frac{1}{2}+  \frac{ 1}{ p}\big(4\pi \Psi_{k,j}(x_{\infty})+3\log 2+2 \big) -\frac{\log p}{p-1} +O\big(\frac{1}{p^{2-\delta}}\big)\Bigr) }\\
=&   e^{-\frac{p-1}4}\cdot
 e^{ -\frac{p-1}{p} \Bigl( 2\pi \Psi_{k,j}(x_{\infty})+\frac{3\log 2}{2}+1  +O\big(\frac{1}{p^{1-\delta}}\big)\Bigr) }\\=&
 e^{-\frac{p}4}\Bigl(
 e^{-\big( 2\pi \Psi_{k,j}(x_{\infty})+\frac{3\log 2}{2}+\frac{3}{4}\big) }+O\big(\frac{1}{p^{1-\delta}}\big)\Bigr).
\end{split}
\]
So we have proved \eqref{nn3-29-03}. Clearly,
 \eqref{3-29-03} follows from \eqref{nn3-29-03}. Finally, \eqref{lst} and \eqref{lst1} follow by Proposition \ref{prop3-2}.
\end{proof}

\vskip 0.1cm

\subsection{Further expansions}\label{subsection:expansionu}
$\,$  \vskip 0.2cm

Thanks to Proposition \ref{lem3-1a}, we can improve the expansion of $u_p$ obtained in Lemma \ref{prop3-1}. This result will allow  to get the estimate \eqref{daluo-gil} in Proposition \ref{lem3-8}, which will be the key to get the uniqueness result in Section \ref{s3} (see Remark \ref{js} for more details).
\begin{Lem}\label{prop:expansionupwithdelta}
For any fixed small $d>0$, it holds
\begin{equation}\label{dluo-1}
u_p(x)= \sum^k_{j=1}C_{p,j}G(x_{p,j},x)+
O\Big(\sum^k_{j=1}\frac{\varepsilon_{p,j}}{p^{2-\delta}}\Big)\,\,
~\mbox{in}~C^1\Big(\Omega\backslash \displaystyle\bigcup^k_{j=1} B_{2d}(x_{p,j})\Big),
\end{equation}
where $\delta$ is a small fixed positive  constant and $C_{p,j}$ are the same constants in
\eqref{def_Cpj}.
\end{Lem}
\begin{proof}
We follow the scheme of the proof of Lemma \ref{prop3-1}, the main point is to improve \eqref{d5-8-14}. First,
since $\frac{w_{p,j}(x)}{p}=\frac{u_p(x_{p,j}+\e_{p,j} x)}{u_p(x_{p,j})}-1$, then $-1<\frac{w_{p,j}(x)}{p}\leq 0$. And using the fact that $\log(1+a)<a$ for any $a\in (-1,0)$ and Lemma \ref{llma}, we have
  \begin{equation}\label{dluo-a1}
  \begin{split}
 &\int_{B_\frac{d}{\varepsilon_{p,j}}(0)\backslash B_p(0)}\big\langle\nabla G(x_{p,j},x),z\big\rangle\Big(1+\frac{ w_{p,j}(z)}{p}
 \Big)^pdz \\=& \int_{B_\frac{d}{\varepsilon_{p,j}}(0)\backslash B_p(0)} \big\langle\nabla G(x_{p,j},x),z\big\rangle e^{ p\log \big(1+\frac{w_{p,j}(z)}{p}
  \big) }dz\leq
   \int_{B_\frac{d}{\varepsilon_{p,j}}(0)\backslash B_p(0)} |z|\cdot e^{w_{p,j}(z) }dz\\=& O\Big(
   \int_{B_\frac{d}{\varepsilon_{p,j}}(0)\backslash B_p(0)}\frac{|z|}{1+|z|^{4-\delta}}dz\Big)=O\Big(\frac{1}{p^{1-\delta}}\Big)~\,\mbox{for some}~\delta\in (0,1).
  \end{split}
  \end{equation}
Next similar to \eqref{adsst}, we have
\begin{equation}\label{dadsst}
\frac{|w_{p,j}|}{p},\,\frac{|w_{p,j}|^2}{p}\to 0~\,\mbox{uniformly in}\, B_p(0).
\end{equation}
Hence using that $w_{p,j}=U+\frac{v_{p,j}}{p}$ by definition \eqref{dsa} and Taylor's expansion, it follows
  \begin{equation*}
  \begin{split}
 \int_{ B_p(0)}& \big\langle\nabla G(x_{p,j},x),z\big\rangle \Big(1+\frac{w_{p,j}(z)}{p}
  \Big)^pdz\\ =& \int_{  B_p(0)} \big\langle\nabla G(x_{p,j},x),z\big\rangle e^{ p\log \big(1+\frac{w_{p,j}(z)}{p}
  \big) }dz\\=&
  \int_{ B_p(0)} \big\langle\nabla G(x_{p,j},x),z\big\rangle e^{w_{p,j}(z)}\Big(1-\frac{w^2_{p,j}(z)}{2p}+o\big(\frac{w^2_{p,j}(z)}{p}\big) \Big)\,dz \\=&
  \int_{ B_p(0)}\big\langle\nabla G(x_{p,j},x),z\big\rangle  e^{U(z)}\Big(1+\frac{v_{p,j}(z)}{p}+o\big(\frac{|v_{p,j}(z)|}{p}\big) \Big)
  \Big(1-\frac{w^2_{p,j}(z)}{2p}+o\big(\frac{w^2_{p,j}(z)}{p}\big) \Big)\,dz.
  \end{split}
  \end{equation*}
Then by Lemma \ref{llma}, Proposition \ref{lem3-1a}, Proposition \ref{prop3-2}, \eqref{dadsst} and the classical dominated convergence theorem, we can deduce that
  \begin{equation}\label{llsad}
  \begin{split}
\int_{ B_p(0)}& \big\langle\nabla G(x_{p,j},x),z\big\rangle \Big(1+\frac{w_{p,j}(z)}{p}
  \Big)^pdz=O
  \Big( \frac{1}{p}  \Big).
  \end{split}
  \end{equation}
Hence by using \eqref{dluo-a1} and \eqref{llsad},  the estimate \eqref{d5-8-14} can be improved into
\begin{equation*}
\begin{split}
 \int_{B_{\frac{d}{\varepsilon_{p,j}}}(0)}
\big\langle\nabla G(x_{p,j},x),z\big\rangle\Big(1+\frac{w_{p,j}(z)}{p}\Big)^pdz=O\Big(\frac{1}{p^{1-\delta}}\Big), ~~\mbox{uniformly in $\Omega\backslash  \displaystyle\bigcup^k_{j=1} B_{2d}(x_{p,j})$}.
\end{split}
\end{equation*}
And the rest is the same as the proof of Lemma \ref{prop3-1}.

\end{proof}
By exploiting the expansion \eqref{dluo-1} of $u_{p}$ we can get the following key relation:
\begin{Prop}\label{lem3-8}
For $j=1,\cdots,k$, there exists a small fixed constant $\delta\in (0,1)$ such that
\begin{equation}\label{daluo-gil}
C_{p,j} \frac{\partial R(x_{p,j})}{\partial{x_i}}-2\displaystyle\sum^k_{m=1,m\neq j}
C_{p,m}
D_{x_i} G(x_{p,m},x_{p,j})
=O\Big(\frac{\varepsilon_{p}}{p^{2-\delta}}\Big),
\end{equation}
where  {\Lp $D_{x_i}G(y,x):=\frac{\partial G(y,x)}{\partial x_i}$} and $C_{p,j}$ is defined in \eqref{def_Cpj}.
\end{Prop}
\begin{proof}
Letting  $u=u_p$  in the Pohozaev identity \eqref{aclp-1},  we have
\begin{equation}\label{afd}
Q_{j}(u_p,u_p)=\frac{2}{p+1}\int_{\partial B_{\theta}(x_{p,j})}u_{p}^{p+1}\nu_i.
\end{equation}
By the expansion of $u_{p}$ in Lemma \ref{prop:expansionupwithdelta} and recalling that $C_{p,j}=O\big(\frac{1}{p}\big)$ by \eqref{luoluo1}, it holds
\begin{equation}\label{5-6-1}
\mbox{LHS of (\ref{afd})}=\sum^k_{s=1}\sum^k_{m=1}
C_{p,s} C_{p,m} Q_{j}\Big( G(x_{p,s},x), G(x_{p,m},x)\Big)+O\Big(\frac{\varepsilon_{p}}{p^{3-\delta}}\Big).
\end{equation}
Moreover from \eqref{11-14-03N}, it follows
\begin{equation}\label{5-6-2}
\mbox{RHS of (\ref{afd})}
=O\Big( \frac{C^p}{p^{p+1}}  \Big)=O\Big(\frac{\varepsilon_{p}}{p^{3-\delta}}\Big).
\end{equation}
Then we find \eqref{daluo-gil} by \eqref{aluo1}, \eqref{afd}, \eqref{5-6-1} and \eqref{5-6-2}.
\end{proof}

\vskip 0.8cm

\section{Non-degeneracy of the positive solutions }\label{s2}
In this section, we prove Theorem \ref{th1.1} arguing by contradiction. Suppose that there exists a sequence  $p_m\rightarrow +\infty$, as $m\to +\infty$ satisfying
\begin{equation*}
\|\xi_{p_m}\|_{L^{\infty}(\Omega)}=1~~\mbox{and}~~\mathcal{L}_{p_m}\xi_{p_m}=0.
\end{equation*}
In order to shorten the notation, we replace $p_m$ by $p$.
\begin{Prop}\label{daprop3-2}
Let $\xi_{p,j}(x):=
\xi_{p}\big(\varepsilon_{p,j}x+x_{p,j}\big)$. Then by taking
a subsequence if necessary, we have
\begin{equation}\label{111}
 \xi_{p,j}(x)=a_j\frac{8-|x|^2}{8+|x|^2}+\sum_{i=1}^2\frac{b_{i,j}x_i}{8+|x|^2}+o(1)\,\,
 ~\mbox{in}~C^1_{loc}\big(\R^2\big)\,\,~\mbox{as}~p\rightarrow +\infty,
\end{equation}
 where $a_j$ and $b_{i,j}$ with $i=1,2$ and $j=1,\cdots,k$ are some constants.
\end{Prop}
\begin{proof}
In view of $\|\xi_{p,j}\|_{L^{\infty}(\R^2)}\leq 1$, by the regularity theory in \cite{GT1983}, it holds $$\xi_{p,j}\in C^{1,\alpha}_{loc}(\R^2)~\mbox{ and }~\|\xi_{p,j}\|_{C^{1,\alpha}_{loc}(\R^2)}\leq C~\mbox{for some}~\alpha \in (0,1),$$
 where $C$ is a constant independent of $j$ and $p$. So we may assume that $$\xi_{p,j}\rightarrow\xi_{j}~\mbox{in}~C^{1}_{loc}(\R^2).$$
Observe that $\xi_{p,j}$ solves
\begin{equation}\label{07-07-11}
-\Delta \xi_{p,j}=p\varepsilon_{p,j}^{2}u_{p}^{p-1}\big(\varepsilon_{p,j}x+x_{p,j}\big) \xi_{p,j} \,\,~\mbox{ in}~\frac{\Omega-x_{p,j}}{\xi_{p,j}}.
\end{equation}
Now from \eqref{5-8-2}, we find
\begin{equation*}
p\varepsilon_{p,j}^{2}u_{p}^{p-1}\big(\varepsilon_{p,j}x+x_{p,j}\big)  \rightarrow e^{U(x)}~\mbox{in}~C^2_{loc}(\R^2).
\end{equation*}
And then passing to the limit in \eqref{07-07-11}, we find that $\xi_j$ solves
\eqref{3-29-01},
which together with Lemma \ref{llm} implies \eqref{111}.
\end{proof}

\begin{Prop}\label{prop-2-2}
It holds
\begin{equation}\label{alaa1}
\begin{split}
\xi_{p}(x)=&\sum^k_{j=1}A_{p,j}G(x_{p,j},x)-8\pi
 \sum^k_{j=1}\sum^2_{i=1} b_{i,j}\varepsilon_{p,j}  \partial_iG(x_{p,j},x)
 +o\big( {\varepsilon}_{p}\big),
\end{split}\end{equation}
in $C^1
\Big(\Omega\backslash \bigcup^k_{j=1}B_{2d}(x_{p,j})\Big)$, where $d>0$ is any small fixed constant,
$\partial_iG(y,x):=\frac{\partial G(y,x)}{\partial y_i}$, ${\varepsilon}_{p}:=\max\big\{\varepsilon_{p,1},
\cdots,\varepsilon_{p,k}\big\}$,
\begin{equation}\label{aaa5-7-4}
\begin{split}
A_{p,j}:=p
\int_{B_d(x_{p,j})}  u_{p}^{p-1}\xi_{p}
\end{split}\end{equation}
and $b_{i,j}$ are the constants in Proposition \ref{daprop3-2}.
\end{Prop}
\begin{proof}
By the Green's representation theorem, we have
\begin{equation}\label{5-21-1}
\begin{split}
 {\xi}_p(x)=&  p \sum^k_{j=1}\int_{B_d(x_{p,j})} G(y,x)
 u_{p}^{p-1}(y) \xi_{p}(y)dy+p\int_{\Omega\backslash \bigcup^k_{j=1}B_{d}(x_{p,j})}G(y,x)
u_{p}^{p-1}(y) \xi_{p}(y)dy.
\end{split}
\end{equation}
Here for $y\in \Omega\backslash \bigcup^k_{j=1}B_{2d}(x_{p,j})$, from \eqref{11-14-03N} we know
$$
pu_{p}^{p-1}=O\Big(\frac{C^p}{p^{p-2}} \Big)=o\big( {\varepsilon}_{p}\big).$$
And it holds
\begin{equation*}
\int_{\Omega}G(x,y)dy\leq C ~\mbox{uniformly for}~x\in \Omega\backslash \bigcup^k_{j=1}B_{d}(x_{p,j}).
\end{equation*}
These give us that
\begin{equation}\label{agil44}
\begin{split}
 {\xi}_p(x)=p\sum^k_{j=1}\int_{B_d(x_{p,j})}  G(y,x)
 u_{p}^{p-1}(y) \xi_{p}(y)dy+o\big( {\varepsilon}_{p}\big)~\mbox{uniformly in}~\Omega\backslash \bigcup^k_{j=1}B_{d}(x_{p,j}).
\end{split}
\end{equation}
Next for $x\in \Omega\backslash \bigcup^k_{j=1}B_{2d}(x_{p,j})$, arguing similarly as in the proof  of Lemma \ref{prop3-1}, we get
\begin{equation*}
\begin{split}
p\int_{B_d(x_{p,j})}& G(y,x)
u_{p}^{p-1}(y) \xi_{p}(y)dy \\=&
 A_{p,j}G(x_{p,j},x)
 + \sum^2_{i=1}\varepsilon_{p,j} B_{p,j,i}\partial_iG(x_{p,j},x)
 +o\left( p\int_{B_d(x_{p,j})}\big|y-x_{p,j}\big|\cdot
u_{p}^{p-1}(y) \xi_{p}(y)dy\right)
\end{split}
\end{equation*}
uniformly in $\Omega\backslash \bigcup^k_{j=1}B_{d}(x_{p,j})$,
where $A_{p,j}$ is the term in \eqref{aaa5-7-4} and
$$ B_{p,j,i}:= \frac{p}{\varepsilon_{p,j}}\int_{B_d(x_{p,j})}(x_i-x_{p,j,i})u_{p}^{p-1}(x)\xi_{p}(x)dx.$$
Also using Lemma \ref{llma}, $|\xi_{p}|\leq 1$ and the dominated convergence theorem, we get
\begin{equation*}
\begin{split}
p\int_{B_d(x_{p,j})}&\big|y-x_{p,j}\big|\cdot
u_{p}^{p-1}(y) \xi_{p}(y)dy\\
\leq & p\int_{B_d(x_{p,j})}\big|y-x_{p,j}\big|\cdot
u_{p}^{p-1}(y) dy
=\e_{p,j}\int_{B_{\frac{d}{\e_{p,j}}(0)}}|x| \Big(1+\frac{w_{p,j}(x)}{p}\Big)^{p-1}dx
=O\big(\e_{p,j}\big).
\end{split}
\end{equation*}
Hence it holds
\begin{equation}\label{agil45}
\begin{split}
p\int_{B_d(x_{p,j})}& G(y,x)
u_{p}^{p-1}(y) \xi_{p}(y)dy =
 A_{p,j}G(x_{p,j},x)
 + \sum^2_{i=1}\varepsilon_{p,j} B_{p,j,i}\partial_iG(x_{p,j},x)+o\big(\varepsilon_{p,j} \big),
\end{split}
\end{equation}
uniformly for $x\in \Omega\backslash \bigcup^k_{j=1}B_{2d}(x_{p,j})$. Namely by \eqref{agil44},
 it follows
\begin{equation*}
{\xi}_p(x)=\sum^k_{j=1}A_{p,j}G(x_{p,j},x)
 +\sum^k_{j=1} \sum^2_{i=1}\varepsilon_{p,j} B_{p,j,i}\partial_iG(x_{p,j},x)+o\big(\varepsilon_{p,j} \big)~\mbox{uniformly in}~\Omega\backslash \bigcup^k_{j=1}B_{d}(x_{p,j}).
\end{equation*}
Now by scaling,  we know that
\begin{equation*}
\begin{split}
B_{p,j,i}=&\int_{B_{\frac{d}{\e_{p,j}}}(0)}x_i\big(1+\frac{w_{p,j}}{p}\big)^{p-1}\xi_{p,j}\, dx.
\end{split}\end{equation*}
Using Lemma \ref{llma}, we have  $\Big|x_i\big(1+\frac{w_{p,j}}{p}\big)^{p-1}\xi_{p,j}\Big|\leq \frac{C}{(1+|x|)^{3-\delta}}$, and then by  dominated convergence theorem, we obtain
\begin{equation}\label{a5-8-53}
\begin{split}
\lim_{p\to \infty}B_{p,j,i}= &\displaystyle\int_{\R^2}\left(x_ie^{U(x)}\Big( a_j\frac{\partial U(\frac{x}{\lambda})}{\partial \lambda}\big|_{\lambda=1}+\sum_{l=1}^2 b_{l,j}\frac{\partial U(x)}{\partial x_i}\Big)\right)
\\=&  b_{i,j} \Big(\int_{\R^2}
x_i\frac{\partial e^{U(x)}}{\partial x_i}dx\Big)
=
b_{i,j}\int_{\R^2}
x_i\frac{\partial e^{U(x)}}{\partial x_i}dx \\
=&-b_{i,j}\int_{\R^2} e^{U(x)} dx=
-8\pi b_{i,j}.
\end{split}\end{equation}
Then \eqref{agil44}, \eqref{agil45} and \eqref{a5-8-53} imply
\begin{equation*}\begin{split}
\xi_{p}(x)=&\sum^k_{j=1}A_{p,j}G(x_{p,j},x)-8\pi
 \sum^k_{j=1}\sum^2_{i=1} b_{i,j}\varepsilon_{p,j}  \partial_iG(x_{p,j},x)
 +o\big( {\varepsilon}_{p}\big)\,\,~\mbox{uniformly in}~
 \Omega\backslash \bigcup^k_{j=1}B_{2d}(x_{p,j}).
 \end{split}
\end{equation*}
On the other hand, from \eqref{agil44}, we know
\begin{equation*}
\begin{split}
\frac{\partial{\xi}_p(x)}{\partial x_i}=& p\int_{\Omega}D_{x_i}G(y,x)
 u_p^{p-1}(y) \xi_{p}(y)dy\,\,~\mbox{for}~i=1,2.
\end{split}
\end{equation*}
Then similar to the above computations for $\xi_{p}(x)$, we can  complete the proof of \eqref{alaa1}.
\end{proof}

The local Pohozaev identities \eqref{dafd} and \eqref{07-08-22} in Proposition \ref{prop:PohozaevLin} applied to $u_p$ and $\xi_p$ will be used, together with the properties of Green's function in Proposition \ref{lem2-1},
to show that in \eqref{111} and \eqref{alaa1}--\eqref{aaa5-7-4}  one has that $a_j=b_{i,j}=0$ and that $A_{p,j}=o(\e_{p,j})$ (see
propositions  \ref{dprop-luo1} and \ref{prop-gl} below).


\begin{Prop}\label{dprop-luo1}
Let $A_{p,j}$ be as in Proposition \ref{prop-2-2} and $a_{j}$ be the constants  in \eqref{111}.
Then it holds
\begin{equation}\label{dddluo-13}
  A_{p,j}=o\big( \e_p\big) ~\mbox{ and }~a_{j}=0\,\,~\mbox{for}~j=1,\cdots,k.
\end{equation}
\end{Prop}
\begin{proof}
We compute both sides of  the Pohozaev identity \eqref{07-08-22} with $\xi=\xi_{p}$ and $u=u_{p}$. Using the expansions of $u_{p}$ and $\xi_{p}$ in \eqref{luo-1} and \eqref{alaa1} respectively, together with
\eqref{1-1}, and \eqref{luoluo1}, we find
 \begin{equation}\label{dgil61}
\begin{split}
 \text{LHS of (\ref{07-08-22})}=& \sum^k_{m=1}\sum^k_{s=1}\big(A_{p,m}C_{p,s}\big)
P_{j}\Big(G(x_{p,s},x),G(x_{p,m},x)\Big)\\&-8\pi \e_{p,j}
\sum^k_{m=1}\sum^k_{s=1} \sum^2_{i=1}\big(b_{i,m}C_{p,s}\big)
P_{j}\Big(G(x_{p,s},x),\partial_iG(x_{p,m},x)\Big)\\&+
o\Big(\frac{ {\varepsilon}_{p}}{p} \Big)
 +
o\Big( {\varepsilon}_{p} \sum^k_{m=1} C_{p,m}\Big)\\=&
\frac{A_{p,j}C_{p,j}}{2\pi} -8\pi \e_{p,j}
\sum^k_{m=1}\sum^k_{s=1} \sum^2_{i=1}\big(b_{i,m}C_{p,s}\big)
P_{j}\Big(G(x_{p,s},x),\partial_iG(x_{p,m},x)\Big)+
o\Big(\frac{ {\varepsilon}_{p}}{p} \Big).
\end{split}
\end{equation}
Also using \eqref{abb1-1} and Proposition \ref{lem3-8}, we find
\begin{equation}\label{hfgil61}
\begin{split}
\sum^k_{m=1}& \sum^k_{s=1} \sum^2_{i=1}\big(b_{i,m}C_{p,s}\big)
P_{j}\Big(G(x_{p,s},x),\partial_iG(x_{p,m},x)\Big)\\
=&\frac{1}{2}
\sum^2_{i=1} \left( b_{i,j} \Big(C_{p,j} \frac{\partial R(x_{p,j})}{\partial{x_i}}-2\displaystyle\sum^k_{m=1,m\neq j}
C_{p,m}
D_{x_i} G(x_{p,m},x_{p,j})\Big)\right)=
o\Big(\frac{ {\varepsilon}_{p}}{p} \Big).
\end{split}
\end{equation}
Then from \eqref{dgil61} and \eqref{hfgil61}, we have
\begin{equation}\label{gil61}
\begin{split}
 \text{LHS of (\ref{07-08-22})}  =&
\frac{A_{p,j}C_{p,j}}{2\pi}  +
o\Big(\frac{ {\varepsilon}_{p}}{p} \Big).
\end{split}
\end{equation}

On the other hand, using \eqref{11-14-03N}, we know
\begin{equation*}
d\int_{\partial B_d(x_{p,j})}u^p_p\xi_p
\leq d \int_{\partial B_d(x_{p,j})}u^p_p=O\Big(\frac{C^p}{p^p} \Big),
\end{equation*}
and then
 \begin{equation}\label{adgil61}
\begin{split}
 \text{RHS of (\ref{07-08-22})}=&d\int_{\partial B_d(x_{p,j})}u^p_p\xi_p-2\int_{B_d(x_{p,j})} u^p_p\xi_p \\=&
-2\int_{B_d(x_{p,j})} u^p_p\xi_p
 +O\Big(\frac{C^p}{p^p} \Big).
\end{split}
\end{equation}
Also using  Lemma \ref{lem2-5}, it holds
 \begin{equation}\label{5-22-03}
\begin{split}
(p-1)  \int_{B_d(x_{p,j})} u^p_p\xi_p  =&
\int_{B_d(x_{p,j})}\Big( \Delta u_p\xi_p  -\Delta \xi_p u_p\Big)\\
=& \int_{\partial B_d(x_{p,j})}\Big( \frac{\partial u_p}{\partial \nu}\xi_p
-\frac{\partial \xi_p}{\partial \nu} u_p  \Big).
\end{split}
\end{equation}
Writing by \eqref{alaa1},
\begin{equation*}
\begin{split}
\xi_{p}(x)=&\sum^k_{j=1}A_{p,j}G(x_{p,j},x)+O\big(\varepsilon_{p}\big)\,\,~\mbox{in}~ C^1
\Big(\Omega\backslash \bigcup^k_{j=1}B_{2d}(x_{p,j})\Big),
\end{split}\end{equation*}
and using also the expansion of $u_{p}$ in \eqref{luo-1}, we then have
\begin{equation}\label{aad}
\begin{split}
\int_{\partial B_d(x_{p,j})} &\Big( \frac{\partial u_p}{\partial \nu}\xi_p
-\frac{\partial \xi_p}{\partial \nu} u_p   \Big)=O\Big(A_{p,j}C_{p,j}\Big)+
O\Big(\frac{ {\varepsilon}_{p}}{p}\Big).
\end{split}\end{equation}
Hence \eqref{adgil61}, \eqref{5-22-03} and \eqref{aad} give us that
 \begin{equation}\label{fgil61}
\begin{split}
 \text{RHS of (\ref{07-08-22})}=-2\int_{B_d(x_{p,j})} u^p_p\xi_p +O\Big(\frac{C^p}{p^p} \Big)
=
O\Big(\frac{A_{p,j}C_{p,j}}{p}\Big)+O\Big(\frac{ {\varepsilon}_{p}}{p^2}\Big).
\end{split}
\end{equation}
Then from \eqref{gil61}
and \eqref{fgil61}, we find
\begin{equation}\label{Adad}
\frac{A_{p,j}C_{p,j}}{2\pi}
 +
o\Big(\frac{ {\varepsilon}_{p}}{p}\Big)=
O\Big(\frac{A_{p,j}C_{p,j}}{p}\Big).
\end{equation}
Also from \eqref{luoluo1}, we know $
C_{p,j} =\frac{1}{p} \Big(8\pi \sqrt{e}+o(1)\Big)$, then from \eqref{Adad}, it follows
\begin{equation}\label{st5-8-52}
A_{p,j}=o\big( \e_p\big)~\mbox{for}~j=1,\cdots,k.
\end{equation}
Now by scaling,  we know
\begin{equation*}
\begin{split}
A_{p,j}=&\int_{B_{\frac{d}{\e_{p,j}}}(0)}\big(1+\frac{w_{p,j}}{p}\big)^{p-1}\xi_{p,j}.
\end{split}\end{equation*}
Using Lemma \ref{llma}, we have  $\Big|\big(1+\frac{w_{p,j}}{p}\big)^{p-1}\xi_{p,j}\Big|\leq \frac{C}{(1+|x|)^{4-\delta}}$, and then by \eqref{5-8-2}, \eqref{111}  and the   dominated convergence theorem, we obtain
\begin{equation}\label{sy5-8-52}
\begin{split}
\lim_{p\to \infty}A_{p,j}=&
\int_{\R^2}e^{U(x)}\left( a_j\frac{\partial U(\frac{x}{\lambda})}{\partial \lambda}\big|_{\lambda=1}+\sum_{i=1}^2 b_{i,j}\frac{\partial U(x)}{\partial x_i}\right)\\=&
  a_j  \Big(\frac{\partial}{\partial \lambda}\int_{\R^2} e^{U(\frac{x}{\lambda})} \Big)\Big|_{\lambda=1}
  =16\pi a_j.
\end{split}\end{equation}
Obviously \eqref{st5-8-52} and \eqref{sy5-8-52} imply that $a_j=0$ for $j=1,\cdots,k$.
\end{proof}
\begin{Rem}
In the previous proof \eqref{hfgil61} is derived by means of Proposition \ref{lem3-8}, which is obtained thanks to the improved expansion \eqref{dluo-1} of $u_{p}$. Let us stress that if one uses the expansion  \eqref{luo-1} of $u_{p}$
 instead of \eqref{dluo-1}, then Proposition \ref{lem3-8} becomes
 \[
C_{p,j} \frac{\partial R(x_{p,j})}{\partial{x_i}}-2\displaystyle\sum^k_{m=1,m\neq j}
C_{p,m}
D_{x_i} G(x_{p,m},x_{p,j})=o\big( \frac{\varepsilon_{p}}{p} \big),
\]
which is still enough to get \eqref{hfgil61}.
\end{Rem}
\begin{Prop}\label{prop-gl}
It holds
\begin{equation}\label{a3-29-02}
b_{i,j}=0,~\mbox{for}~i=1,2~\mbox{and}~j=1,\cdots,k,
\end{equation}
where $b_{i,j}$ are  constants in \eqref{111}.
\end{Prop}
\begin{proof}
We compute both sides of the Pohozaev identity \eqref{dafd} with $\xi=\xi_{p}$ and $u=u_{p}$. From the expansions of $u_{p}$ and $\xi_{p}$ in \eqref{luo-1} and \eqref{alaa1} respectively, we deduce
\begin{equation}\label{aluo21}
\begin{split}
 \text{LHS of (\ref{dafd})}=&
 Q_{j}\big(\xi_p,u_p\big)=\sum^k_{s=1}\sum^k_{m=1}\big(A_{p,s}C_{p,m}\big)
 Q_{j}\Big(G(x_{p,s},x),G(x_{p,m},x)\Big)
\\&-8\pi
\sum^k_{s=1}\sum^k_{m=1} \sum^2_{h=1} \varepsilon_{p,s} b_{h,s}C_{p,m} Q_{j}
\Big(G(x_{p,m},x),\partial_hG(x_{p,s},x)\Big)
+o\Big(\frac{ {\varepsilon}_p}{p}\Big).
\end{split}
\end{equation}
Furthermore \eqref{aluo1}, \eqref{luoluo1} and \eqref{dddluo-13} imply
\begin{equation}\label{a5-7-1}
\begin{split}
\sum^k_{s=1}\sum^k_{m=1}&\big(A_{p,s}C_{p,m}\big)
 Q_{j}\Big(G(x_{p,s},x),G(x_{p,m},x)\Big)=O\Big(\sum^k_{s=1}\sum^k_{m=1} \big|A_{p,s}C_{p,m}\big|
 \Big)=
o\Big(\frac{ {\varepsilon}_p}{p}\Big).
\end{split}
\end{equation}
Hence from \eqref{aluo21} and \eqref{a5-7-1}, we have
\begin{equation}\label{asluo21}
\begin{split}
 \text{LHS of (\ref{dafd})}= -8\pi
\sum^k_{s=1}\sum^k_{m=1} \sum^2_{h=1} \varepsilon_{p,s} b_{h,s}C_{p,m} Q_{j}
\Big(G(x_{p,m},x),\partial_hG(x_{p,s},x)\Big)
+o\Big(\frac{ {\varepsilon}_p}{p}\Big).
\end{split}
\end{equation}
Moreover from \eqref{11-14-03N}, it follows
\begin{equation}\label{aluo22}
\begin{split}
 \text{RHS of (\ref{dafd})}= &\frac{1}{p+1} \int_{\partial B_d(x_{p,j})}u^p_p \xi_p \nu_i =
O\Big( \frac{C^p}{p^{p+1}}  \Big)=o\Big(\frac{\varepsilon_{p}}{p}\Big).
\end{split}
\end{equation}
So from \eqref{asluo21} and \eqref{aluo22}, we have
\begin{equation} \label{a5-7-3}
\begin{split}
\sum^k_{s=1}\sum^k_{m=1} \sum^2_{h=1}\varepsilon_{p,s}\big(b_{h,s}C_{p,m}\big)  Q_{j}
\Big(G(x_{p,m},x),\partial_hG(x_{p,s},x)\Big)
=o\Big(\frac{ {\varepsilon}_p}{p}\Big).
\end{split}
\end{equation}
 Putting \eqref{aluo41} into \eqref{a5-7-3}, we obtain
 \begin{equation}\label{07-08-25}
 \begin{split}
 \e_{p,j}&C_{p,j}\sum^2_{h=1}\frac{\partial^2R(x_{p,j})}{\partial x_i\partial x_h}b_{h,j}
 -\sum^2_{h=1}\sum_{s\neq j}\e_{p,s}b_{h,s}C_{p,j}D_{x_i}\partial_hG(x_{p,s},x_{p,j})
\\& -\sum^2_{h=1}\sum_{s\neq j}\e_{p,j}b_{h,j}C_{p,s}D^2_{x_ix_h}G(x_{p,s},x_{p,j})=o\Big(\frac{ {\varepsilon}_p}{p}\Big).
 \end{split}
 \end{equation}
Letting $p\to \infty$ in \eqref{07-08-25}, together with \eqref{luoluo1}, we have
\begin{equation}\label{a7-30-2}
\begin{split}
\sum^2_{h=1} b_{h,j}&\Big(\frac{\partial^2 R(x_{\infty,j})}{\partial{x_ix_h}}-\sum_{s\neq j}
 \partial^2_{x_i  x_h}G(x_{\infty,j},x_{\infty,s}) \Big)
-
\sum^2_{h=1}\sum_{s\neq j} b_{h,s}
D_{x_h}\partial_{x_i}G(x_{\infty,j},x_{\infty,s})=0.
\end{split}
\end{equation}
Since $x_{\infty}:=(x_{\infty,1},\cdots,x_{\infty,k})$
is the nondegenerate critical point of $\Psi_{k}(x)$, we   obtain
\eqref{a3-29-02} from \eqref{a7-30-2}.
\end{proof}

  \vskip 0.1cm

\begin{proof}[\underline{\textbf{Proof of Theorem \ref{th1.1}}}]
Suppose that $\mathcal{L}_p \xi_p=0$ and $\xi_p\not\equiv0$.
Let $x_p$ be a maximum point of $\xi_p$ in
$\Omega$. We can assume that $\xi_p(x_p)=1$.
By propositions \ref{daprop3-2}, \ref{dprop-luo1} and \ref{prop-gl}, we have, for any $R>0$,
\begin{equation*}
\|\xi_{p,j}\|_{L^{\infty}\big(B_R(0)\big)}=o(1)\,\,~\mbox{for}~j=1,\cdots,k.
\end{equation*}
Therefore, $x_p\in \Omega\backslash \displaystyle\bigcup^k_{j=1}B_{R\varepsilon_{p,j}}(x_{p,j})$, for any $R>0$. In particular,
\begin{equation}\label{lsas}
\frac{|x_p-x_{p,j}|}{\e_{p,j}}\to +\infty.
\end{equation}
Let us write  \begin{equation*}\begin{split}
\Omega\backslash \displaystyle\bigcup^k_{j=1}B_{R\varepsilon_{p,j}}(x_{p,j})=&
\Big(\Omega\backslash \bigcup^k_{j=1}B_{d}(x_{p,j})\Big) \bigcup
 \Big( \bigcup^k_{j=1}  \bigl( B_{d}(x_{p,j})\backslash  B_{2p\varepsilon_{p,j}}(x_{p,j})\bigr)\Big)\\&
  \bigcup  \Big( \bigcup^k_{j=1}   \bigl(B_{2p\varepsilon_{p,j}}(x_{p,j})\backslash  B_{R\varepsilon_{p,j}}(x_{p,j})\bigr)\Big).
  \end{split}\end{equation*}
Now we divide the proof into following three steps.

\vskip 0.1cm

\noindent \textbf{Step 1. We show that $x_p\notin \Omega\backslash \displaystyle\bigcup^k_{j=1}B_{d}(x_{p,j})$.}

We just need to prove that
\begin{equation}\label{hhs}
\xi_{p}=o(1),~~\mbox{uniformly in}~~\Omega\backslash \displaystyle\bigcup^k_{j=1}B_{d}(x_{p,j}).
\end{equation}
By Proposition \ref{prop-2-2} and $b_{i,j}=0$ in Proposition \ref{prop-gl}, we have
\begin{equation*}
\begin{split}
 {\xi}_p(x)=& \sum^k_{j=1} A_{p,j}
  G(x_{p,j},x) + o(\e_p),~\mbox{
  uniformly in $\Omega\backslash \displaystyle\bigcup^k_{j=1}B_{d}(x_{p,j})$}.
\end{split}
\end{equation*}
Hence  \eqref{hhs} follows since  $A_{p,j}=o(\e_p)$
by Proposition \ref{dprop-luo1}, and observing that
\begin{equation}\label{lsst}
 \sup_{\Omega\backslash {\displaystyle\bigcup^k_{j=1}}
 	B_{R\varepsilon_{p,j}}(x_{p,j})}G(x_{p,j},x) =
  \sup_{\Omega\backslash {\displaystyle\bigcup_{j=1}^{k}} B_{R\varepsilon_{p,j}}(x_{p,j})}\left(-\frac{1}{2\pi}\log |x-x_{p,j}|-H(x,x_{p,j})\right)
 =O\Big(\big|\log \e_{p,j}\big|\Big).
\end{equation}

\noindent \textbf{Step  2. We show that $x_p\notin  \displaystyle\bigcup^k_{j=1} \bigl(  B_{d}(x_{p,j})\backslash  B_{ 2p\varepsilon_{p,j} }(x_{p,j})\bigr)$.}

We claim that
\begin{equation}\label{ddc}
{\xi}_p(x)=o\big(1\big)~~\mbox{for}~x\in  \displaystyle\bigcup^k_{j=1} B_{d}(x_{p,j})\backslash  B_{ 2p\varepsilon_{p,j}}(x_{p,j}).
\end{equation}
By the Green's representation theorem, \eqref{5-21-1} and similarly to the proof of Proposition \ref{prop-2-2}, we get
\begin{equation}\label{luos1}
\begin{split}
 {\xi}_p(x){=}& p \sum^k_{j=1}\int_{B_{d}(x_{p,j})} G(y,x)
 u_{p}^{p-1}(y) \xi_{p}(y)dy+p\int_{\Omega\backslash \bigcup^k_{j=1}B_{d}(x_{p,j})}G(y,x)
u_{p}^{p-1}(y) \xi_{p}(y)dy\\=&   p \sum^k_{j=1}\int_{ B_{ p\varepsilon_{p,j}}(x_{p,j})} G(y,x)
 u_{p}^{p-1}(y) \xi_{p}(y)dy\\&
 +p\sum^k_{j=1} \int_{B_{d}(x_{p,j})\backslash  B_{ p\varepsilon_{p,j}}(x_{p,j})}G(y,x)
u_{p}^{p-1}(y) \xi_{p}(y)dy+ o(\e_p).
\end{split}
\end{equation}
By Taylor's expansion and Lemma \ref{llma}, we find
\begin{equation}\label{tta}
\begin{split}
p  &\int_{ B_{ p\varepsilon_{p,j}}(x_{p,j})} G(y,x)
 u_{p}^{p-1}(y) \xi_{p}(y)dy \\=&
p \int_{ B_{ p\varepsilon_{p,j}}(x_{p,j})}  G(x_{p,j},x)
 u_{p}^{p-1}(y) \xi_{p}(y)dy\\&
 +O\left(p
 \int_{ B_{ p\varepsilon_{p,j}}(x_{p,j})} |y-x_{p,j} |\cdot |\nabla G(\xi,x)|
 u_{p}^{p-1}(y) \xi_{p}(y)dy  \right)\\=&
 p\int_{ B_{ p\varepsilon_{p,j} }(x_{p,j})}  G(x_{p,j},x)
 u_{p}^{p-1}(y) \xi_{p}(y)dy
 +O\left(\frac{1}{p} \int_{ B_{p}(0)} |y|\cdot \Big(1+\frac{w_{p,j}}{p}\Big)^{p-1}  dy  \right)\\=&
 p G(x_{p,j},x) \int_{ B_{ p\varepsilon_{p,j} }(x_{p,j})}
 u_{p}^{p-1}(y) \xi_{p}(y)dy
 +O\Big(\frac{1}{p}\Big),
~\,\,\mbox{where $\xi$ is between $x_{p,j}$ and $y$.}
\end{split}
\end{equation}
Moreover, recalling the definition of  $A_{p,j}$ in \eqref{aaa5-7-4}, the definition of $w_{p,j}$ in \eqref{defwpj} and Lemma \ref{llma},
for $x\in  \displaystyle\bigcup^k_{j=1} B_{d}(x_{p,j})\backslash  B_{ 2p\varepsilon_{p,j}}(x_{p,j})$, we have
\begin{equation}\label{tta1}
\begin{split}
 p & G(x_{p,j},x) \int_{ B_{p \varepsilon_{p,j}}(x_{p,j})}
 u_{p}^{p-1}(y) \xi_{p}(y)dy  \\=&
  G(x_{p,j},x) A_{p,j}-
 p G(x_{p,j},x) \int_{ B_d(x_{p,j})\backslash B_{p\varepsilon_{p,j}}(x_{p,j})}
 u_{p}^{p-1}(y) \xi_{p}(y)dy \\=&
  G(x_{p,j},x) A_{p,j}+O\Big(\log (p\e_{p,j})\Big)  \int^{ \frac{d}{\e_{p,j}}}_p\frac{1}{r^{3-\delta}}dr\\=&
  G(x_{p,j},x) A_{p,j}+O\Big(\frac{1}{p^{1-\delta}}\Big).
\end{split}
\end{equation}
Hence \eqref{tta} and \eqref{tta1} imply that for $x\in  \displaystyle\bigcup^k_{j=1} B_{d}(x_{p,j})\backslash  B_{ 2p\varepsilon_{p,j}}(x_{p,j})$, it holds
\begin{equation}\label{luos1t}
\begin{split}
 p \sum^k_{j=1}\int_{ B_{ p\varepsilon_{p,j}}(x_{p,j})} G(y,x)
 u_{p}^{p-1}(y) \xi_{p}(y)dy=
  \sum^k_{j=1}G(x_{p,j},x) A_{p,j}+O\Big(\frac{1}{p^{1-\delta}}\Big).
\end{split}
\end{equation}

Now we estimate the second term of the right part of \eqref{luos1}.
\begin{equation}\label{luos1t12}
\begin{split}
p&\int_{B_{d}(x_{p,j})\backslash  B_{ p\varepsilon_{p,j} }(x_{p,j})}G(y,x)
u_{p}^{p-1}(y) \xi_{p}(y)dy\\=&
O\left(\int_{B_{\frac{d}{\varepsilon_{p,j}}}(0)\backslash  B_{p}(0)}G(
\e_{p,j}y+x_{p,j},x)
 \Big(1+\frac{w_{p,j}(y)}{p}\Big)^{p-1}   dy\right)\\=&
O\left( \int_{B_{\frac{d}{\varepsilon_{p,j}}}(0)\backslash  B_{p}(0)} \frac{|\log \e_{p,j}|+  \big|\log |y+\frac{x_{p,j}-x}{\e_{p,j}}|\big|  }{(1+|y|)^{4-\delta}} dy\right)\\=&
O\left( \int_{B_{\frac{d}{\varepsilon_{p,j}}}(0)\backslash  B_{p}(0)} \frac{\big|\log |y+\frac{x_{p,j}-x}{\e_{p,j}}|\big|  }{(1+|y|)^{4-\delta}} dy\right)+O\left(
 \frac{|\log \e_{p,j}|}{ p^{2-\delta}}\right).
\end{split}
\end{equation}
Also by H\"older's inequality, \eqref{nn3-29-03} and $|\log \e_{p,j}|=O\big(p\big)$ by definition,  we have
\begin{equation}\label{luos1t1}
\begin{split}
 & \int_{B_{\frac{d}{\varepsilon_{p,j}}}(0)\backslash  B_{p}(0)} \frac{ \big|\log |y+\frac{x_{p,j}-x}{\e_{p,j}}|\big|  }{(1+|y|)^{4-\delta}} dy \\=&
O\left(
\Big(
 \int_{B_{\frac{d}{\varepsilon_{p,j}}}(0)\backslash  B_{p}(0)}  \big|\log |y+\frac{x_{p,j}-x}{\e_{p,j}}|\big|^{p}  dy \Big)^{\frac{1}{p}} \cdot \Big(
 \int_{B_{\frac{d}{\varepsilon_{p,j}}}(0)\backslash  B_{p}(0)} \frac{1 }{(1+|y|)^{\frac{(4-\delta)p}{p-1}}}dy \Big)^{\frac{p-1}{p}}  \right) \\&
 \\=&
 O\left( \frac{ \e_{p,j}^{-\frac{2}{p}} \cdot |\log \e_{p,j}| }{ p^{\big((4-\delta)\frac{p}{p-1}-2\big)\frac{p-1}{p}}} \right)=
O\left( \frac{1}{p^{1-\delta+\frac{2}{p}}}\right).
\end{split}
\end{equation}
Hence from the fact that $|\log \e_{p,j}|=O\big(p\big)$, \eqref{luos1}, \eqref{luos1t}, \eqref{luos1t12} and \eqref{luos1t1}, we get
\begin{equation*}
\begin{split}
 {\xi}_p(x)=&  \sum^k_{j=1} G(x_{p,j},x) A_{p,j}
 + O\left( \frac{1}{p^{1-\delta}}\right),~\mbox{uniformly in}~x\in  \displaystyle\bigcup^k_{j=1}   B_{d}(x_{p,j})\backslash  B_{ 2p\varepsilon_{p,j}}(x_{p,j}).
\end{split}
\end{equation*}
Then the proof of \eqref{ddc} follows from \eqref{lsst} and recalling that
$A_{p,j}=o\big( \e_p\big)$ as in \eqref{st5-8-52}.

\noindent \textbf{Step  3.} We now have $x_p\in \displaystyle\bigcup^k_{i=1}   B_{2p\varepsilon_{p,i}}(x_{p,i})\backslash  B_{R\varepsilon_{p,i}}(x_{p,i})$.
Suppose that there exists $j\in \{1,\cdots,k\}$ such that
$$x_p\in B_{2p\varepsilon_{p,j}}(x_{p,j})\backslash  B_{R\varepsilon_{p,j}}(x_{p,j})~~\mbox{and}~~r_p:=|x_p|.$$ Then $\frac{r_p}{\e_{p,j}}\geq R\gg 1$.
By translation, we assume that  $x_{p,j}=0$. Taking
\begin{equation}\label{llts}
\widetilde{\xi}_p(y):=\xi_p(r_py),
\end{equation}
then using Lemma  \ref{llma} and \eqref{lsas}, for any $y\in \Omega_{r_p}:=\{x,\,r_px\in\Omega\}$, we find

\begin{equation*}
\begin{split}
-\Delta \widetilde{\xi}_p(y)=&r^2_p pu^{p-1}(r_py)\widetilde{\xi}_p(y)=
\big(\frac{r_p}{\e_{p,j}}\big)^2\Big(1+\frac{w_{p,j}(\frac{r_py}{\e_{p,j}})}{p}\Big)^{p-1}\\
=&\big(\frac{r_p}{\e_{p,j}}\big)^2 e^{(p-1)\log \big(1+\frac{w_{p,j}(\frac{r_py}{\e_{p,j}})}{p}\big)}
\\
\leq & C \big(\frac{r_p}{\e_{p,j}}\big)^2 e^{\frac{p-1}p w_{p,j}(\frac{r_py}{\e_{p,j}})}\le C \big(\frac{r_p}{\e_{p,j}}\big)^2\frac{1}{\big(1+|\frac{r_py}{\e_{p,j}}|\big)^{4-\delta}}\rightarrow 0.
\end{split}\end{equation*}
Then there exists a bounded function $\xi$ such that
$$\widetilde{\xi}_p \rightarrow \xi~\mbox{in}~C\big(B_R(0)\backslash B_\delta(0)\big)
~\mbox{and}~\Delta \xi =0~~\mbox{in}~\R^2,$$
which implies $\xi=\widetilde{C}$.  From  $\widetilde{\xi}_p(\frac{x_p}{r_p})=\xi_p(x_p)=1$, we see $\widetilde{C}=1$, namely
\begin{equation}\label{altts}
\widetilde{\xi}_p \rightarrow 1~\mbox{in}~C\big(B_R(0)\backslash B_\delta(0)\big).
\end{equation}

Now let $\xi_{p,j}(y):=\xi_p(\e_{p,j}y+x_{p,j})$, then we have
\begin{equation*}
-\Delta \xi_{p,j}-e^{U}\xi_{p,j}=f_p:=\Big(\big(1+\frac{w_{p,j}}{p}\big)^{p-1}-e^{U}\Big)
\xi_{p,j}.
\end{equation*}
Similarly as in
\eqref{5-7-31}, \eqref{St} and \eqref{5-7-32},
 we can verify that
\begin{equation*}
|f_p(y)|\leq \frac{1}{p\big(1+|y|\big)^{4-\delta-\tau/2}}.
\end{equation*}
The average of $\xi_{p,j}$, denoted by
${\xi}^*_{p,j}(r):=\displaystyle\int^{2\pi}_0 {\xi}_{p,j}(r,\theta)d\theta$, solves the ODE
\begin{equation*}
-\left({\xi}^*_{p,j}\right)''-\frac{1}{r}\left({\xi}^*_{p,j}\right)'-e^{U}{\xi}^*_{p,j}=\int^{2\pi}_0f_p(r,\theta)d\theta:=f^*_p(r).
\end{equation*}
Observe that, since $u_0(r):=\frac{8-r^2}{8+r^2}$ is a bounded solution of
\begin{equation}\label{jjsts}
-v''-\frac{1}{r}v'-e^{U(r)}v=0,
\end{equation}
then by ODE's theory, the second solution of \eqref{jjsts} is given by
\begin{equation*}
u_1(r):=\frac{(8-r^2)\log r+16}{r^2+8},
\end{equation*}
and  the function
\begin{equation*}
V(r):=u_0(r)\int^r_0su_1(s)f_p^*(s)ds-u_1(r)\int^r_0su_0(s)f_p^*(s)ds
\end{equation*}
is a solution of $-v''-\frac{1}{r}v'-e^{U(r)}v=f_p^*(r)$.
As a consequence there exist two constants $C_{0,p}$, $C_{1,p}\in\mathbb R$ such that
\begin{equation*}
\xi^*_{p,j}(r)=C_{0,p}u_0(r)+C_{1,p}u_1(r)+V(r).
\end{equation*}
Since $\xi^*_{p,j}$ must be bounded at $r=0$ and $V(0)=0$, then we have that $C_{1,p}=0$ and computing at $0$ we get
$C_{0,p}=o(1)$.
Then we find
\begin{equation*}
\begin{split}
|\xi^*_{p,j}(r)|=&|V(r)|+o(1)\\ \leq &\frac{C}{p}\int^r_0 \frac{s\log s}{(1+s)^{4-\delta-\frac{\tau}{2}}}ds+
\frac{C\log r}{p}\int^r_0  \frac{s}{(1+s)^{4-\delta-\frac{\tau}{2}}}ds+o(1)\\ \leq &\frac{C\log r}{p}+o(1).
\end{split}\end{equation*}
Last we evaluate $\xi^*_{p,j}$ at $\frac{r_p}{\e_{p,j}}$.
Indeed we have that $\frac{r_p}{\e_{p,j}}\leq 2p$ and
\begin{equation}\label{Stsa}
 \xi^*_p\big(\frac{r_p}{\e_{p,j}}\big)\leq \frac{C\log \frac{r_p}{\e_{p,j}}}{p}+o(1)\to 0.
\end{equation}

However  by \eqref{altts}, we know
\begin{equation}\label{Stsat}
\widetilde{\xi}_p(x)\to 1,~~\mbox{for any}~~|x|=1.
\end{equation}
Then by \eqref{llts} and \eqref{Stsat}, we find
\begin{equation*}
\begin{split}
 \xi^*_{p,j}(\frac{r_p}{\e_{p,j}}) =&
 \int^{2\pi}_0 \xi_{p,j}\big(\frac{r_p}{\e_{p,j}},\theta\big) d\theta =
 \int^{2\pi}_0 \xi_{p}\big(r_p,\theta\big) d\theta
 \\=&\int^{2\pi}_0\widetilde{\xi}_p\big(1,\theta\big)d\theta  \geq c_0>0,~\mbox{for some constant}~c_0,
\end{split}\end{equation*}
which is a contraction to \eqref{Stsa}.
This ends the proof of  Theorem \ref{th1.1}.

\end{proof}

\section{Local uniqueness of the $1$-spike solutions }\label{s3}

Let $u_p^{(1)}$ and $u_p^{(2)}$ be two positive solutions to \eqref{1.1} with
\begin{equation*}
\lim_{p\rightarrow +\infty} p \int_{\Omega}|\nabla u^{(l)}_{p}(x)|^2dx=8\pi e,~\,~\mbox{for}\,~l=1,2.
\end{equation*}
From Theorem A in the introduction, we know that they both concentrate at a unique point, which is a critical point of the Robin function. Let us assume that they
 concentrate at the same point $x_{\infty,1}$ and that $x_{\infty,1}$ is a non-degenerate critical point  of Robin function $R(x)$.
Let us set   $x^{(l)}_{p,1}$ for $l=1,2$, the points such that
 \begin{equation*}
u^{(l)}_{p}(x^{(l)}_{p,1})=\max_{\overline{B_{2r}(x_{\infty,1})}}u^{(l)}_{p}(x)\,\,~\mbox{for some small fixed}~r>0.
\end{equation*}
And let us also define
$$C^{(l)}_{p,1}:= \displaystyle\int_{B_d(x^{(l)}_{p,1})}
\big(u^{(l)}_{p}(y)\big)^{p}dy\quad\mbox{ and }\quad \varepsilon^{(l)}_{p,1}:=\Big(p\big(u^{(l)}_{p}(x^{(l)}_{p,1})\big)^{p-1}\Big)^{-1/2}, ~~\mbox{for} ~~l=1,2.$$
\begin{Lem}\label{assd}
There exists some small fixed constant $\delta>0$ such that
\begin{equation*}\label{llsb}
\frac{\e^{(1)}_{p,1}}{\e^{(2)}_{p,1}}=1+O\big(\frac{1}{p^{1-\delta}}\big).
\end{equation*}
\end{Lem}
\begin{proof} It follows from \eqref{nn3-29-03}.
\end{proof}

\begin{Lem}\label{ssag}
If $x_{\infty,1}$ is a nondegenerate critical point of $R(x)$, then it holds
\begin{equation}\label{aaaa}
\Big|x^{(1)}_{p,1}-x^{(2)}_{p,1}\Big|=O\Big(\frac{\widetilde{\varepsilon}_{p}}{p^{1-\delta}}\Big),
\end{equation}
where $\widetilde{\varepsilon}_{p}=\max\big\{\varepsilon^{(1)}_{p,1}, \varepsilon^{(2)}_{p,1}\big\}$
and $\delta$ is a small fixed positive constant.
\end{Lem}
\begin{proof}
First,  Proposition \ref{lem3-8} with $k=1$ becomes
\begin{equation}\label{aluo-gil}
  \frac{\partial R(x^{(l)}_{p,1})}{\partial{x_i}}
=O\Big(\frac{\varepsilon^{(l)}_{p,1} }{p^{1-\delta}}\Big)  ~\mbox{for}~l,i=1,2.
\end{equation}
Since $x_{\infty,1}$ is a nondegenerate critical point of $R(x)$,
it holds
\begin{equation*}
\nabla R\big(x^{(l)}_{p,1}\big)=\nabla^2R(\zeta_p)\cdot \big(x^{(l)}_{p,1}-x_{\infty,1}\big),~~\mbox{for}~
\zeta_p~~\mbox{between}~x_{p,1}~\mbox{and}~x_{\infty,1},
\end{equation*}
which, together with \eqref{aluo-gil}, gives
\begin{equation}\label{lsls}
\big|x^{(l)}_{p,1}-x_{\infty,1}\big| =O\Big(\frac{\varepsilon^{(l)}_{p,1} }{p^{1-\delta}}\Big).
\end{equation}
So, we have proved \eqref{aaaa}.
\end{proof}

In the following, we will consider the same quadric forms already introduced in \eqref{07-08-20} and \eqref{abd}. In particular, thanks to the previous two lemmas, and since
 $k=1$ in all this section, it will be enough to consider only the two following quadric forms
\begin{equation}\label{5.3B}
\begin{split}
P^{(1)}_1(u,v):=&- 2\theta\int_{\partial B_\theta(x^{(1)}_{p,1})}
\big\langle \nabla u ,\nu\big\rangle
\big\langle \nabla v,\nu\big\rangle
+  \theta  \int_{\partial B_\theta(x^{(1)}_{p,1})}
\big\langle \nabla u , \nabla v \big\rangle,
\end{split}
\end{equation}
and
\begin{equation}\label{5.3BB}
Q^{(1)}_1(u,v):=- \int_{\partial B_\theta(x^{(1)}_{p,1})}\frac{\partial v}{\partial \nu}\frac{\partial u}{\partial x_i}-
 \int_{\partial B_\theta(x^{(1)}_{p,1})}\frac{\partial u}{\partial \nu}\frac{\partial v}{\partial x_i}
+ \int_{\partial B_\theta(x^{(1)}_{p,1})}\big\langle \nabla u,\nabla v \big\rangle \nu_i.
\end{equation}
Note that if $u$ and $v$ are harmonic in $ B_d(x^{(1)}_{p,1})\backslash \{x^{(1)}_{p,1}\}$, then similarly as in Lemma \ref{llas}, we deduce that $P_1^{(1)}(u,v)$ and  $Q^{(1)}_1(u,v)$ are independent of $\theta\in (0,d]$.
 Moreover, replacing $x_{p,1}$ by $x^{(1)}_{p,1}$ in Proposition \ref{lem2-1}, we have the following computations.
\begin{Prop}
It holds
\begin{equation*}
P_1^{(1)}\Big(G(x^{(1)}_{p,1},x), G(x^{(1)}_{p,1},x)\Big)=
-\frac{1}{2\pi},
\end{equation*}
 \begin{equation}\label{bb1-1}
P_1^{(1)}\Big(G(x^{(1)}_{p,1},x),\partial_hG(x^{(1)}_{p,1},x)\Big)=
\frac{1}{2}\frac{\partial R(x^{(1)}_{p,1})}{\partial h},
\end{equation}

\begin{equation}\label{luo1}
Q^{(1)}_1\Big(G(x^{(1)}_{p,1},x),G(x^{(1)}_{p,1},x)\Big)=
-\frac{\partial R(x^{(1)}_{p,1})}{\partial{x_i}},
\end{equation}
and
\begin{equation}\label{luo41}
Q^{(1)}_1\Big(G(x^{(1)}_{p,1},x),\partial_h G(x^{(1)}_{p,1},x)\Big)=
- \frac{\partial^2 R(x^{(1)}_{p,1})}{\partial{x_ix_h}},
\end{equation}
where $G$ and $R$ are the Green and Robin functions respectively (see \eqref{greensyst}, \eqref{GreenS-H}, \eqref{Robinf}), $\partial_hG(y,x):=\frac{\partial G(y,x)}{\partial y_h}$.
\end{Prop}
Furthermore, we can also show the following results.
\begin{Prop}
For some small fixed $\delta>0$, it holds
\begin{equation}\label{sbb1-1}
P_1^{(1)}\Big(G(x^{(2)}_{p,1},x), G(x^{(1)}_{p,1},x)\Big)=
-\frac{1}{2\pi}+O\Big(\frac{\widetilde{\varepsilon}_{p} }{p^{1-\delta}}\Big),
\end{equation}
 \begin{equation}\label{dbb1-1}
P_1^{(1)}\Big(G(x^{(2)}_{p,1},x),\partial_hG(x^{(1)}_{p,1},x)\Big)=
\frac{1}{2}\frac{\partial R(x_{\infty,1})}{\partial h}+O\Big(\frac{\widetilde{\varepsilon}_{p} }{p^{1-\delta}}\Big),
\end{equation}

\begin{equation}\label{dluo1}
Q^{(1)}_1\Big(G(x^{(2)}_{p,1},x),G(x^{(1)}_{p,1},x)\Big)=
-\frac{\partial R(x_{\infty,1})}{\partial{x_i}}+O\Big(\frac{\widetilde{\varepsilon}_{p} }{p^{1-\delta}}\Big),
\end{equation}
and
\begin{equation}\label{dluo41}
Q^{(1)}_1\Big(G(x^{(2)}_{p,1},x),\partial_h G(x^{(1)}_{p,1},x)\Big)=
- \frac{\partial^2 R(x_{\infty,1})}{\partial{x_ix_h}}+O\Big(\frac{\widetilde{\varepsilon}_{p} }{p^{1-\delta}}\Big),
\end{equation}
where $G$ and $R$ are the Green and Robin functions respectively (see \eqref{greensyst}, \eqref{GreenS-H}, \eqref{Robinf}).
\end{Prop}
\begin{proof}
First  similar to the proof of \eqref{1-1} in Lemma \ref{lem2-1}, we find
\begin{equation}\label{tgil21}
\begin{split}
P_{1}^{(1)}\Big(G(x^{(2)}_{p,1},x),G(x^{(1)}_{p,1},x)\Big)=
 P_{1}^{(1)}\Big(S(x^{(2)}_{p,1},x),S(x^{(1)}_{p,1},x)\Big).
 \end{split}
\end{equation}
Also using Taylor's expansion, we can get
\begin{equation}\label{ytgil21}
\begin{split}
 P_{1}^{(1)}\Big(S(x^{(2)}_{p,1},x),S(x^{(1)}_{p,1},x)\Big) =&
 P_{1}^{(1)}\Big(S(x^{(1)}_{p,1},x),S(x^{(1)}_{p,1},x)\Big)+O\Big(\big|x^{(2)}_{p,1}-x^{(1)}_{p,1}\big|\Big).
 \end{split}
\end{equation}
Hence \eqref{sbb1-1} follows by \eqref{1-1}, \eqref{aaaa}, \eqref{lsls}, \eqref{tgil21} and \eqref{ytgil21}.
\vskip 0.2cm
Similarly by similar processes, we can obtain \eqref{dbb1-1}--\eqref{dluo41}.
\end{proof}

\vskip 0.4cm

Now if $u_{p}^{(1)}\not \equiv u_{p}^{(2)}$ in $\Omega$, we set
\begin{equation}\label{a3.1}
\eta_{p}:=\frac{u_{p}^{(1)}-u_{p}^{(2)}}
{\|u_{p}^{(1)}-u_{p}^{(2)}\|_{L^{\infty}(\Omega)}},
\end{equation}
then $\eta_{p}$ satisfies $\|\eta_{p}\|_{L^{\infty}(\Omega)}=1$ and
\begin{equation}\label{aaad}
- \Delta \eta_{p}=-\frac{\Delta u_{p}^{(1)}-\Delta u_{p}^{(2)}}
{\|u_{p}^{(1)}-u_{p}^{(2)}\|_{L^{\infty}(\Omega)}}=\frac{\big(u_{p}^{(1)}\big)^p
-\big(u_{p}^{(2)}\big)^p}
{\|u_{p}^{(1)}-u_{p}^{(2)}\|_{L^{\infty}(\Omega)}}=D_{p,1}\eta_{p},
\end{equation}
where
\begin{equation}\label{5.14BIS}
D_{p,1}(x):=p\displaystyle\int_{0}^1
\Big(tu_{p}^{(1)}(x)+(1-t)u_{p}^{(2)}(x)\Big)
^{p-1}dt.
\end{equation}

\begin{Lem}\label{ttst}
It holds
\begin{equation}\label{asd0}
\big(\varepsilon^{(1)}_{p,1}\big)^{2}D_{p,1}\big(\varepsilon^{(1)}_{p,1}x+x^{(1)}_{p,1}\big)  =e^{U(x)} +O\Big(\frac{e^{U(x)}}{p^{1-\delta}}\Big),
\end{equation}
uniformly on compact sets, in particular
\begin{equation}\label{asd}
\big(\varepsilon^{(1)}_{p,1}\big)^{2}D_{p,1}\big(\varepsilon^{(1)}_{p,1}x+x^{(1)}_{p,1}\big)  \rightarrow e^{U(x)}~\, ~\mbox{in}~C_{loc}\big(\R^2\big).
\end{equation}
Also for some small fixed $\gamma,d>0$, it follows
\begin{equation}\label{5.15BIS}
\big(\varepsilon^{(1)}_{p,1}\big)^{2}D_{p,1}\big(\varepsilon^{(1)}_{p,1}x+x^{(1)}_{p,1}\big) =O\Big(\frac{1}{1+|x|^{4-\gamma}}\Big)~\,\,~\mbox{for}~~ |x|\leq \frac{d}{\e^{(1)}_{p,1}}.
\end{equation}
\end{Lem}
\begin{proof}
First we know
\begin{equation}\label{adv}
 \big(\e^{(1)}_{p,1}\big)^2D_{p,1}\big(\e^{(1)}_{p,1}y+x^{(1)}_{p,1}\big)
= p\big(\e^{(1)}_{p,1}\big)^2\int^1_0\big(F_t(y)\big)^{p-1}dt,
\end{equation}
where
\begin{equation*}
 F_t(y):= tu^{(1)}_p\big(\e^{(1)}_{p,1}y+x^{(1)}_{p,1}\big) +(1-t)u^{(2)}_p\big(\e^{(1)}_{p,1}y+x^{(1)}_{p,1}\big).
\end{equation*}
Then we can write $F_t(y)$ as follows
\begin{equation*}
 F_t(y)= u^{(1)}_p\big( x^{(1)}_{p,1}\big) \left(  t \big(1+ \frac{w^{(1)}_{p,1}(y)}{p}\big) +(1-t)\frac{u^{(2)}_p\big( x^{(2)}_{p,1}\big)}{u^{(1)}_p\big( x^{(1)}_{p,1}\big)}
 \Big(1+\frac{1}{p}w^{(2)}_{p,1}\big(\frac{\e^{(2)}_{p,1}}{\e^{(1)}_{p,1}}y+
 \frac{x^{(2)}_{p,1}-x^{(1)}_{p,1}}{\e^{(1)}_{p,1}}\big)\Big)\right).
\end{equation*}
where $w^{(1)}_{p,1}$ and $w^{(2)}_{p,1}$ are defined as in \eqref{defwpj}, taking the solutions $u^{(1)}_p$ and $u^{(2)}_p$ in place of $u_{p}$ respectively.

Let $a(t):=t+(1-t)\frac{u^{(2)}_p\big( x^{(2)}_{p,1}\big)}{u^{(1)}_p\big( x^{(1)}_{p,1}\big)}-1$, then we can write
\begin{equation*}
\begin{split}
 F_t(y)= & u^{(1)}_p\big( x^{(1)}_{p,1}\big) \left(1+a(t)+  \frac{1}{p} \Big( tw^{(1)}_{p,1}(y) +(1-t)\frac{u^{(2)}_p\big( x^{(2)}_{p,1}\big)}{u^{(1)}_p\big( x^{(1)}_{p,1}\big)}
 w^{(2)}_{p,1}\big(\frac{\e^{(2)}_{p,1}}{\e^{(1)}_{p,1}}y+\frac{x^{(2)}_{p,1}-x^{(1)}_{p,1}}{\e^{(1)}_{p,1}}\big)
 \Big)\right)\\=&
 u^{(1)}_p\big( x^{(1)}_{p,1}\big)\Big(1+a(t)\Big)
 \left(1+  \frac{b_t(y)}{p\big(1+a(t)\big)} \right),
 \end{split}
\end{equation*}
where $$b_t(y):= tw^{(1)}_{p,1}(y) +(1-t)\frac{u^{(2)}_p\big( x^{(2)}_{p,1}\big)}{u^{(1)}_p\big( x^{(1)}_{p,1}\big)}
 w^{(2)}_{p,1}\Big(\frac{\e^{(2)}_{p,1}}{\e^{(1)}_{p,1}}y+\frac{x^{(2)}_{p,1}-x^{(1)}_{p,1}}{\e^{(1)}_{p,1}}\Big).$$
Then it holds
\begin{equation}\label{hhls}
\begin{split}
p\big(\e^{(1)}_{p,1}\big)^2\int^1_0\big(F_t(y)\big)^{p-1}dt=&\int^1_0\Big(1+a(t)\Big)^{p-1}
 \left(1+  \frac{b_t(y)}{p\big(1+a(t)\big)} \right)^{p-1}dt\\=&
 \int^1_0\Big(1+a(t)\Big)^{p-1}
e^{(p-1)\log  \left(1+  \frac{b_t(y)}{p\big(1+a(t)\big)} \right)}dt
\\=&
 \int^1_0\Big(1+a(t)\Big)^{p-2}e^{w^{(1)}_{p,1}}
\left( 1+b_t(y)-w^{(1)}_{p,1}(y) +O\Big(\frac{1}{p}\Big)\right)dt.
 \end{split}
\end{equation}
Also we recall that from Proposition \ref{prop3-2},
\begin{equation}\label{hls1}
\begin{split}
 w^{(l)}_{p,1}(y)=U(y)+O\Big(\frac{1}{p}\Big),~~\mbox{for}~l=1,2,
 \end{split}
\end{equation}
uniformly in any compact set, from \eqref{5-7-52},
\begin{equation}\label{hls2}
\begin{split}
\frac{u^{(2)}_p\big( x^{(2)}_{p,1}\big)}{u^{(1)}_p\big( x^{(1)}_{p,1}\big)}=1+O\Big(\frac{1}{p^{2-\delta}}\Big),
 \end{split}
\end{equation}
and from Lemma \ref{ssag},
\begin{equation}\label{hls3}
\begin{split}
\frac{\big|x^{(1)}_{p,1}-x^{(2)}_{p,1}\big|}{\widetilde{\varepsilon}_{p}}=O\Big(\frac{1}{p^{1-\delta}}\Big).
 \end{split}
\end{equation}
Then using \eqref{hls1} \eqref{hls2} and \eqref{hls3}, we have
\begin{equation*}
 a(t)=O\Big(\frac{1}{p^{2-\delta}}\Big),
 \end{equation*}
and \begin{equation*}\quad b_t(y)-w^{(1)}_{p,1}(y) = O\Big(\frac{1}{p^{1-\delta}}\Big)~\,~\mbox{uniformly in any compact set}.
\end{equation*}
Finally, we obtain
\begin{equation*}
\begin{split}
p&\big(\e^{(1)}_{p,1}\big)^2\int^1_0\big(F_t(y)\big)^{p-1}dt\\=&
 \int^1_0\Big(1+a(t)\Big)^{p-2}e^{U(y)}
\left( 1+(t-1)U(y) +(1-t)U\Big(\frac{\e^{(2)}_{p,1}}{\e^{(1)}_{p,1}}y\Big)\right)dt+O\Big(\frac{e^{U(y)}}{p^{1-\delta}}\Big)\\=&
e^{U(y)} +O\Big(\frac{e^{U(y)}}{p^{1-\delta}}\Big),
 \end{split}
\end{equation*}
which implies \eqref{asd0}.

\vskip 0.1cm

 On the other hand, using Lemma \ref{llma}, we have
 \begin{equation}\label{hka}
 1+b_t(y)-w^{(1)}_{p,1}(y)=O\Big(\log |y|\Big)~\,\,~\mbox{for}~~R\leq |y|\leq \frac{d}{\e^{(1)}_{p,1}}.
 \end{equation}
Hence from \eqref{asd}, \eqref{hhls} and  \eqref{hka}, we find
\begin{equation*}
\begin{split}
p\big(\e^{(1)}_{p,1}\big)^2\int^1_0\big(F_t(y)\big)^{p-1}dt= O\Big(\frac{1}{1+|y|^{4-\gamma}}\Big)~\,\,~\mbox{for}~~|y|\leq \frac{d}{\e^{(1)}_{p,1}},
 \end{split}
\end{equation*}
which, together with \eqref{adv}, gives \eqref{5.15BIS}.
\end{proof}
\begin{Prop}\label{aprop3-2}
Let $\eta_{p,1}(x):=\eta_{p}\big(\varepsilon^{(1)}_{p,1}x+x^{(1)}_{p,1}\big)$, where $\eta_{p}$ is defined in \eqref{a3.1}. Then by taking
a subsequence if necessary, we have
\begin{equation}\label{a111}
 \eta_{p,1}(x)=\widetilde{a}_1\frac{8-|x|^2}{8+|x|^2}+\sum_{i=1}^2\frac{\widetilde{b}_{i,1}x_i}{8+|x|^2}+o(1)
 ~\mbox{in}~C^1_{loc}\big(\R^2\big), ~\mbox{as}~p\rightarrow +\infty,
\end{equation}
 where $\widetilde{a}_1$ and $\widetilde{b}_{i,1}$ with $i=1,2$ are some constants.

\end{Prop}
\begin{proof}
Since $|\eta_{p,1}|\leq 1$, then we can suppose that
$$\eta_{p,1}\to \eta_0~~\mbox{in}~C^1_{loc}\big(\R^2\big).$$
Since, by \eqref{aaad}, $\eta_{p,1}$ solves $$-\Delta \eta_{p,1} =\big(\e^{(1)}_{p,1}\big)^2D_{p,1}\big(\e^{(1)}_{p,1}+x^{(1)}_{p,1} x\big)\eta_{p,1},$$
by Lemma \ref{ttst}, we find that $\eta_0$ is a solution of \eqref{3-29-01}. Hence
\eqref{a111} holds by Lemma \ref{llm}.
\end{proof}
\begin{Prop}
Let $\eta_{p}$ be defined as  in  \eqref{a3.1}, then it holds
\begin{equation}\label{laa1}
\begin{split}
\eta_{p}(x)=& \widetilde{A}_{p}G(x^{(1)}_{p,1},x)+
\sum^2_{i=1}\widetilde{B}_{p,i}\varepsilon^{(1)}_{p,1}  \partial_iG(x^{(1)}_{p,1},x)
 +o\Big(\widetilde{\varepsilon}_{p}\Big) \,~\mbox{in}~ C^1
\Big(\Omega\backslash  B_{2d}(x^{(1)}_{p,1})\Big),
\end{split}\end{equation}
where $d>0$ is any small fixed constant, $\widetilde{\varepsilon}_{p}=\max\big\{\varepsilon^{(1)}_{p,1}, \varepsilon^{(2)}_{p,1}\big\}$,
$\partial_iG(y,x)=\frac{\partial G(y,x)}{\partial y_i}$,
\begin{equation}\label{5.17B}
\widetilde{A}_{p}:=
\int_{B_d\big(x^{(1)}_{p,1}\big)}  D_{p,1}(x)\eta_{p}(x)dx
\end{equation}
and
\begin{equation}\label{5.17BB}
\widetilde{B}_{p,i}:= \frac{1}{\varepsilon^{(1)}_{p,1}}\int_{B_d\big(x^{(1)}_{p,1}\big)}(x_i-(x^{(1)}_{p,1})_{i})  D_{p,1}(x)\eta_{p}(x)dx,
\end{equation}
with $D_{p,1}$ as in \eqref{5.14BIS}.
\end{Prop}
\begin{proof}
The proof is similar to the one of Proposition \ref{prop-2-2}. Indeed,
replacing $pu_{p}^{p-1}(y)$ by $D_{p,1}(y)$ in \eqref{5-21-1}, then similarly as in \eqref{agil45} (with $k=1$), we find
\begin{equation*}
\begin{split}
\eta_{p}(x)=&\widetilde{A}_{p}G(x^{(1)}_{p,1},x)+
\sum^2_{i=1} \widetilde{B}_{p,i}\varepsilon^{(1)}_{p,1}  \partial_iG(x^{(1)}_{p,1},x)
 +o\Big(\widetilde{\varepsilon}_{p}\Big)\,\,~\mbox{uniformly in}~
 \Omega\backslash  B_{2d}(x^{(1)}_{p,1}).
 \end{split}
\end{equation*}
The convergence of $\frac{\partial{\eta}_p}{\partial x_i}$  follows in a similar way, since
\begin{equation*}
\begin{split}
\frac{\partial{\eta}_p(x)}{\partial x_i}=& \int_{\Omega}D_{x_i} G(y,x)
 D_{p,1}(y) \eta_{p}(y)dy.
\end{split}
\end{equation*}
\end{proof}

In order to show that the constants in Proposition \ref{aprop3-2} satisfy $\widetilde{a}_1=\widetilde{b}_{1,1}=\widetilde{b}_{2,1}=0$, following the same strategy  exploited in Section \ref{s2}, we will need the following local Pohozaev identities, which are a direct consequence of \eqref{aclp-1} and \eqref{aclp-10}.
\begin{Prop}
For $\eta_{p}$ defined by \eqref{a3.1} and small fixed $\theta>0$, we have the following local Pohozaev identities:
 \begin{equation}\label{aa2}
 \begin{split}
Q^{(1)}_1\big(u_p^{(1)}+u_p^{(2)},\eta_p\big)=\frac{2}{p+1}
 \int_{\partial  B_{\theta}(x^{(1)}_{p,1})}\widetilde{D}_{p,1} \eta_p \nu_i\,d\sigma,
\end{split}
\end{equation}
and
\begin{equation}\label{aaclp1}
\begin{split}
P^{(1)}_1\big(u_p^{(1)}+u_p^{(2)},\eta_p\big)
= \frac{2\theta}{p+1}\int_{\partial  B_{\theta}(x^{(1)}_{p,1})} \widetilde{D}_{p,1}\eta_p\,d\sigma
-\frac{4}{p+1}\int_{  B_{\theta}(x^{(1)}_{p,1})} \widetilde{D}_{p,1}\eta_p\,dx,
\end{split}
\end{equation}
where $P_1^{(1)}$ and $Q^{(1)}_1$ are the quadric forms in \eqref{5.3B} and \eqref{5.3BB},   $\nu=\big(\nu_{1},\nu_2\big)$ is the outward unit normal of $\partial  B_{\theta}(x^{(1)}_{p,1})$ and
\begin{equation}\label{5.19B}
\widetilde{D}_{p,1}(x):=(p+1)\int_{0}^1\Big(tu_{p}^{(1)}(x)+(1-t)u_{p}^{(2)}(x)\Big)^{p}
dt.
\end{equation}
\end{Prop}
\begin{proof}
Taking $u=u_{p}^{(l)}$ with $l=1,2$
 in \eqref{aclp-1} (with $j=1$), then it reduces to
 \begin{equation*}
Q^{(1)}_{1}(u_{p}^{(l)}, u_{p}^{(l)})=
\frac{2}{p+1}\int_{\partial B_{\theta}(x^{(1)}_{p,1})}  \big(u_{p}^{(l)}\big)^{p+1} \nu_i.
\end{equation*}
Hence \eqref{aa2} follows by making the difference.
Similarly taking   $u=u_{p}^{(l)}$ with $l=1,2$ in \eqref{aclp-10} (with $j=1$) and repeating the above process, we can deduce \eqref{aaclp1}.
\end{proof}
\begin{Prop}\label{adprop-luo1}
Let $\widetilde{A}_p$ be as in \eqref{5.17B} and $\widetilde{a}_1$ be the constant in \eqref{a111}, then
\begin{equation*}
\widetilde{A}_{p}=
o\big(\frac{1}{p}\big)~\,\mbox{and}~\,\widetilde{a}_1=0.
\end{equation*}
\end{Prop}
\begin{proof}
The proof is similar to the one of Proposition \ref{dprop-luo1}. Indeed using \eqref{luo-1},  \eqref{luoluo1}, \eqref{sbb1-1} and \eqref{laa1}, we can deduce that
\begin{equation}\label{sagil61}
\begin{split}
 \text{LHS of (\ref{aaclp1})}=&\frac{\widetilde{A}_{p}}{2}
\sum^2_{l=1}C^{(l)}_{p,1}
P_1^{(1)}\Big(G(x^{(l)}_{p,1},x),G(x^{(1)}_{p,1},x)\Big) \\&
+  \e^{(1)}_{p,1} \sum^2_{l=1} \sum^2_{i=1}\big(\widetilde{B}_{p,i}C^{(l)}_{p,1}\big)
P_1^{(1)}\Big(G(x^{(l)}_{p,1},x),\partial_iG(x^{(1)}_{p,1},x)\Big) +
o\Big(\frac{\widetilde{\varepsilon}_{p}}{p} \Big) \\=&
\widetilde{A}_{p} \sum^2_{l=1} \frac{C^{(l)}_{p,1}}{4\pi} + \e^{(1)}_{p,1}
 \sum^2_{i=1}\sum^2_{l=1} \big(\widetilde{B}_{p,i}C^{(l)}_{p,1}\big)
P_1^{(1)}\Big(G(x^{(l)}_{p,1},x),\partial_iG(x^{(1)}_{p,1},x)\Big)+
o\Big(\frac{\widetilde{\varepsilon}_{p}}{p}\Big).
\end{split}
\end{equation}
By \eqref{asd}, \eqref{5.15BIS}, the definition of $\widetilde{B}_{p,i}$ in \eqref{5.17BB} and the dominated convergence theorem, we know
\begin{equation*}
\begin{split}
|\widetilde{B}_{p,i}|=&\bigl|\big(\varepsilon^{(1)}_{p,1}\big)^2\int_{B_{\frac{d}{\e^{(1)}_{p,1}}(0)} } D_{p,1}\big(x^{(1)}_{p,1}+\varepsilon^{(1)}_{p,1}y\big)\eta_{p,1}(y)y_idy\bigr|\\
\leq &\big(\varepsilon^{(1)}_{p,1}\big)^2\int_{B_{\frac{d}{\e^{(1)}_{p,1}}(0)} } D_{p,1}\big(x^{(1)}_{p,1}+\varepsilon^{(1)}_{p,1}y\big)\cdot |y|dy\to
\int_{\R^2}e^{U(y)}\cdot |y|dy,
\end{split}\end{equation*}
which implies $\widetilde{B}_{p,i}=O\big(1\big)$.
Also using Proposition \ref{lem3-8} (which reduces to \eqref{aluo-gil} when $k=1$, as already observed),  \eqref{bb1-1} and \eqref{dbb1-1}, we find
\begin{equation}\label{sfgil61}
\begin{split}
 \sum^2_{i=1}&\sum^2_{l=1} \big(\widetilde{B}_{p,i}C^{(l)}_{p,1}\big)
P_1^{(1)}\Big(G(x^{(l)}_{p,1},x),\partial_iG(x^{(1)}_{p,1},x)\Big)\\=&
  2\sum^2_{i=1}\big(\widetilde{B}_{p,i}C^{(1)}_{p,1}\big)
P_1^{(1)}\Big(G(x^{(1)}_{p,1},x),\partial_iG(x^{(1)}_{p,1},x)\Big)+o\Big(\frac{ \widetilde{\varepsilon}_{p}}{p} \Big)=
o\Big(\frac{ \widetilde{\varepsilon}_{p}}{p} \Big).
\end{split}
\end{equation}
Then from \eqref{luoluo1}, \eqref{sagil61} and \eqref{sfgil61}, we have
\begin{equation}\label{tabgil61}
\begin{split}
 \text{LHS of (\ref{aaclp1})}  =&\Big(
\sum^2_{l=1} \frac{C^{(l)}_{p,1}}{4\pi}\Big)\widetilde{A}_{p}  +
o\Big(\frac{ \widetilde{\varepsilon}_{p}}{p} \Big)=
 \widetilde{A}_{p}\Big( \frac{4\sqrt{e}}{p}+o\big(\frac{1}{p}\big)\Big) +
o\Big(\frac{ \widetilde{\varepsilon}_{p}}{p} \Big).
\end{split}
\end{equation}

On the other hand,   using \eqref{11-14-03N} and similar to \eqref{sy5-8-52}, we know
 \begin{equation}\label{bgil61}
\begin{split}
 \text{RHS of (\ref{aaclp1})}=&
 \frac{d}{p+1}\int_{\partial B_d(x^{(1)}_{p,1})} \widetilde{D}_{p,1}\eta_p-\frac{2}{p+1}\int_{B_d(x^{(1)}_{p,1})} \widetilde{D}_{p,1}\eta_p \\=&
-\frac{2}{p+1}\int_{B_d(x^{(1)}_{p,1})} \widetilde{D}_{p,1}\eta_p +O\Big(\frac{C^p}{p^p} \Big).
\end{split}
\end{equation}
Also by the definitions of $\eta_p $ and $\widetilde{D}_{p,1}$ in \eqref{a3.1} and \eqref{5.19B}, it holds
 \begin{equation*}
\begin{split}
 \int_{B_d(x^{(1)}_{p,1})}& \widetilde{D}_{p,1}\eta_p
 \\=& \frac{1}{\|u_p^{(1)}-u_p^{(2)}\|_{L^{\infty}(\Omega)}}
 \int_{B_d(x^{(1)}_{p,1})}  \big(u_p^{(2)}\big)^p \big(u_p^{(1)}-u_p^{(2)}\big) \\&
 + \frac{1}{\|u_p^{(1)}-u_p^{(2)}\|_{L^{\infty}(\Omega)}}
 \int_{B_d(x^{(1)}_{p,1})}
 \Big(\big(u_p^{(1)}\big)^p -\big(u_p^{(2)}\big)^p \Big)u_p^{(1)}
  \\=&
 \int_{B_d(x^{(1)}_{p,1})} \eta_p   \big(u_p^{(2)}\big)^p -
 \int_{B_d(x^{(1)}_{p,1})} u_p^{(1)}  \Delta \eta_p ,
\end{split}
\end{equation*}
and by integration by parts, \eqref{luo-1} and \eqref{laa1}, we find
 \begin{equation*}
\begin{split}
 -\int_{B_d(x^{(1)}_{p,1})}&u^{(1)}_p   \Delta \eta_p
\\=&-\int_{\partial B_d(x^{(1)}_{p,1})} \frac{\partial \eta_p}{\partial \nu}u_p^{(1)}
+ \int_{ B_d(x^{(1)}_{p,j})} \nabla \eta_p\cdot \nabla u_p^{(1)}\\
=&-\int_{\partial B_d(x^{(1)}_{p,1})} \frac{\partial \eta_p}{\partial \nu}u_p^{(1)}
+ \int_{\partial B_d(x^{(1)}_{p,1})}\frac{\partial u_p^{(1)}}{\partial \nu} \eta_p
+ \int_{B_d(x^{(1)}_{p,1})} \big(u_p^{(1)}\big)^p \eta_p \\=&
 \int_{B_d(x^{(1)}_{p,1})} \eta_p  \big(u_p^{(1)}\big)^p
 +O\Big(\frac{\big|\widetilde{A}_{p}\big|}{p}+\frac{\widetilde{\varepsilon}_p}{p}\Big).
\end{split}
\end{equation*}
Hence from above estimates, we have
 \begin{equation}\label{s5-22-03}
\begin{split}
 \int_{B_d(x^{(1)}_{p,1})} \widetilde{D}_{p,1}\eta_p =
 \int_{B_d(x^{(1)}_{p,1})} \eta_p  \Big( \big(u_p^{(1)}\big)^p + \big(u_p^{(2)}\big)^p\Big)
 +O\Big(\frac{\big|\widetilde{A}_{p}\big|}{p} +\frac{\widetilde{\varepsilon}_p}{p}\Big).
\end{split}
\end{equation}
Also by \eqref{luoluo1}, it follows
 \begin{equation}\label{s5-22-03a}
\begin{split}
\int_{B_d(x^{(1)}_{p,1})}& \eta_p  \Big( \big(u_p^{(1)}\big)^p + \big(u_p^{(2)}\big)^p\Big)
\\=& O\left(  \int_{B_d(x^{(1)}_{p,1})} \Big( \big(u_p^{(1)}\big)^p + \big(u_p^{(2)}\big)^p\Big) \right)
=O\left(\frac{1}{p}\right).
\end{split}
\end{equation}
Then from \eqref{bgil61}, \eqref{s5-22-03} and \eqref{s5-22-03a},  we find
 \begin{equation}\label{s5-22-03t}
\begin{split}
 \text{RHS of (\ref{aaclp1})}= O\Big(\frac{\big|\widetilde{A}_{p}\big|}{p^2} +\frac{1}{p^2}\Big).
\end{split}
\end{equation}
Hence \eqref{tabgil61} and \eqref{s5-22-03t} give us that
\begin{equation}\label{a5-8-52sy}\widetilde{A}_{p}=O\Big(\frac{1}{p}\Big).\end{equation}

On the other hand,   using \eqref{asd} we know
\begin{equation}\label{a5-8-52}
\begin{split}
\widetilde{A}_{p}=&
\widetilde{a}_1p\big(u^{(1)}_p(x_{p,1})\big)^{p-1}\big(\varepsilon^{(1)}_{p,1}\big)^2\int_{\R^2}
e^{U(x)}\frac{\partial U(\frac{x}{\lambda})}{\partial \lambda}\big|_{\lambda=1}dx +o(1)\\
=&\widetilde{a}_1\int_{\R^2}
e^{U(x)}\frac{\partial U(\frac{x}{\lambda})}{\partial \lambda}\big|_{\lambda=1}dx+o(1)
=16\pi \widetilde{a}_1+o(1).
\end{split}\end{equation}
Hence from \eqref{a5-8-52sy} and \eqref{a5-8-52}, we have $\widetilde{a}_1=0$.

\vskip 0.2cm

Also we can compute \eqref{s5-22-03a} more precisely,
\begin{equation}\label{s5-22-03adt}
\begin{split}
\int_{B_d(x^{(1)}_{p,1})}& \eta_p (x) \Big( \big(u_p^{(1)}(x)\big)^p + \big(u_p^{(2)}(x)\big)^p\Big)\,dx
\\=& \frac{2}{p}\left(\int_{\R^2}e^{U(x)}\Big( \widetilde{a}_1\frac{\partial U(\frac{x}{\lambda})}{\partial \lambda}\big|_{\lambda=1}+\sum_{i=1}^2 \widetilde{b}_{i,1}\frac{\partial U(x)}{\partial x_i}\Big)\,dx+o\big(1\big)\right)
=o\left(\frac{1}{p}\right).
\end{split}
\end{equation}
Hence \eqref{tabgil61}, \eqref{bgil61}, \eqref{s5-22-03} and \eqref{s5-22-03adt} give us that $\widetilde{A}_{p}=o\Big(\frac{1}{p}\Big)$.
\end{proof}

\begin{Prop}\label{adprop-luo2}
Let $\widetilde{B}_{p,i}$ be as in \eqref{5.17BB} and $\widetilde{b}_{i,1}$ as in Proposition \ref{aprop3-2}, then
\begin{equation}\label{3-29-02}
\widetilde{B}_{p,i}=o(1)\quad\mbox{and}\quad\widetilde{b}_{i,1}=0~\mbox{for}~i=1,2.
\end{equation}
\end{Prop}
\begin{proof}
The proof is similar to the one of Proposition \ref{prop-gl}. Indeed
using \eqref{11-14-03N}, Lemma \ref{assd} and Lemma \ref{ssag}, it holds
\begin{equation}\label{luo22}
\begin{split}
 \text{RHS of (\ref{aa2})}= &\frac{2}{p+1} \int_{\partial B_d(x^{(1)}_{p,1})}\widetilde{D}_{p,1} \eta_p \nu_i =o\Big(\frac{\widetilde{\varepsilon}_p}{p}\Big).
\end{split}
\end{equation}
Moreover using \eqref{luo1}, \eqref{luo41},
\eqref{dluo1} and \eqref{dluo41}, we have
\begin{equation}\label{luo21}
\begin{split}
 \text{LHS of (\ref{aa2})}=& \widetilde{A}_{p} \sum^2_{l=1}C^{(l)}_{p,1}
 Q^{(1)}_1
 \Big(G(x^{(l)}_{p,1},x),G(x^{(1)}_{p,1},x)\Big)
\\&+  \e^{(1)}_{p,1} \sum^2_{l=1} \sum^2_{h=1}\big(\widetilde{B}_{p,h}C^{(l)}_{p,1}\big)  Q^{(1)}_1
\Big(G(x^{(l)}_{p,1},x),\partial_hG(x^{(1)}_{p,1},x)\Big)
+o\Big(\frac{\widetilde{\varepsilon}_p}{p}\Big).
\end{split}
\end{equation}
From Proposition \ref{lem3-8} (which reduces to \eqref{aluo-gil} when $k=1$, as already observed), \eqref{luo1} and \eqref{dluo1}, we know
\begin{equation}\label{5-7-1}
\begin{split}
 \sum^2_{l=1} C^{(l)}_{p,1}
 Q^{(1)}_1\Big(G(x^{(l)}_{p,1},x),G(x^{(1)}_{p,1},x)\Big) =
o\Big(\frac{\widetilde{\varepsilon}_p}{p}\Big).
\end{split}
\end{equation}
Then from Proposition \ref{adprop-luo1}, \eqref{luo22}, \eqref{luo21} and \eqref{5-7-1}, we have
\begin{equation} \label{5-7-3}
\begin{split}
 \e^{(1)}_{p,1} \sum^2_{l=1} \sum^2_{h=1}\big(\widetilde{B}_{p,h}C^{(l)}_{p,1}\big)  Q^{(1)}_1
\Big(G(x^{(l)}_{p,1},x),\partial_hG(x^{(1)}_{p,1},x)\Big)
=o\Big(\frac{\widetilde{\varepsilon}_p}{p}\Big).
\end{split}
\end{equation}
 Putting \eqref{luoluo1}, \eqref{luo41} and \eqref{dluo41} into \eqref{5-7-3},  we then have
\begin{equation}\label{7-30-2}
\begin{split}
\sum^2_{h=1} &\widetilde{B}_{p,h} \Big(\frac{\partial^2 R(x_{\infty,1})}{\partial{x_ix_h}}+o(1)\Big)
=o(1).
\end{split}
\end{equation}
Since $x_{\infty,1}$
is a nondegenerate critical point of $R(x)$, from \eqref{7-30-2},
we deduce
$$\widetilde{B}_{p,i}=o(1),~\mbox{for}~i=1,2.$$
Finally, similarly as in \eqref{a5-8-53}, it holds
\begin{equation}\label{5-8-53}
\begin{split}
\widetilde{B}_{p,i}= &
\widetilde{b}_{i,1}p\Big(u^{(1)}_p(x^{(1)}_{p,1})\Big)^{p-1}\Big(\varepsilon^{(1)}_{p,1}\Big)^2
\Big(\int_{\R^2}
x_ie^{U(x)}\frac{\partial U(x)}{\partial x_i}dx +o(1)\Big)\\=&-8\pi \widetilde{b}_{i,1}+o(1)\,\,~\mbox{for}~i=1,2.
\end{split}\end{equation}
Then  $\widetilde{b}_{1,1}=\widetilde{b}_{2,1}=0$  by \eqref{3-29-02} and \eqref{5-8-53}.
\end{proof}

 \vskip 0.1cm

\begin{proof}[\underline{\textbf{Proof of Theorem \ref{th1-1}}}]
Suppose $u^{(1)}_p\not\equiv u^{(2)}_p$, then let $\eta_p:=\frac{u^{(1)}_p- u^{(2)}_p}{
\|u^{(1)}_p-u^{(2)}_p\|_{L^\infty{(\Omega)}}}$, we have $\|\eta_p\|_{L^\infty{(\Omega)}}=1$.
Taking $\eta_{p,1}(x):=
\eta_{p}\big(\varepsilon^{(1)}_{p,1}x+x^{(1)}_{p,1}\big)$, by propositions \ref{aprop3-2}, \ref{adprop-luo1} and \ref{adprop-luo2}, we have, for any $R>0$,
\begin{equation*}
\|\eta_{p,1}\|_{L^{\infty}\big(B_R(0)\big)}=o(1).
\end{equation*}
Let $y_p$ be a maximum point of $\eta_p$ in
$\Omega$. We can assume that $\eta_p(y_p)=1$, then $y_p\in \Omega\backslash  B_{R\varepsilon^{(1)}_{p,1}}(x^{(1)}_{p,1})$.
Now we write   $$
\Omega\backslash  B_{R\varepsilon^{(1)}_{p,1}}(x^{(1)}_{p,1})=\Big(\Omega\backslash B_{d}(x^{(1)}_{p,1})\Big) \bigcup
 \Big(   B_{d}(x^{(1)}_{p,1})\backslash  B_{2p\varepsilon^{(1)}_{p,1}}(x^{(1)}_{p,1})\Big)
  \bigcup  \Big(  B_{2p\varepsilon^{(1)}_{p,1}}(x^{(1)}_{p,1})\backslash  B_{R\varepsilon^{(1)}_{p,1}}(x^{(1)}_{p,1})\Big),$$
and divide the proof into following three steps.

\vskip 0.2cm

\noindent \textbf{Step 1. We show that $y_p\not\in \Omega\backslash B_{d}(x^{(1)}_{p,1})$.}
It is enough to prove
\begin{equation}\label{ttsts}
 \eta_{p}=o(1)~~\mbox{uniformly in}~~\Omega\backslash B_{d}(x^{(1)}_{p,1}).
\end{equation}
By Green's representation theorem and similarly as in the proof of Proposition \ref{prop-2-2},
then from \eqref{laa1} and \eqref{3-29-02}, we can deduce that
\begin{equation*}
\begin{split}
 {\eta}_p(x)=&   \widetilde{A}_{p}
  G(x^{(1)}_{p,1},x) +o(\widetilde{\e}_p),~\,\mbox{uniformly in}~\Omega\backslash  B_{d}(x^{(1)}_{p,1}).
\end{split}
\end{equation*}
Hence \eqref{ttsts} follows since $\widetilde{A}_{p}=o\big(\frac{1}{p}\big)$
by Proposition \ref{adprop-luo1} and observing that
\begin{equation*}
 \sup_{\Omega\backslash  B_{d}(x^{(1)}_{p,1})}G(x^{(1)}_{p,1},x) =
  \sup_{\Omega\backslash  B_{d}(x^{(1)}_{p,1})}\left(-\frac{1}{2\pi}\log |x-x^{(1)}_{p,1}|-H(x,x^{(1)}_{p,1})\right)
 =O\Big(\log \e^{(1)}_{p,1}\Big).
\end{equation*}

\vskip 0.2cm

\noindent \textbf{Step 2. We show that $y_p \not\in  B_{d}(x^{(1)}_{p,1})\backslash  B_{2p\varepsilon^{(1)}_{p,1}}(x^{(1)}_{p,1})$.} We just need to prove
\begin{equation}\label{ttstsa}
\eta_{p}(x)=o(1) ~\mbox{uniformly in}~B_{d}(x^{(1)}_{p,1})\backslash  B_{2p\varepsilon^{(1)}_{p,1}}(x^{(1)}_{p,1}).
\end{equation}
First we know that for $pu_{p}^{p-1}$
and $D_{p,1}$ the  following similar estimates holds true:
$$pu_{p}^{p-1}(x)=O\left(\sum^k_{j=1}\frac{1}{\e_{p,j}^2\Big(1+\frac{|x-x_{p,j}|}{\e_{p,j}}\Big)^{4-\delta}} \right)~\mbox{and}~ D_{p,1}(x) =O\left( \frac{1}{\big(\e^{(1)}_{p,1}\big)^2\Big(1+\frac{|x-x^{(1)}_{p,1}|}{\e^{(1)}_{p,1}}\Big)^{4-\delta}} \right).$$
Then using  some basic inequalities, we can  prove \eqref{ttstsa},
 which is similar to Step 2 in the proof of Theorem \ref{th1.1}.

 \vskip 0.2cm

\noindent \textbf{Step 3.} We now have $ y_p\in   B_{2p\varepsilon^{(1)}_{p,1}}(x^{(1)}_{p,1})\backslash  B_{R\varepsilon^{(1)}_{p,1}}(x^{(1)}_{p,1})$.
Let $t_p:=|y_p|,$ then $\frac{t_p}{\e^{(1)}_{p,1}}\geq R\gg 1$. By translation we may assume that $x^{(1)}_{p,1}=0$.  Taking
\begin{equation}\label{lltsa}
\widetilde{\eta}_p(y):=\eta_p(t_py),
\end{equation} we find, by \eqref{aaad}, that
\begin{equation*}
\begin{split}
-\Delta \widetilde{\eta}_p(y)=&t^2_p D_{p,1}(t_py)\widetilde{\eta}_p(y) \leq C \Big(\frac{t_p}{\e^{(1)}_{p,1}}\Big)^2\frac{1}{\Big(1+|\frac{t_py}{\e^{(1)}_{p,1}}|\Big)^{4-\delta}}\rightarrow 0,~~\forall y\in \Omega_{t_p}:=\big\{x,t_px\in\Omega\big\}.
\end{split}\end{equation*}
Then there exists a bounded function $\eta$ such that
$$\widetilde{\eta}_p \rightarrow \eta~\mbox{in}~C\big(B_R(0)\backslash B_\delta(0)\big)
~\mbox{and}~\Delta \eta =0~~\mbox{in}~\R^2,$$
which implies $\eta= \widehat{C}$. From   $\widetilde{\eta}_p(\frac{y_p}{t_p})=\eta_p(y_p)=1$, we find $\widehat{C}=1$, namely
\begin{equation}\label{alttsa}
\widetilde{\eta}_p \rightarrow 1~\mbox{in}~C\big(B_R(0)\backslash B_\delta(0)\big).
\end{equation}
Now let $\eta_{p,1}(y):=\eta_p\big(\e^{(1)}_{p,1}y+x^{(1)}_{p,1}\big)$, then using \eqref{aaad}, $\eta_{p,1}$ solves $$-\Delta \eta_{p,1} =\big(\e^{(1)}_{p,1}\big)^2D_{p,1}\big(\e^{(1)}_{p,1}+x^{(1)}_{p,1} y\big)\eta_{p,1},$$
so by \eqref{asd0} in Lemma \ref{ttst} we have
\begin{equation}\label{jla}
-\Delta \eta_{p,1}- e^U\eta_{p,1}= f,
\end{equation}
 where \[ f(y):= \left[\big(\e^{(1)}_{p,1}\big)^2D_{p,1}\big(\e^{(1)}_{p,1}+x^{(1)}_{p,1} y\big)-e^{U(y)}\right]\eta_{p,1} (y)\overset{\eqref{asd0}}{=}O\Big(\frac{e^{U(y)}}{p^{1-\delta}}\Big), \mbox{uniformly on compact sets}.\]
Then the average
$\eta^*_{p,1}(r):=\displaystyle\int^{2\pi}_0\eta_{p,1}(r,\theta)d\theta$, solves the ODE
\begin{equation*}
-\left(\eta^*_{p,1}\right)''-\frac{1}{r}\left(\eta^*_{p,1}\right)'- e^U \eta^*_{p,1}= f^*,\ \mbox{ with }\quad f^*(r):= \displaystyle\int^{2\pi}_0f (r,\theta)d\theta.
\end{equation*}
Then similarly as in Step 3 in the proof of Theorem \ref{th1.1}, by ODE's theory, we have
\begin{equation*}
\eta^*_{p,1}(r)=c_{0,p}u_0(r)+W(r),
\end{equation*}
with
\begin{equation}\label{bds}
c_{0,p}=o(1)~~\mbox{and}~~\big|W(r)\big|\leq \frac{C\log r}{p^{1-\delta}}.\end{equation}
Hence it follows
\begin{equation}\label{Stsad}
\big|\eta^*_{p,1}\big(\frac{t_p}{\e^{(1)}_{p,j}}\big)\big|=
C\frac{\log \frac{t_p}{\e^{(1)}_{p,j}}}{p^{1-\delta}}+o(1)\to 0.
\end{equation}
However by \eqref{alttsa}, we know
\begin{equation}\label{lltsb}\widetilde{\eta}_p(y)\to 1,~~\mbox{for any}~~|y|=1.\end{equation}
Then from \eqref{lltsa} and \eqref{lltsb}, we find
\begin{equation*}
\begin{split}
\eta^*_{p,1}\big(\frac{t_p}{\e^{(1)}_{p,j}}\big) =&
 \int^{2\pi}_0 \eta_{p,1}\big(\frac{t_p}{\e^{(1)}_{p,j}},\theta\big) d\theta =
 \int^{2\pi}_0 \eta_{p}\big(t_p,\theta\big) d\theta
 \\=&\int^{2\pi}_0\widetilde{\eta}_p\big(1,\theta\big)d\theta  \geq \widetilde{c}_0>0,~\mbox{for some constant}~\widetilde{c}_0,
\end{split}\end{equation*}
which is a contraction to \eqref{Stsad}. This completes the proof of  Theorem \ref{th1-1}.
\end{proof}
\begin{Rem}\label{js}
We point out that the improved expansion \eqref{dluo-1} of $u_{p}$ is crucial to prove Theorem \ref{th1-1}.
In fact if one uses the expansion  \eqref{luo-1} of $u_{p}$
 instead of \eqref{dluo-1}, then Proposition \ref{lem3-8} changes into
\begin{equation}\label{bjs5}
C_{p,j} \frac{\partial R(x_{p,j})}{\partial{x_i}}-2\displaystyle\sum^k_{m=1,m\neq j}
C_{p,m}
D_{x_i} G(x_{p,m},x_{p,j})=o\big( \frac{\varepsilon_{p}}{p} \big)
.
\end{equation}
Then using \eqref{bjs5}, we find that \eqref{aaaa} is replaced  by
\begin{equation*}
\Big|x^{(1)}_{p,1}-x^{(2)}_{p,1}\Big|=o\big( \widetilde{\varepsilon}_{p} \big),
\end{equation*}
which implies
\begin{equation}\label{asd0d}
\big(\varepsilon^{(1)}_{p,1}\big)^{2}D_{p,1}\big(\varepsilon^{(1)}_{p,1}x+x^{(1)}_{p,1}\big)  =e^{U(y)} +o\big( e^{U(y)} \big),
\end{equation}
rather than \eqref{asd0} in Lemma \ref{ttst}.
We claim that \eqref{asd0d} is not enough to
get a contradiction in Step 3 in the proof of Theorem \ref{th1-1}. Indeed, using \eqref{asd0d} instead of \eqref{asd0},  \eqref{jla} changes into
\begin{equation*}
-\Delta \eta_{p,1} - e^U\eta_{p,1}= f,\quad\mbox{ where }~~f(y)=o\big(e^{U(y)}\big),
\end{equation*}
and so \eqref{bds} becomes
\[c_{0,p}=o(1)~~\mbox{and}~~W(r)= o\big(\log r\big).\]
Hence  \eqref{Stsad} changes into
\begin{equation*}
\Big|\eta^*_{p,1}\big(\frac{t_p}{\e^{(1)}_{p,j}}\big)\Big|=
o\big(\log p\big)+o(1),
\end{equation*}
which may not tend to $0$.
\end{Rem}

 \begin{Rem}[Local uniqueness when $k\geq 2$] \label{rem5}  Let $u_p^{(1)}$ and $u_p^{(2)}$ be  two different positive solutions to \eqref{1.1} with
\begin{equation*}
\lim_{p\rightarrow +\infty} p \int_{\Omega}|\nabla u^{(l)}_{p}(x)|^2dx=8k\pi e,\,~\mbox{for}\,~l=1,2
~\mbox{and}~k\geq 2,
\end{equation*}
which concentrate at the same point $x_\infty=\big(x_{\infty,1},\cdots,x_{\infty,k}\big)$. By Proposition \ref{lem3-8}, we know  that \begin{equation}\label{aluo-gilt}
C^{(l)}_{p,j} \frac{\partial R(x^{(l)}_{p,j})}{\partial{x_i}}-2\displaystyle\sum^k_{m=1,m\neq j}
C^{(l)}_{p,m}
D_{x_i} G(x^{(l)}_{p,m},x^{(l)}_{p,j})
=O\Big(\frac{\widetilde{\e}_p}{p^{2-\delta}}\Big)~~\mbox{with some}~\delta\in(0,1).
\end{equation}
Anyway \eqref{aluo-gilt} is not enough in order to generalize properly Lemma \ref{ssag}. As a consequence, we are not able to deduce the analogous of Proposition  \ref{aprop3-2} and hence
the local uniqueness result for the general case $k\geq 2$.
\end{Rem}


\appendix
\renewcommand{\theequation}{A.\arabic{equation}}

\setcounter{equation}{0}

\section{Proofs of (\ref{abb1-1})--(\ref{aluo41}) involving Green's function}\label{s6}
\setcounter{equation}{0}

In this appendix, we give the proofs of \eqref{abb1-1}--\eqref{aluo41} involving Green's function. Part of these can be found in \cite{CGPY2019} and we give the details for completeness.

\begin{proof}[\underline{\textbf{Proof of \eqref{abb1-1}}}]
By the bilinearity of $P_{j}(u,v)$, we have
\begin{equation}\label{ddgil21}
\begin{split}
P_{j}\Big(G(x_{p,j},x),\partial_hG(x_{p,j},x)\Big)=&
 P_{j}\Big(S(x_{p,j},x),\partial_hS(x_{p,j},x)\Big) - P_{j}\Big(H(x_{p,j},x),\partial_h S(x_{p,j},x)\Big)
 \\&
+P_{j}\Big(G(x_{p,j},x),\partial_h H(x_{p,j},x)\Big).
 \end{split}
\end{equation}
For $P_{j}\Big(S(x_{p,j},x),\partial_hS(x_{p,j},x)\Big)$, the oddness of the integrand function yields
\begin{equation}\label{ddgil22}
 P_{j}\Big(S(x_{p,j},x),\partial_hS(x_{p,j},x)\Big)=0.
\end{equation}

Next, we calculate $P_{j}\Big(H(x_{p,j},x),\partial_h S(x_{p,j},x)\Big)$.
First, let  $y=x-x_{p,j}$, then we get
\begin{equation}\label{8-17-1}
\partial_{h} S\big(x_{p,j},x\big)=-\frac{y_h}{2\pi|y|^{2}}, ~
\big\langle D \partial_{h} S\big(x_{p,j},x\big),\nu \big\rangle =-\frac{y_h}{2\pi|y|^{3}}
~\mbox{and}~
D_{x_l}\partial_{h} S\big(x_{p,j},x\big) =\frac{\delta_{hl}}{2\pi|y|^{2}}-
\frac{y_hy_l}
{\pi|y|^{4}}.
\end{equation}
Then we know
\begin{equation*}
\begin{split}
\theta\int_{\partial B_\theta(x_{p,j})}& \big\langle \partial_h D S(x_{p,j},x),\nu\big\rangle \big\langle
 D H(x_{p,j},x),\nu \big\rangle \\
=&-\theta D_{h} H\big(x_{p,j},x_{p,j}\big) \int_{\partial B_\theta(0)}\frac{y_h^2}{2\pi|y|^{4}} +O\big(\theta\big)=-\frac{1}{2}D_h H\big(x_{p,j},x_{p,j}\big)+O\big(\theta\big),
\end{split}
\end{equation*}
and
\begin{equation*}
\begin{split}
\theta\int_{\partial B_\theta(x_{p,j})}& \big\langle \partial_h  D S(x_{p,j},x),
D H(x_{p,j},x) \big\rangle \\
=& \theta D_{h}H\big(x_{p,j},x_{p,j}\big)\int_{\partial B_\theta(0)}
\big(\frac{1}{2\pi|y|^{2}}-
\frac{ y^2_h}
{\pi |y|^{4}}\big)
+O\big(\theta\big)=O\big(\theta\big),
\end{split}
\end{equation*}
which together imply
\begin{equation}\label{ddll-2}
P_{j}\Big(H\big(x_{p,j},x\big),\partial_h S\big(x_{p,j},x\big)\Big)
=-D_hH\big(x_{p,j},x_{p,j}\big)+O\big(\theta\big).
\end{equation}
Now we calculate $P_{j}\Big(G(x_{p,j},x),\partial_h H(x_{p,j},x)\Big)$.  Since $\partial_hD_{\nu} H(x_{p,j},x)$   is bounded in $B_d(x_{p,j})$,
then  we get
\begin{equation}\label{ddgil27}
P_{j}\Big(G(x_{p,j},x),\partial_h H(x_{p,j},x)\Big)=
O\Big(\theta\int_{\partial B_\theta(x_{p,j})} \big| D G(x_{p,j},x)\big|\Big)=O\big(\theta \big).
\end{equation}
Letting $\theta\rightarrow 0$, from \eqref{ddgil21}, \eqref{ddgil22},
\eqref{ddll-2} and \eqref{ddgil27}, we get
\begin{equation*}
\begin{split}
P_{j}\Big(G(x_{p,j},x),\partial_h G(x_{p,j},x)\Big)=D_h H\big(x_{p,j},x_{p,j}\big)=\frac{1}{2}\frac{\partial R(x_{p,j})}{\partial h}.
 \end{split}
\end{equation*}
Now we calculate  the term $
P_{j}\Big(G(x_{p,m},x),\partial_hG(x_{p,j},x)\Big)$. Similar to the estimate of \eqref{ddll-2}, we find
\begin{equation*}
P_{j}\Big(G(x_{p,m},x),\partial_hG(x_{p,j},x)\Big)=-D_hG \big(x_{p,m},x_{p,j}\big)\,\,~\mbox{for}~m\neq j.
\end{equation*}
Next, for $m\neq j$, since $\partial_hD_{\nu} G(x_{p,m},x)$   is bounded in $B_d(x_{p,j})$, we know
\begin{equation}\label{ddgil29}
\begin{split}
P_{j}\Big(G(x_{p,s},x),\partial_hG(x_{p,m},x)\Big)=&
O\Big(\theta\int_{\partial B_\theta(x_{p,j})} \big| D G(x_{p,s},x)\big|\Big)=O\big(\theta \big)\,\,~\mbox{for}~s=1,\cdots,k.
 \end{split}
\end{equation}
So letting $\theta\rightarrow 0$ in \eqref{ddgil29}, we know
\begin{equation*}
P_{j}\Big(G(x_{p,s},x),\partial_hG(x_{p,m},x)\Big)=0\,\,~\mbox{for}~m\neq j~\mbox{and}~s=1,\cdots,k.
\end{equation*}
\end{proof}
\begin{proof} [\underline{\textbf{Proof of \eqref{aluo1}}}]
First, by the bilinearity of $Q_{j}(u,v)$, we have
\begin{equation}\label{gil123}
\begin{split}
Q_{j} \Big(G(x_{p,j},x),G(x_{p,j},x)\Big)=&
 Q_{j}\Big(S(x_{p,j},x),S(x_{p,j},x)\Big)
 -2Q_{j}\Big(S(x_{p,j},x),H(x_{p,j},x)\Big)\\&
 +Q_{j}\Big(H(x_{p,j},x),H(x_{p,j},x)\Big).
 \end{split}
\end{equation}
Then for $Q_{j}\Big(S(x_{p,j},x),S(x_{p,j},x)\Big)$, the oddness of integrand functions  means
\begin{equation}\label{gil-7}
Q_{j}\Big(S(x_{p,j},x),S(x_{p,j},x)\Big)=0.
\end{equation}
Now we calculate the term $Q_{j}\Big(S(x_{p,j},x),H(x_{p,j},x)\Big)$. First, we know
\begin{equation}\label{gil-4}
\begin{split}
\int_{\partial B_\theta(x_{p,j})}& D_{\nu} S(x_{p,j},x)
D_{x_i}H(x_{p,j},x) \\
=&D_{x_i} H \big(x_{p,j},x_{p,j} \big) \int_{\partial B_\theta(x_{p,j})}D_{\nu} S(x_{p,j},x) +O \big(\theta \big)\int_{\partial B_\theta(x_{p,j})} \big|D_{\nu} S(x_{p,j},x) \big|\\
=&-D_{x_i} H \big(x_{p,j},x_{p,j} \big)+O \big(\theta \big),
\end{split}
\end{equation}
\begin{equation}\label{gil-5}
\begin{split}
\int_{\partial B_\theta(x_{p,j})}&D_{\nu} H(x_{p,j},x)
D_{x_i} S(x_{p,j},x) \\
=&\sum^2_{l=1} D_{x_l} H \big(x_{p,j},x_{p,j} \big)\int_{\partial B_\theta(x_{p,j})}D_{x_i} S(x_{p,j},x) \nu_l+O \big(\theta \big)\int_{\partial B_\theta(x_{p,j})} \big|D_{x_i} S(x_{p,j},x) \big|\\
=&D_{x_i} H \big(x_{p,j},x_{p,j} \big) \int_{\partial B_\theta(x_{p,j})}D_{x_i} S(x_{p,j},x) \nu_i+O \big(\theta \big),
\end{split}
\end{equation}
and
\begin{equation}\label{gil-6}
\begin{split}
\int_{\partial B_\theta(x_{p,j})}& \big\langle D S(x_{p,j},x),D H(x_{p,j},x) \big\rangle \nu_i\\
=&\sum^2_{l=1} D_{x_l} H \big(x_{p,j},x_{p,j} \big) \int_{\partial B_\theta(x_{p,j})}D_{x_l} S(x_{p,j},x) \nu_i+O \big(\theta \big)\int_{\partial B_\theta(x_{p,j})} \big|D S(x_{p,j},x) \big| \\
=&D_{x_i} H \big(x_{p,j},x_{p,j} \big) \int_{\partial B_\theta(x_{p,j})}D_{x_i} S(x_{p,j},x) \nu_i+O \big(\theta \big),
\end{split}
\end{equation}
which together imply
\begin{equation}\label{gil-8}
Q_{j}\Big(S(x_{p,j},x),H(x_{p,j},x)\Big)=D_{x_i} H \big(x_{p,j},x_{p,j} \big)+O \big(\theta \big).
\end{equation}
Also since $D_{\nu} H(x_{p,j},x)$ is bounded in $B_d(x_{p,j})$, it holds
\begin{equation}\label{gil-9}
Q_{j}\Big(H(x_{p,j},x),H(x_{p,j},x)\Big)=O \big(\theta\big).
\end{equation}
Letting $\theta\rightarrow 0$, from \eqref{gil123}, \eqref{gil-7}, \eqref{gil-8} and \eqref{gil-9},  we get
\begin{equation*}
\begin{split}
Q_{j}\Big(G(x_{p,j},x),G(x_{p,j},x)\Big)=D_{x_i} H \big(x_{p,j},x_{p,j} \big).
 \end{split}
\end{equation*}
Next, for $m\neq j$,
\begin{equation}\label{gil1234}
\begin{split}
Q_{j}\Big(G(x_{p,j},x),G(x_{p,m},x)\Big)=&
 Q_{j}\Big(S(x_{p,j},x),G(x_{p,m},x)\Big)
 - Q_{j}\Big(H(x_{p,j},x),G(x_{p,m},x)\Big).
 \end{split}
\end{equation}
Similar to  \eqref{gil-4}, \eqref{gil-5} and \eqref{gil-6}, we know
\begin{equation*}
\begin{split}
\int_{\partial B_\theta(x_{p,j})}&D_{\nu} S(x_{p,j},x)
D_{x_i} G(x_{p,m},x) \\
=&D_{x_i} G\big(x_{p,m},x_{p,j}\big) \int_{\partial B_\theta(x_{p,j})}D_{\nu} S(x_{p,j},x) +O\big(\theta\big)\int_{\partial B_\theta(x_{p,j})}|D_{\nu} S(x_{p,j},x) |\\
=&-D_{x_i} G\big(x_{p,m},x_{p,j}\big) +O\big(\theta\big),
\end{split}
\end{equation*}
\begin{equation*}
\begin{split}
\int_{\partial B_\theta(x_{p,j})}&D_{\nu}G(x_{p,m},x)
D_{x_i}S(x_{p,j},x) \\
=&\sum^2_{l=1} D_{x_l} G\big(x_{p,m},x_{p,j}\big) \int_{\partial B_\theta(x_{p,j})}D_{x_i} S(x_{p,j},x) \nu_l+O\big(\theta\big)\int_{\partial B_\theta(x_{p,j})}\big|D_{x_i} S(x_{p,j},x) \big| \\
=&D_{x_i}G\big(x_{p,m},x_{p,j}\big) \int_{\partial B_\theta(x_{p,j})}D_{x_i} S(x_{p,j},x) \nu_i+O\big(\theta\big),
\end{split}
\end{equation*}
and
\begin{equation*}
\begin{split}
\int_{\partial B_\theta(x_{p,j})}&\big\langle \nabla S(x_{p,j},x),\nabla G(x_{p,m},x) \big\rangle \nu_i\\
=&\sum^2_{l=1}  D_{x_l} G\big(x_{p,m},x_{p,j}\big) \int_{\partial B_\theta(x_{p,j})}D_{x_l} S(x_{p,j},x) \nu_i+O\Big(\theta \int_{\partial B_\theta(x_{p,j})}\big|D_{\nu} S(x_{p,j},x) \big|\Big)\\
=& D_{x_i} G\big(x_{p,m},x_{p,j}\big) \int_{\partial B_\theta(x_{p,j})}D_{x_i} S(x_{p,j},x) +O\big(\theta\big),
\end{split}
\end{equation*}
which together imply
\begin{equation}\label{gil-13}
Q_{j}\Big(S(x_{p,j},x),G(x_{p,m},x)\Big)= D_{x_i} G\big(x_{p,m},x_{p,j}\big)+O\big(\theta\big).
\end{equation}
Since $D_{\nu} H(x_{p,j},x)$ and $D_{\nu}G(x_{p,m},x)$ are bounded in $B_d(x_{p,j})$, it holds
\begin{equation}\label{gil-14}
Q_{j}\Big(H(x_{p,j},x),G(x_{p,m},x)\Big)=O\big(\theta\big).
\end{equation}
Letting $\theta\rightarrow 0$, from \eqref{gil1234}, \eqref{gil-13} and \eqref{gil-14}, we know
\begin{equation}\label{la1}
Q_{j}\Big(G(x_{p,j},x),G(x_{p,m},x)\Big)= D_{x_i} G\big(x_{p,m},x_{p,j}\big)\,\,~\mbox{for}~m\neq j.
\end{equation}
By the symmetry of $Q_{j}(u,v)$, \eqref{la1} imply
$$
Q_{j}\Big(G(x_{p,m},x),G(x_{p,j},x)\Big)= D_{x_i} G\big(x_{p,m},x_{p,j}\big)\,\, ~\mbox{for}~m\neq j.$$

Finally, since $D_{\nu}G(x_{p,s},x)$ and $D_{\nu}G(x_{p,m},x)$ are bounded in $B_d(x_{p,j})$ for $s,m\neq j$, it holds that  $
Q_{j}\Big(G(x_{p,s},x),G(x_{p,m},x)\Big)=O\big(\theta\big)$.
So letting $\theta\rightarrow 0$, we know
\begin{equation*}
Q_{j}\Big(G(x_{p,s},x),G(x_{p,m},x)\Big)=0\,\,~\mbox{for}~s,m\neq j.
\end{equation*}
\end{proof}
\begin{proof} [\underline{\textbf{Proof of \eqref{aluo41}}}]
By the bilinearity of $Q_{j}(u,v)$, we have

\begin{small}
\begin{equation}\label{gil71}
\begin{split}
Q_{j}\Big(G(x_{p,j},x),\partial _hG(x_{p,j},x)\Big)=&
Q_{j}\Big(S(x_{p,j},x),\partial _h S(x_{p,j},x)\Big)
  - Q_{j}\Big(H(x_{p,j},x),\partial _hS
  (x_{p,j},x)\Big)  \\&- Q_{j}\Big(S(x_{p,j},x),\partial _hH(x_{p,j},x)\Big)+Q_{j}\Big(H(x_{p,j},x),\partial_h H(x_{p,j},x)\Big).
 \end{split}
\end{equation}
\end{small}
Also by direct computations, we find
\begin{equation*}
Q_{j}\Big(S(x_{p,j},x),\partial _h S(x_{p,j},x)\Big)=
\frac{1}{2}\partial_h\Big(Q_{j}\big(S(x_{p,j},x),S(x_{p,j},x)\big)\Big)=0.
\end{equation*}
Now we calculate the term $Q_{j}\Big(S(x_{p,j},x),\partial_h H(x_{p,j},x)\Big)$. First, we know
\begin{equation*}
\begin{split}
\int_{\partial B_\theta(x_{p,j})}& D_{\nu} S\big(x_{p,j},x\big)
D_{x_i}\partial _h H\big(x_{p,j},x\big) \\
=&D_{x_i}\partial _h H\big(x_{p,j},x_{p,j}\big) \int_{\partial B_\theta(x_{p,j})}D_{\nu} S\big(x_{p,j},x\big) +O\big(\theta\big)\int_{\partial B_\theta(x_{p,j})}\big|D_{\nu} S(x_{p,j},x)\big|\\
=&-D_{x_i}\partial _h H\big(x_{p,j},x_{p,j}\big)+O\big(\theta\big),
\end{split}
\end{equation*}
\begin{equation*}
\begin{split}
\int_{\partial B_\theta(x_{p,j})}&D_{\nu} \partial _hH\big(x_{p,j},x\big)
D_{x_i} S\big(x_{p,j},x\big) \\
=&\sum^2_{l=1} D_{x_l} \partial _hH\big(x_{p,j},x_{p,j}\big)\int_{\partial B_\theta(x_{p,j})}D_{x_i} S\big(x_{p,j},x\big) \nu_l+O\big(\theta\big)\int_{\partial B_\theta(x_{p,j})}\big|D S(x_{p,j},x)\big|\\
=&D_{x_i} \partial _hH\big(x_{p,j},x_{p,j}\big) \int_{\partial B_\theta(x_{p,j})}D_{x_i} S\big(x_{p,j},x\big) \nu_i+O\big(\theta\big),
\end{split}
\end{equation*}
and
\begin{equation*}
\begin{split}
\int_{\partial B_\theta(x_{p,j})}&\big\langle D S(x_{p,j},x),D \partial _hH\big(x_{p,j},x\big) \big\rangle \nu_i\\
=&\sum^2_{l=1}D_{x_l} \partial _h H\big(x_{p,j},x_{p,j}\big) \int_{\partial B_\theta(x_{p,j})}D_{x_l} S\big(x_{p,j},x\big) \nu_i+O\big(\theta\big)\int_{\partial B_\theta(x_{p,j})}\big|D S\big(x_{p,j},x\big)\big|\\
=&D_{x_i}\partial _h H\big(x_{p,j},x_{p,j}\big) \int_{\partial B_\theta(x_{p,j})}D_{x_i} S\big(x_{p,j},x\big) \nu_i+O\big(\theta\big).
\end{split}
\end{equation*}
These together imply
\begin{equation}\label{ll-2}
Q_{j}\Big(S\big(x_{p,j},x\big),\partial_hH\big(x_{p,j},x\big)\Big)
=D_{x_i}\partial _h H\big(x_{p,j},x_{p,j}\big)+O\big(\theta\big).
\end{equation}
Next  we calculate the term $Q_{j}\Big(\partial_h S(x_{p,j},x), H(x_{p,j},x)\Big)$. Using \eqref{8-17-1}, we find
\begin{equation}\label{gil76}
\begin{split}
\int_{\partial B_\theta\big(x_{p,j}\big)}& D_{\nu} \partial _h  S\big(x_{p,j},x\big)
D_{x_i}H\big(x_{p,j},x\big) \\
=&\int_{\partial B_\theta(x_{p,j})}D_{\nu} \partial _h  S\big(x_{p,j},x\big)
\big\langle D D_{x_i}H\big(x_{p,j},x_{p,j}\big), x-x_{p,j}\big\rangle+O\big(\theta\big)\\=&-\frac{1}{2\pi}
\sum^2_{s=1} D^2_{x_ix_s}H\big(x_{p,j},x_{p,j}\big)
\int_{|y|=\theta} \frac{y_hy_s}{|y|^{3}}+O\big(\theta\big)  \\
=&-\frac{1}{2}D^2_{x_ix_h} H\big(x_{p,j},x_{p,j}\big)+O\big(\theta\big),
\end{split}
\end{equation}
\begin{equation}\label{gil77}
\begin{split}
\int_{\partial B_\theta(x_{p,j})}&D_{\nu} H\big(x_{p,j},x\big)
D_{x_i}\partial _h S\big(x_{p,j},x\big) \\
=&\int_{\partial B_\theta(x_{p,j})}D_{x_i} \partial _h  S\big(x_{p,j},x\big)
\big\langle D^2H\big(x_{p,j},x_{p,j}\big)\big(x-x_{p,j}\big),
\nu\big\rangle+O\big(\theta\big)\\
=&\frac{1}{2\pi}\sum^2_{s=1}\sum^2_{t=1} D^2_{x_tx_s}H\big(x_{p,j},x_{p,j}\big)\int_{|y|=\theta} \frac{y_ty_s}{|y|}
\big(\frac{\delta_{hi}}{|y|^{2}}-
\frac{2 y_hy_i}{|y|^{4}}\big)+O\big(\theta\big)\\
=&\begin{cases}
-D^2_{x_ix_h} H\big(x_{p,j},x_{p,j}\big)+O\big(\theta\big) ~& \mbox{for}~i\neq h,\\[1mm]
O\big(\theta\big) ~& \mbox{for}~i= h,
\end{cases}
\end{split}
\end{equation}
and
\begin{equation}\label{gil78}
\begin{split}
\int_{\partial B_\theta(x_{p,j})}&\big\langle D\partial _h S\big(x_{p,j},x\big),D H\big(x_{p,j},x\big) \big\rangle \nu_i\\
=&\int_{\partial B_\theta(x_{p,j})}
\langle D^2H\big(x_{p,j},x_{p,j}\big)\big(x-x_{p,j}\big),D \partial _h  S\big(x_{p,j},x\big)\big\rangle \nu_i+O\big(\theta\big)\\
=&\frac{1}{2\pi}\sum^2_{s=1}\sum^2_{t=1} D^2_{x_tx_s}H\big(x_{p,j},x_{p,j}\big)\int_{|y|=\theta}  y_t
\big(\frac{\delta_{hs}}{|y|^{2}}-
\frac{2y_hy_s}{|y|^{4}}\big)\frac{y_i}{|y|}+O\big(\theta\big)\\
=&\begin{cases}
-\frac{1}{2}D^2_{x_ix_h} H\big(x_{p,j},x_{p,j}\big)+O\big(\theta\big) ~& \mbox{for}~i\neq h,\\[1mm]
\frac{1}{2}D^2_{x_ix_h} H\big(x_{p,j},x_{p,j}\big)+O\big(\theta\big) ~& \mbox{for}~i= h.
\end{cases}
\end{split}
\end{equation}
From \eqref{gil76}, \eqref{gil77} and \eqref{gil78}, we get
\begin{equation}\label{ll-1}
Q_{j}\Big(S\big(x_{p,j},x\big),\partial_hH\big(x_{p,j},x\big)\Big)= D^2_{x_ix_h} H\big(x_{p,j},x_{p,j}\big)+O\big(\theta\big).
\end{equation}
Here the last two equalities hold by the fact $\Delta G\big(x,x_{p,j}\big)=0$ for $x\in \Omega\backslash B_{\theta}(x_{p,j})$.
Also since $H\big(x_{p,j},x\big)$ and $D_{\nu} H\big(x_{p,j},x\big)$ are bounded in $B_d\big(x_{p,j}\big)$, it holds that
\begin{equation}\label{gil79}
Q_{j}\Big(H\big(x_{p,j},x\big),\partial _h H\big(x_{p,j},x\big)\Big)=O\big(\theta\big).
\end{equation}
Letting $\theta\rightarrow 0$, from \eqref{gil71},\eqref{ll-2}, \eqref{ll-1} and \eqref{gil79}, we get
\begin{equation*}
\begin{split}
Q_{j}\Big(G\big(x_{p,j},x\big), \partial _hG\big(x_{p,j},x\big)\Big)=-D_{x_i}\partial _h H\big(x_{p,j},x_{p,j}\big)-D^2_{x_ix_h} H\big(x_{p,j},x_{p,j}\big)=-\frac{\partial^2 R(x_{p,j})}{\partial x_i\partial x_h}.
 \end{split}
\end{equation*}
Next, similar to the estimates of  \eqref{ll-2}, for $m\neq j$, we know
\begin{equation}\label{gil81}
\begin{split}
Q_{j} \Big(G\big(x_{p,j},x\big),\partial _h G\big(x_{p,m},x\big)\Big) =&
 Q_{j}\Big(S\big(x_{p,j},x\big),\partial _hG\big(x_{p,m},x\big)\Big)
 - Q_{j}\Big(H\big(x_{p,j},x\big),\partial _hG\big(x_{p,m},x\big)\Big)\\=&
D_{x_i}\partial _h G\big(x_{p,m},x_{p,j}\big)+O\big(\theta\big).
 \end{split}
\end{equation}
Letting $\theta\rightarrow 0$ in \eqref{gil81},  we obtain
\begin{equation*}
\begin{split}
Q_{j}\Big(G\big(x_{p,j},x\big), \partial _hG\big(x_{p,m},x\big)\Big)= D^2_{x_ix_h} H\big(x_{p,m},x_{p,j}\big)\,\,~\mbox{for}~m\neq j.
 \end{split}
\end{equation*}
Also similar to the estimates of \eqref{ll-1}, for $m\neq j$, we know
\begin{equation}\label{gil84}
\begin{split}
Q_{j} \Big(G\big(x_{p,m},x\big),\partial _h G\big(x_{p,j},x\big)\Big)=&
 Q_{j}\Big(S\big(x_{p,m},x\big),\partial _hG\big(x_{p,j},x\big)\Big)
 - Q_{j}\Big(H\big(x_{p,m},x\big),\partial _hG\big(x_{p,j},x\big)\Big)\\=&
 D^2_{x_ix_h} G\big(x_{p,m},x_{p,j}\big)+O\big(\theta\big). \end{split}
\end{equation}
Letting $\theta\rightarrow 0$ in \eqref{gil84},  we obtain
\begin{equation*}
\begin{split}
Q_{j}\Big(G\big(x_{p,m},x\big), \partial _hG\big(x_{p,j},x\big)\Big)= D_{x_i}\partial _h H\big(x_{p,m},x_{p,j}\big)\,\,~\mbox{for}~m\neq j.
 \end{split}
\end{equation*}
Finally,  since $G\big(x_{p,s},x\big)$ and $D_{\nu} G\big(x_{p,s},x\big)$ are bounded in $B_d\big(x_{p,j}\big)$ for $s\neq j$, it holds
\begin{equation*}
Q_{j}\Big(G\big(x_{p,m},x\big),\partial _h G\big(x_{p,s},x\big)\Big)=0\,\,~\mbox{for}~m,s\neq j.
\end{equation*}
\end{proof}

\section{Proof of Proposition \ref{blem3-1a}}\label{76}
\setcounter{equation}{0}
\renewcommand{\theequation}{B.\arabic{equation}}

\setcounter{equation}{0}

Let $M_p:=\displaystyle\max_{|x|\leq \frac{d_0}{\e_{p,j}}}\frac{|k_{p,j}|}{(1+|x|)^{\tau_1}}$,
then $\eqref{blpy1} \Leftrightarrow M_p \leq C$.
We will prove that $M_p \leq C$ by contradiction.  Set
\begin{equation*}
M_p^*:=\max_{|x|\leq \frac{d_0}{\e_{p,j}}}\max_{|x'|=|x|}\frac{|k_{p,j}(x)-k_{p,j}(x')|}{(1+|x|)^{{\tau_1}}}.
\end{equation*}

\begin{Lem}\label{hnb}
If $M_p\to +\infty$, then it holds
\begin{equation}\label{sstd}
M_p^*=o(1)M_p.
\end{equation}
\end{Lem}

\begin{proof}
Suppose this is not true. Then there exists $c_0>0$ such that $M^*_p\geq c_0M_p$. Let $x'_p$ and
$x''_p$ satisfy $|x'_p|=|x''_p| {\leq\frac{d_{0}}{\varepsilon_{p,j}}}$ and
\begin{equation*}
M_p^*=  \frac{|k_{p,j}(x'_p)-k_{p,j}(x''_p)|}{(1+|x'_p|)^{{\tau_1}}}.
\end{equation*}
Without loss of generality, we may assume that $x'_p$ and
$x''_p$ are symmetric with respect to the $x_1$ axis. Set
\begin{equation*}
l^*_p(x):=k_{p,j}(x)-k_{p,j}(x^-),~\,~x^-:=(x_1,-x_2),\  \mbox{ for }x=(x_{1},x_{2}),  ~\,x_2>0.
\end{equation*}
 Hence $x^{''}_{p}={x^{'}_{p}}^{-}$ and
$M_p^*=\frac{ l^*_p(x'_{p}) }{ (1+|x'_{p}| )^{\tau_1} }$.
Let us define
\begin{equation}\label{assfd}
l_p(x):=\frac{l^*_p(x)}{(1+x_2)^{\tau_1}}.
\end{equation}
Then direct calculations show that $l_p$ satisfies
\begin{equation}\label{fsd}
-\Delta l_p-\frac{2{\tau_1}}{1+x_2}\frac{\partial l_p}{\partial x_2}+
\frac{{\tau_1}(1-{\tau_1})}{(1+x_2)^2}l_p=\frac{h_p(x)}{(1+x_2)^{\tau_1}},
\end{equation}
where $h_p(x):=h^*_p(x)-h^*_p(x^-)$ and $h^*_p=O\Big(\frac{ \log(1+|x|)}{ 1+|x|^{4-\delta-2\tau}}\Big)$.
Also let $y^{**}_p$ satisfy
\begin{equation}\label{assf1d}
\mbox{$|y^{**}_p|\leq\frac{d_{0}}{\varepsilon_{p,j}}$, $y^{**}_{p,2}\geq 0$ and}~~~ |l_p(y^{**}_p)|= M^{**}_p:=\max_{|x|\leq \frac{d_0}{\e_{p,j}},~x_2\geq 0} |l_p(x)|.
\end{equation}
Then it follows
\begin{equation}\label{dsdd}
M^{**}_p\geq \frac{|k_{p,j}(x'_p)-k_{p,j}(x''_p)|}{(1+|x'_p|)^{\tau_1}}=M^*_p\to +\infty.
\end{equation}
We may assume that $l_p(x^{**}_p)>0$. We claim that
\begin{equation}\label{1hd}
|y^{**}_p|\leq C.
\end{equation}
The proof of  \eqref{1hd} is divided into   two steps.

\vskip 0.2cm

\noindent
\emph{Step 1. We prove that $|y^{**}_p|\leq \frac{d_0}{2\e_{p,j}}$.}

\vskip 0.2cm
Suppose this is not true. By the definition of  $k_{p,j}$ and \eqref{llls}, we have
\begin{equation*}
\begin{split}
k_{p,j}(y^{**}_p)-k_{p,j}(y^{**-}_p)
=& p\Big(\big(v_{p,j}(y^{**}_p)-v_{p,j}(y^{**-}_p)\big)-\big(w_0(y^{**}_p)-w_0(y^{**-}_p)\big)\Big)\\=&
 O\Big( p {\e_{p,j}|y^{**}_{p,2}|}\Big)+o\Big( p\sum^k_{l=1} \e_{p,l}\Big)-p\Big(w_0(y^{**}_p)-w_0(y^{**-}_p)\Big).
\end{split}\end{equation*}
Also we recall that
 \begin{equation*}
\begin{split}
w_0(x)=w(x)+ c_0\frac{\partial U(\frac{x}{\lambda})}{\partial \lambda}\Big|_{\lambda=1}+c_1  \frac{\partial U(x)}{\partial x_1}+c_2
\frac{\partial U(x)}{\partial x_2},
\end{split}\end{equation*}
where $w(x)$ is the radial solution of $-\Delta u-e^{U(x)}u=-\frac{U^2(x)}{2}e^{U(x)}$.
Hence, in view of $|U'(r)|\le \frac{C}{1+r}$, we obtain
\begin{equation*}
\begin{split}
 w_0(y^{**}_p)-w_0(y^{**-}_p) =&
 c_0\frac{\partial U(\frac{y^{**}_p}{\lambda})}{\partial \lambda}\Big|_{\lambda=1}+c_1  \frac{\partial U(y^{**}_p)}{\partial x_1}+c_2
\frac{\partial U(y^{**}_p)}{\partial x_2}\\&
-c_0\frac{\partial U(\frac{y^{**-}_p}{\lambda})}{\partial \lambda}\Big|_{\lambda=1}-c_1  \frac{\partial U(y^{**-}_p)}{\partial x_1}-c_2
\frac{\partial U(y^{**-}_p)}{\partial x_2}\\=&
O\Big(|y^{**}_p|^{-1}\Big).
\end{split}\end{equation*}
As a result,
\begin{equation*}
M^{**}_p\leq C p  \e_{p,j} |y^{**}_{p,2}|^{1-\tau}+o(1)+O\Big(p|y^{**}_p|^{-1}\Big)
\leq  C p  \e_{p,j}^{\tau}  +o(1)+O\Big(p \e_{p,j}\Big)=o(1).
\end{equation*}
This is a contradiction.

\vskip 0.2cm

\noindent
\emph{Step 2. We prove that $|y^{**}_p|\leq C$.}

\vskip 0.2cm
 Suppose this is not true.
Now by Step 1, we have $y^{**}_p\in B_{\frac{d_0}{2\e_{p,j}}}(0)$. Thus
\begin{equation}\label{stsd}|\nabla l_p(y^{**}_p)|=0~\mbox{and}~\Delta l_p(y^{**}_p)\leq 0.
\end{equation}
So from \eqref{fsd}  and \eqref{stsd}, we find
\begin{equation}\label{fs2d}
\begin{split}
0\leq &-\Delta l_p(y^{**}_p)=-
\frac{\tau_1(1-\tau_1)}{(1+y^{**}_{p,2})^2}l_p(y^{**}_p)
+\frac{h_p(y^{**}_p)}{(1+y^{**}_{p,2})^{\tau_1}}.
\end{split}\end{equation}
Hence using \eqref{fs2d}  and the fact that $|y^{**}_p|\to \infty$, we deduce
\begin{equation*}
\frac{l_p(y^{**}_p)}{(1+y^{**}_{p,2})^2}\leq \frac{Ch_p(y^{**}_p)}{(1+y^{**}_{p,2})^{\tau_1}}.
\end{equation*}
Then it follows
\begin{equation}\label{shd}
l_p(y^{**}_p) \leq Ch_p(y^{**}_p)(1+y^{**}_{p,2})^{2-\tau_1}.
\end{equation}
Combining $h^*_p=O\Big(\frac{ \log(1+|x|)}{ 1+|x|^{4-\delta-2\tau}}\Big)$ and \eqref{shd}, we obtain
\begin{equation*}
M^{**}_p\leq
\frac{C\big(\log (1+|y_p^{**}|)\big)^2}{(1+|y_p^{**}|)^{2+\tau_1-\delta-2\tau}}\to 0~~\mbox{as}~~p\to +\infty,
\end{equation*}
which is a contradiction with \eqref{dsdd}. Then \eqref{1hd} follows.\\

Now let $l_p^{**}(x)=\frac{l_p(x)}{M_p^{**}}$,
where $l_p$ is defined in \eqref{assfd} and $M_p^{**}$ is defined in \eqref{assf1d}.
From \eqref{fsd}, it follows that $l_p^{**}(x)$ solves
\begin{equation}\label{afsd}
-\Delta l^{**}_p-\frac{2\tau_1}{1+x_2}\frac{\partial l^{**}_p}{\partial x_2}+
\frac{\tau_1(1-\tau_1)}{(1+x_2)^2}l^{**}_p=\frac{h_p(x)}{M_p^{**}(1+x_2)^{\tau_1}}.
\end{equation}
Moreover $|l^{**}_{p}(x)|\leq 1$ and
$|l^{**}_{p}(y^{**}_p)|=1$, hence $l^{**}_p(x)\to \gamma(x)$ uniformly
in any compact subset of $\R^2$. And   $\gamma(x)\not\equiv0$ because $l_p^{**}(y_{p}^{**})=1$ and  $|y_{p}^{**}|\le C$. Observe that
$$\frac{h_p(x)}{M_p^{**}(1+x_2)^{\tau_1}}\to 0~\mbox{uniformly in any compact set of}~\big\{|x|\leq \frac{d_0}{\e_{p,j}}\big\}~\mbox{as}~p\to +\infty.$$
Hence passing to the limit $p\to \infty$ into \eqref{afsd},  we can deduce that   $\gamma$ solves
\begin{equation*}
\begin{cases}
-\Delta \gamma -\frac{2\tau}{1+x_2}\frac{\partial \gamma}{\partial x_2}
+\left(\frac{\tau_1(1-\tau_1)}{(1+x_2)^2}-e^{U(x)} \right)\gamma =0,~x_2>0,\\[2mm]
\omega(x_1,0)=0.
\end{cases}
\end{equation*}
Hence $\bar \gamma=(1+x_2)^{\tau}\gamma$ satisfies
\begin{equation*}
\begin{cases}
-\Delta \bar \gamma  -e^{U(x)} \bar \gamma =0,~x_2>0,\\[2mm]
\bar\gamma(x_1,0)=0.
\end{cases}
\end{equation*}
And then using Lemma \ref{llm}, we have
\begin{equation}\label{lssd}
\bar \gamma(x)=c_0\frac{\partial U(\frac{x}{\lambda})}{\partial \lambda}\Big|_{\lambda=1}+c_1  \frac{\partial U(x)}{\partial x_1}+c_2
\frac{\partial U(x)}{\partial x_2},~~
\mbox{for some constants $c_0$, $c_1$ and $c_2$}.
\end{equation}
On the other hand, by the definition of $k_{p,j}(x)$ and $l^*_p(x)$, we know
\begin{equation*}
\begin{split}
\nabla l_p^*(0)=&\nabla \big( k_{p,j}(x)-k_{p,j}(x^-)\big)\big|_{x=0}\\=&
p\big(\nabla v_{p,j}(x)-\nabla w_0(x)\big)\big|_{x=0}-p\big(\nabla v_{p,j}(x^-)-\nabla w_0(x^-)\big)\big|_{x=0} \\=& -p \nabla  \big( w_0(x)- w_0(x^-)\big)\big|_{x=0}.
\end{split}
\end{equation*}
Also we recall  that
\begin{equation*}
\begin{split}
w_0(x)=w(x)+e_0\frac{\partial U(\frac{x}{\lambda})}{\partial \lambda}\Big|_{\lambda=1}+e_1  \frac{\partial U(x)}{\partial x_1}+e_2
\frac{\partial U(x)}{\partial x_2}~~
\mbox{for some constants $e_0$, $e_1$ and $e_2$},
\end{split}
\end{equation*}
where $w(x)$ is a radial function.
Hence we can verify that
\begin{equation*}
\begin{split}
 \nabla \big(w_0(x)-w_0(x^-)\big)\big|_{x=0}=0.
\end{split}
\end{equation*}
And then it holds
\begin{equation*}
\nabla \big((1+x_2)^{\tau_1}l_p^{**}\big)  \big|_{x=0}=
\frac{\nabla l_p^*(0)}{M_p^{**}}= 0,
\end{equation*}
which means $\nabla \bar\gamma(0)=0$. Then from \eqref{lssd}, we find $e_1=e_2=0$.
Moreover from $\bar\gamma(x_1,0)=0$, we also have $e_0=0$. So $\bar\gamma=0$, which is a contradiction.
Hence \eqref{sstd} follows.
\end{proof}

  \vskip 0.1cm

\begin{proof}[\underline{\textbf{Proof of Proposition \ref{blem3-1a}}}] First we recall that
$\eqref{blpy1} \Leftrightarrow M_p \leq C$.
Suppose by contradiction that $M_p\to +\infty$, as $p\to +\infty$ and set
\begin{equation*}
 \varphi_{p,j}(r)=\frac{1}{2\pi}\int^{2\pi}_0k_{p,j}(r,\theta)d \theta,~\,\,r=|x|.
\end{equation*}
Then by Lemma \ref{hnb}, it follows
\begin{equation*}
\max_{r\leq \frac{d_0}{\e_{p,j}}}\frac{|\varphi_{p,j}(r)|}{(1+r)^{\tau_1}}=M_p\big(1+o(1)\big).
\end{equation*}
Assume that $\frac{|\varphi_{p,j}(r)|}{(1+r)^{\tau_1}}$ attains its maximum at $s_p$.
Then we claim
\begin{equation}\label{tty}
s_p\leq C.
\end{equation}
In fact, let $\phi(x)=\frac{\partial U(\frac{x}{\lambda})}{\partial \lambda}\big|_{\lambda=1}$ and  we recall that
$-\Delta\phi(x)=e^{U(x)}\phi(x)$.
Now multiplying \eqref{5-7-33}   by $\phi(x)$ and using integration by parts, we have
\begin{equation*}
\begin{split}
\int_{|x|=r}&\Big[\frac{\partial k_{p,j}(x)}{\partial \nu}\phi(x)-  \frac{\partial \phi(x)}{\partial \nu}k_{p,j}(x)\Big]d\sigma=
o(1)M_p+O(1).
\end{split}
\end{equation*}
Hence similar to \eqref{ss2}, \eqref{aas} and \eqref{aasa}, we have
$(1+s_p)^{\tau_1}M_p= O\big(M_p \big)$,
 which implies \eqref{tty}.
\vskip 0.2cm

Now integrating \eqref{b5-7-33} and using Lemma \ref{llma}, we get
\begin{equation}\label{equazionePsid}
-\Delta \varphi_{p,j}=\varphi_{p,j}e^{U(x)}+\frac{1}{2\pi}\displaystyle\int^{2\pi}_0h^*_p(r,\theta)d \theta,~\,~\mbox{for}~|x|\leq \frac{d_0}{\e_{p,j}}.
\end{equation}
Next we define $\varphi^*_{p,j}(x)=\frac{\varphi_{p,j}(|x|)}{\varphi_{p,j}(s_p)}$ and   pass to the limit in the equation \eqref{equazionePsid} divided by $\varphi_{p,j}(s_p)$. One has
\begin{equation*}
|\varphi^*_{p,j}(x)|\leq \frac{C\big|\varphi_{p,j}(|x|)\big|}{M_p}\leq C\big(1+|x|\big)^{\tau_1}.
\end{equation*}
Also we know
\begin{equation*}
\frac{1}{\varphi_{p,j}(s_p)} \displaystyle\int^{2\pi}_0h^*_p(r,\theta)d \theta \leq  \frac{C}{M_p}\to 0.
\end{equation*}
Hence from the above computations and dominated convergence theorem, we can deduce that $\varphi^*_{p,j}\to \varphi(|x|)$ in  $C^2_{loc}(\R^2)$ and $\varphi$ satisfies
\begin{equation*}
-\varphi''-\frac{1}{r}\varphi'=e^{U}\varphi.
\end{equation*}
Therefore, it follows
\begin{equation}\label{stsy}
\varphi=c_0\frac{8-|x|^2}{8+|x|^2},~\mbox{with some constant}~c_0>0.
\end{equation}
Since $\varphi^*_{p,j}{ (s_{p})}=1$ and $s_p\leq C$, we find $\varphi(|x|)\not\equiv 0$.

\vskip 0.1cm

On the other hand, we know that
$$\varphi_{p,j}(0) =k_{p,j}(0)=p\big(v_{p,j}(x)-w_0\big)=-pw_0~
~\mbox{and}~~w_0(0)=\displaystyle\lim_{p\to \infty} v_{p,j}(0)=0.$$ Hence $\varphi_{p,j}(0)=0$ and then  we have $\varphi(0)=0$, which, together with
\eqref{stsy}, implies $\varphi\equiv 0$. This is a contradiction. This completes the proof of \eqref{blpy1}.
\end{proof}

\vskip 0.5cm

\noindent\textbf{Acknowledgments} This work was done while Peng Luo was visiting the Mathematics Department of University of Rome ``La Sapienza"
whose members he would like to thank for their warm hospitality. Isabella Ianni was supported by PRIN 2017JPCAPN-003 project and by VALERE project.
 Peng Luo was supported by NNSF of China(No.11701204,11831009) and  the China Scholarship Council. Shusen Yan was supported
by NNSF of China(No.11629101).

\renewcommand\refname{References}
\renewenvironment{thebibliography}[1]{%
\section*{\refname}
\list{{\arabic{enumi}}}{\def\makelabel##1{\hss{##1}}\topsep=0mm
\parsep=0mm
\partopsep=0mm\itemsep=0mm
\labelsep=1ex\itemindent=0mm
\settowidth\labelwidth{\small[#1]}%
\leftmargin\labelwidth \advance\leftmargin\labelsep
\advance\leftmargin -\itemindent
\usecounter{enumi}}\small
\def\newblock{\ }
\sloppy\clubpenalty4000\widowpenalty4000
\sfcode`\.=1000\relax}{\endlist}
\bibliographystyle{model1b-num-names}

\end{document}